\documentclass[final,3p,times]{elsarticle}

\usepackage{lineno}
\modulolinenumbers[5]
\usepackage{amsmath}
\usepackage{amssymb}
\usepackage{bm}
\usepackage{booktabs}
\usepackage{multirow}
\usepackage{graphicx}
\usepackage{subfig}
\usepackage{algorithm}
\usepackage{algpseudocode}
\usepackage{tikz}
\usepackage{xcolor}
\usepackage{pgfplots}
\pgfplotsset{compat=1.9}

\usepackage{amsthm}
	\theoremstyle{remark}
	\newtheorem{rem}{Remark}[section]

	\theoremstyle{plain}
	\newtheorem{thm}{Theorem}[section]
	\newtheorem{prop}[thm]{Proposition}

\usepackage{hyperref}
\journal{arXiv}

\begin{document}

\begin{frontmatter}

\title{A novel block non-symmetric preconditioner for mixed-hybrid finite-element-based flow simulations}


\author[HBKU]{Stefano Nardean}
\ead{snardean@hbku.edu.qa}
\author[UNIPD]{Massimiliano Ferronato}
\ead{massimiliano.ferronato@unipd.it}
\author[HBKU]{Ahmad S. Abushaikha}
\ead{aabushaikha@hbku.edu.qa}

\address[HBKU]{Division of Sustainable Development, College of Science and Engineering, Hamad Bin Khalifa University, Education City, Qatar Foundation, Doha, Qatar}
\address[UNIPD]{Department of Civil, Environmental and Architectural Engineering, University of Padova, Padova, Italy}




\begin{abstract}
In this work we propose a novel block preconditioner, labelled Explicit Decoupling Factor Approximation (EDFA), to accelerate the convergence of Krylov subspace solvers used to address the sequence of non-symmetric systems of linear equations originating from flow simulations in porous media. The flow model is discretized blending the Mixed Hybrid Finite Element (MHFE) method for Darcy's equation with the Finite Volume (FV) scheme for the mass conservation. The EDFA preconditioner is characterized by two features: the exploitation of the system matrix decoupling factors to recast the Schur complement and their inexact fully-parallel computation by means of restriction operators. We introduce two adaptive techniques aimed at building the restriction operators according to the properties of the system at hand. The proposed block preconditioner has been tested through an extensive experimentation on both synthetic and real-case applications, pointing out its robustness and computational efficiency.
\end{abstract}


\end{frontmatter}


\section{Introduction}
Numerical modelling of fluid flow in porous media is a key requirement for a wide number of applications in subsurface hydrology and petroleum engineering. In general, computer simulators are fundamental tools for the proper management and exploitation of aquifer systems, as well as oil and gas fields. The growing demand for a higher accuracy of the simulation, assisted by the increasing availability of computational and storage resources, leads to a continuous development and refinement of virtual simulators. The degree of approximation of the overall numerical model is defined, first of all, by the underlying mathematical model, but also by the selected discretization scheme. Discretization schemes should handle effectively non $\mathbb{K}$-orthogonal unstructured grids, as well as highly heterogeneous and anisotropic rock/fluid properties, frequently introduced as full-tensors in the model (see, for instance, \cite{Abd2020} about the numerical issues, related to abrupt changes in permeability, in node control volume finite element discretizations). 

The Mixed Hybrid Finite Element (MHFE) method, and in general the whole class of Mixed Finite Element (MFEM) methods~\cite{Brezzi1991}, coupled with the Finite Volume (FV) method, has been gaining a growing popularity in recent years. 
The enforcement of the mass balance at the elemental level, the continuity of normal fluxes across internal faces of the discretized domain, the accuracy of the velocity and pressure fields, the possibility of handling either structured or unstructured grids and the elegant treatment of full tensor fluid properties~\cite{Younes2010,Mose1994,Chavent1991} have made the MHFE method attractive for several applications, such as contaminant transport~\cite{Radu2008,Yoisell2012,Brunner2014}, energy storage~\cite{Smejkal2020}, poromechanics~\cite{Niu2020,Frigo2020} and, of course, single-phase~\cite{Younes2008}, variably saturated~\cite{Bause2004,Belfort2009}, multi-phase~\cite{Fucik2011,Abushaikha2015,Fucik2019,Hou2016,Moortgat2016} and more recently compositional \cite{Abushaikha2017} flow problems. However, the MHFE method exhibits some critical aspects as well, for instance the violation of the Discrete maximum principle~\cite{Hoteit2002} and the higher number of unknowns per element, as compared to other schemes like the FV or the classical Finite element methods. In this regard, much effort was devoted to try to reduce the overall number of unknowns per cell to one per face~\cite{Younes2010,Chavent1991} and even only one per element~\cite{Younes1999}. The main focus was on triangular cells in two-dimensional (2-D) applications, however with some limitations in the shape of tetrahedra in three-dimensional (3-D) domains~\cite{Younes2004}.

Recently, a MHFE-based simulator was developed in~\cite{Hou2016} to model the two-phase flow in heterogeneous porous media. Fu{\v{c}}{\'{i}}k et al.~\cite{Fucik2019} introduced multi-component compositional flow in the previous model and focused also on the design of a parallel implementation on both CPU and GPU. A similar approach, but extended to compressible multi-phase flow, was developed in~\cite{Moortgat2016} and applied to several real-field applications.
Puscas et al.~\cite{Puscas2018} proposed a two-phase flow Multiscale MHFE (MMHFE) simulator, where much care was devoted to the design of a robust parallel implementation, while Devloo et al.~\cite{Devloo2019}, instead, introduced the Discrete Fracture Model in a 2D MMHFE model.
Abushaikha et al.~\cite{Abushaikha2017} introduced a fully implicit general-purpose MHFE-based simulator for highly heterogeneous reservoirs.

The modelling approach in~\cite{Abushaikha2017}, considered in this work as well, is characterized by three challenging properties: (i) the high number of unknowns per element (7 for hexahedra in the lowest-order Raviart-Thomas ($\mathbb{RT}_0$) space~\cite{Raviart1977}), giving rise to large-size systems of equations, (ii) the non-symmetric nature of such systems, and (iii) the inherent block structure of the discrete linearized problem, which can be exploited for the design of specific solvers. In fact, the computational efficiency of the linear solver is a key issue in a virtual simulator, since most of the overall CPU time spent in a full-transient flow simulation is allocated to address the sequence of large-size, usually ill-conditioned, linear systems of equations~\cite{Wang2018}. Given the size and sparsity degree of these systems, Krylov subspace methods~\cite{Saad2003} are usually the method of choice, but their performance needs to be boosted by means of appropriate preconditioning operators.

The main objective of this paper is the efficient solution of the systems of equations, stemming from the aforementioned MHFE-FV modelling approach, by designing a novel preconditioning technique that copes with their non-symmetric nature. Preconditioning large-size block systems of equations is still an open issue for several numerical applications. In recent years, block preconditioning techniques have been developed for different problems, such as the solution of the Navier-Stokes equations~\cite{Yang2015,Farrell2019,Bootland2019}, applied also to hemodynamic simulations~\cite{Wu2014,Liu2020}, coupled physical processes like flow and poromechanics~\cite{Castelletto2016,White2016} or Stokes-Darcy models~\cite{Chidyagwai2016}, fissure/fault mechanics~\cite{Franceschini2019} and multi-phase flow in porous media~\cite{Liu2016,Cusini2015,Roy2020}. For instance, a popular physics-based preconditioner for reservoir simulations is the Constraint Pressure Residual (CPR)~\cite{Wallis1983,Wallis1985,Cao2005}, which is the standard for commercial simulators. CPR was designed with the aim at exploiting the block structure of the system matrix and the different properties of those submatrices, resulting from the description of different kind of processes. The CPR-type algorithms are multi-level preconditioners (like SIMPLE~\cite{Patankar1980,Elman2008}), whose application (to a vector) goes through several stages during which the groups of unknowns are repeatedly updated. In its original formulation, CPR has two stages but other multi-level variants exist~\cite{Wang2018,Liu2016}. Recently, a two-stage CPR scheme, suitable for non-isothermal multi-phase flow simulations, namely Constrained Pressure-Temperature Residual (CPTR), has been designed by Roy et al.~\cite{Roy2020}. However, given the ill-conditioning and non-symmetric nature of the systems of equations originating from our modelling approach, CPR-like schemes are usually ineffectual. 

The issue of preconditioning in the framework of the MFE discretization of flow problems in porous media is not new, e.g.~\cite{Perugia2000,Maryska2000}, but the resulting systems had the typical structure of symmetric saddle-point problems \cite{Benzi2005}. The main feature of our block preconditioner is twofold: (i) the exploitation of the block matrix decoupling factors to recast the Schur complement, and (ii) their approximated computation by means of appropriate restriction and prolongation operators. The overall preconditioning approach was originally devised in~\cite{Ferronato2019} for coupled flow/poromechanical models, and later on extended to contact mechanics~\cite{Franceschini2019}. The reference model for the development of our preconditioner is the basic MHFE-FV discretized single-phase flow in porous media. However, this represents the first stage of a more extensive research project aimed at designing an algebraic preconditioning framework for a MHFE-FV multi-phase and multi-component reservoir simulator.

The rest of the paper is organized as follows. The model problem, together with the algebraic properties of the system matrix, is first presented, then the block-structured preconditioning framework is introduced and tested in four challenging applications. The experimental stage helped highlight advantages and drawbacks of the proposed preconditioner, which are reported in the discussion section. The conclusions and hints on the ongoing and future work finally close the paper. 


\section{MHFE-FV model of single-phase flow in porous media}
The set of equations governing the single-phase flow in porous media consists of the \emph{mass conservation} and \emph{Darcy's law}. The monolithic solution approach addresses these equations simultaneously by means of a fully-implicit coupling.

\subsection{Governing equations}
Consider the finite porous domain $\Omega \subset \mathbb{R}^3$, its boundary $\Gamma$ and their union $\hat{\Omega}=\Omega \cup \Gamma$. $\Gamma_p$ and $\Gamma_{\bm{v}}$ are partitions of $\Gamma$ such that $\Gamma_p \cup \Gamma_{\bm{v}} = \Gamma$ and $\Gamma_p \cap \Gamma_{\bm{v}} = \varnothing$. Let $t$ and $\mathbb{T}=]0, T [$ indicate the time variable and the simulated open temporal domain, respectively.
Denoting with $s: \Omega \times \mathbb{T} \to \mathbb{R}$ the source or sink term, $p: \hat{\Omega} \times [0,T] \to \mathbb{R}$ the fluid pressure, $\bm{v}: \hat{\Omega} \times [0,T] \to \mathbb{R}^3$ the velocity vector and $c: \hat{\Omega}\to\mathbb{R}^+$ the specific storage coefficient, representative of both the fluid and porous matrix compressibilities, the set of governing PDEs reads:
\begin{linenomath}
\begin{subequations}
\begin{align}
    \bm{v} = -\frac{K}{\gamma} \nabla p & & \text{on} \ \Omega \times \mathbb{T} & & & (\text{Darcy's law}),  \label{eq:darcy} \\
    \nabla \cdot \bm{v} + c \dot{p}= s & & \text{on} \ \Omega \times \mathbb{T} & & &  (\text{mass conservation}), \label{eq:mass_balance}
\end{align}
\end{subequations}
\end{linenomath}
where the symbol $\nabla$ indicates the gradient operator, $\nabla\cdot$ the divergence operator and $\dot{()}$ the derivative with respect to time. In equation~\eqref{eq:darcy}, $\gamma$ is the fluid specific weight and $K$ is the conductivity tensor, assumed to be symmetric and positive definite (SPD). The gravitational term is here neglected. The specific storage coefficient $c$ in equation~\eqref{eq:mass_balance} can be expressed as $c = \gamma(\alpha + \phi \beta)$ where $\alpha$ is the soil compressibility, $\phi$ the medium porosity and $\beta$ the fluid volumetric compressibility~\cite{Gambolati2015}.
The solution to the system of equations~\eqref{eq:darcy} and~\eqref{eq:mass_balance} is a well-posed problem provided that a set of appropriate initial and boundary conditions is supplied:
\begin{linenomath}
\begin{subequations}
\begin{align}
    p|_{t=0} &= p_0 & & \text{in} \ \hat{\Omega}   & & \text{(initial fluid pressure),}\\
    p &= \overline{p} & & \text{on} \ \Gamma_p \times \mathbb{T} & & \text{(prescribed fluid pressure),}\\
    - \frac{K}{\gamma} \nabla p \cdot \bm{n} &= \overline{v_n}  & & \text{on} \ \Gamma_{\bm{v}} \times \mathbb{T} & & \text{(prescribed Darcy's flux),} \label{eq:prescr_flow}
\end{align}
\end{subequations}
\end{linenomath}
for assigned functions $p_0: \hat{\Omega} \to \mathbb{R}$, $\overline{p}: \Gamma_p \times \mathbb{T} \to \mathbb{R}$, and $\overline{v_n}: \Gamma_{\bm{v}} \times \mathbb{T} \to \mathbb{R}$. In equation~\eqref{eq:prescr_flow}, $\bm{n}$ denotes the outer unit normal vector to $\Gamma_{\bm{v}}$.

\subsection{Discretization of the governing equations}
The model domain is partitioned into non-overlapping hexahedral elements, which accommodate, as shown in Figure~\ref{fig:loc_unknown}, two types of pressure unknowns, located on each face barycentre, $\pi$, and on the element centroid, $p^E$. The former act the part of Lagrange multipliers and express the face average pressure, whereas the latter represents the average elemental value. 

Let $\mathcal{E}^h$ and $\mathcal{F}^h$ be the collections of elements and faces of the discretized domain, respectively. In our modelling approach, equation~\eqref{eq:darcy} is discretized by means of the MHFE method, using the $\mathbb{RT}_0$ space to approximate the velocity $\bm{v}$ and the 
$\mathbb{P}_0$ space 
for the pressure 
$\hat{p}$ and Lagrange multipliers $\hat{\pi}$: 
\begin{linenomath}
\begin{subequations}
\begin{align}
    \mathcal{V}^h &= \left\{ \bm{v} \mid \bm{v} \in H(\text{div},\Omega), -\bm{v} \cdot \bm{n} = \overline{v_n} \ \text{on} \ \Gamma_{\bm{v}}, \ \bm{v}|_E \in \mathbb{RT}_0(E), \ \forall E \in \mathcal{E}^h \right\},  \\
    \mathcal{L}^h &= \left\{ \hat{p} \mid \hat{p} \in L^2(\Omega), \ \hat{p}|_E \in \mathbb{P}_0(E), \ \forall E \in \mathcal{E}^h \right\},  \\
    \mathcal{M}^h &= \left\{ \hat{\pi} \mid \hat{\pi} \in L^2(\Psi), \hat{\pi} = \overline{p} \ \text{on} \ \Gamma_p, \ \hat{\pi}|_F \in \mathbb{P}_0(F), \ \forall F \in \mathcal{F}^h \right\},
\end{align}
\end{subequations}
\end{linenomath}
where $L^2(\Omega)$ and $L^2(\Psi)$ denote the spaces of square Lebesgue-integrable functions on the domain $\Omega$ and the union of elemental faces $\Psi$, respectively, and $H(\text{div}, \Omega)$ is the Sobolev space of square integrable vector functions with square integrable divergence in $\Omega$~\cite{Brezzi1991}.

The $\mathcal{V}^h$ trial space for 3-D problems is generated by local piecewise trilinear vector functions, $\bm{\eta}_i^E(x,y,z)$, defined per each face $i$ of element $E$~\cite{Chavent1991,Matringe2007}. Such functions exhibit two basic properties~\cite{Mose1994}:
\begin{enumerate}
\item The flux of function $\bm{\eta}_i^E$ is unitary across face $i$ and null elsewhere:
\begin{linenomath}
\begin{equation}
    \int_{A_j} \bm{\eta}_i^E \cdot \bm{n}_j \ dA = \delta_{ij}, \quad i,j=1,\ldots,N_f^E,
    \label{eq:MFE_property_1}
\end{equation}
\end{linenomath}
where $\bm{n}_j$ denotes the outer normal at face $j$, $A_j$ the relevant area, $\delta_{ij}$ is the Kronecker delta and $N_f^E$ is the number of faces of element $E$. From equation~\eqref{eq:MFE_property_1} it follows that $\bm{\eta}_i^E$ has a continuous normal component at face $i$, so the normal fluxes are also continuous.
\item The integral of the divergence of $\bm{\eta}_i^E$ is unitary over element $E$:
\begin{linenomath}
\begin{equation}
    \int_{\Omega^E} \nabla \cdot \bm{\eta}_i^E \ d \Omega = 1,  \quad i=1,\ldots,N_f^E,
    \label{eq:MFE_property_2}
\end{equation}
\end{linenomath}
where $\Omega^E$ is the elemental volume.
\end{enumerate}
Darcy's velocity $\bm{v}$ is approximated at the elemental level by a linear combination of the basis functions $\bm{\eta}_j^E(x,y,z)$~\cite{Younes2010}:
\begin{linenomath}
\begin{equation}
    \bm{v}^E = \sum_{j=1}^{N_f^E} q_j^E \bm{\eta}_j^E(x,y,z),
    \label{eq:velocity_MHFEM}
\end{equation}
\end{linenomath}
where $q_j^E$ represents the flux across face $j$.

Let $\left\{ \xi^E \right\}_{E \in N_e}$ be the set of basis functions for $\mathcal{L}^h$, where $N_e$ is the number of elements in the grid, 
such that $\xi^E(\bm{x})=1$ if $\bm{x} \in E$ and $\xi^E(\bm{x})=0$ if $\bm{x} \notin E$. Similarly, 
the basis for $\mathcal{M}^h$, $\left\{ \zeta_f \right\}_{f \in N_f}$, with $N_f$ the number of faces in the grid, consists of functions such that $\zeta_f(\bm{x})=1$ if $\bm{x} \in F$ and $\zeta_f(\bm{x})=0$ if $\bm{x} \notin F$. Therefore, the pressure and Lagrange multiplier fields read:
\begin{linenomath}
\begin{equation}
    \hat{p}=\sum_{E=1}^{N_e} \xi^E p^E \qquad \text{and} \qquad \hat{\pi}=\sum_{f=1}^{N_f} \zeta_f \pi_f.
    \label{eq:Lebesgue_press}
\end{equation}
\end{linenomath}
The Galerkin weak form of equation~\eqref{eq:darcy} is written element-by-element as:
\begin{linenomath}
\begin{equation}
    \gamma \int_{\Omega^E} \bm{\eta}_i^{E,T} {K^{E}}^{-1} \bm{v}^E \ d\Omega = -\int_{\Omega^E} \bm{\eta}_i^{E,T} \nabla p \ d\Omega, \qquad i=1,\ldots,N_f^E.
    \label{eq:fluxes_MHFEM}
\end{equation}
\end{linenomath}
Applying the Green-Gauss lemma to the Right-Hand Side (RHS) of equation~\eqref{eq:fluxes_MHFEM} and substituting equations~\eqref{eq:MFE_property_1}, \eqref{eq:MFE_property_2} and \eqref{eq:Lebesgue_press} entails:
\begin{linenomath}
\begin{equation}
\begin{split}
    -\int_{\Omega^E} \bm{\eta}_i^{E,T} \nabla p \ d\Omega & = \int_{\Omega^E} \nabla\cdot\bm{\eta}_i^E \ p \ d\Omega - \sum_{j=1}^{N_f^E} \int_{A_j} \bm{\eta}_i^E \cdot \bm{n}_j \ p \ dA \\
    & = p^E - \pi_i^E, \qquad i=1,\ldots,N_f^E,
\end{split}
\label{eq:fluxes_MHFEM_1}
\end{equation}
\end{linenomath}
where the superscript on $\pi_i^E$ indicates that those unknowns belongs to element $E$.
Introducing equations~\eqref{eq:velocity_MHFEM} and \eqref{eq:fluxes_MHFEM_1} in \eqref{eq:fluxes_MHFEM} yields:
\begin{linenomath}
\begin{equation}
    \gamma \sum_{j=1}^{N_f^E} \int_{\Omega^E} {\bm{\eta}_i^E}^T {K^{E}}^{-1} \bm{\eta}_j^E \ d\Omega \ q_j^E = p^E - \pi_i^E, \qquad i=1,\ldots,{N_f^E}.
    \label{eq:fluxes_MHFEM_2}
\end{equation}
\end{linenomath}
\\ Defining the elementary matrices $B^E \in \mathbb{R}^{N_f^E \times N_f^E}$, whose components are~\cite{Younes2010}:
\begin{linenomath}
\begin{equation}
    B_{ij}^E = \gamma \int_{\Omega^E} {\bm{\eta}_i^E}^T {K^{E}}^{-1} \bm{\eta}^E_j \ d\Omega, \quad i,j=1, \ldots, N_f^E,
    \label{eq:B_ij}
\end{equation}
\end{linenomath}
the local final expression for equation~\eqref{eq:fluxes_MHFEM_2} reads:
\begin{linenomath}
\begin{equation}
    \bm{q}^E = {B^{E}}^{-1}(p^E\bm{1} - \bm{\pi}^E),
    \label{eq:fluxes_MHFEM_3}
\end{equation}
\end{linenomath}
which allows to link the face fluxes with the local pressure differences, being $\bm{q}^E$ and $\bm{\pi}^E$ the vectors gathering the interface fluxes and pressures of element $E$ and $\bm{1}\in \mathbb{R}^{N_f^E}$ the vector of unitary components. Being $K^E$ SPD, $B^E$ is so as well.
The numerical evaluation of integrals~\eqref{eq:B_ij} may be troublesome when performed in the model space with general elements. In this regard, Piola transformation comes into play, allowing to map the element in the model space into the prototype hexahedron in the reference space, perform the integrals and then map back the result to the physical space (see for instance~\cite{Matringe2007,Huyakorn1983,Zienkiewicz2000}).
\begin{figure}
\centering
\begin{tikzpicture} [scale=0.35]
\draw (0,0) -- ++(8,0) -- ++(0,8) -- ++(-8,0) -- ++(0,-8);
\draw (8,0) -- ++(3,4);
\draw (8,8) -- ++(3,4);
\draw (0,8) -- ++(3,4) -- ++(8,0) -- ++(0,-8);
\draw[dashed] (0,0) -- ++(3,4) -- ++(0,8);
\draw[dashed] (3,4) -- ++(8,0);
\node[align=center, font = \normalsize] at (4,4){$\blacktriangle$};
\node[align=center, font = \normalsize] at (5.5,2){$\blacktriangle$};
\node[align=center, font = \normalsize] at (9.5,6){$\blacktriangle$};
\node[align=center, font = \normalsize] at (5.5,10){$\blacktriangle$};
\node[align=center, font = \normalsize] at (1.5,6){$\blacktriangle$};
\node[align=center, font = \normalsize] at (7,8){$\blacktriangle$};
\node[align=center, font = \normalsize] at (5.5,6){$\bigstar$};
\fill[opacity=0.25, cyan] (0,0) -- ++(8,0) -- ++(0,8) -- ++(-8,0) -- ++(0,-8);
\fill[opacity=0.5, cyan] (0,8) -- ++(8,0) -- ++(3,4) -- ++(-8,0) -- ++(-3,-4);
\fill[opacity=0.5, cyan] (8,0) -- ++(3,4) -- ++(0,8) -- ++(-3,-4) -- ++(0,-8);
\node [align=right] at (18.2,2.7){{\normalsize $\bigstar$} $\to$ element pressure $p^E$};
\node [align=right] at (17.5,1){{\normalsize $\blacktriangle$} $\to$ face pressure $\pi_i$};
\end{tikzpicture}
\caption{Location of the unknowns in the hexahedral reference element.}
\label{fig:loc_unknown}
\end{figure}
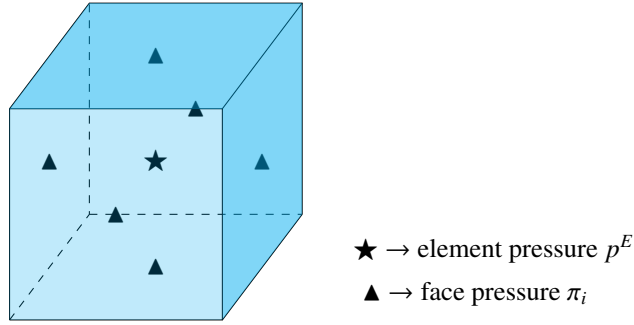

For the discretization of the mass balance equation~\eqref{eq:mass_balance}, we use a FV approximation in space. Choosing the elements of the grid as \emph{control volumes}, 
we have:
\begin{linenomath}
\begin{equation}
    \int_{\Omega^E} c \dot{p} \ d\Omega + \int_{\Omega^E} \nabla \cdot \bm{v} \ d\Omega = \int_{\Omega^E} s \ d\Omega, \quad E=1,\ldots,N_e. 
    \label{eq:mass_bal_FV}
\end{equation}
\end{linenomath}

Recognizing that the second term on the Left-Hand Side (LHS) is equivalent to the sum of the fluxes across the faces of the element, equation~\eqref{eq:mass_bal_FV} gives:
\begin{linenomath}
\begin{equation}
    \Omega^E \overline{c}^E \frac{p^E_{n+1} - p^E_n}{\Delta t_n} + \sum_{i=1}^{N_f^E} q_i^{E,E'} = \Omega^E \overline{s}^E, \quad E=1,\ldots,N_e,
\label{eq:discr_bal_FV_2}
\end{equation}
\end{linenomath}
where a first-order backward Finite Difference scheme has been introduced for the integration in time.
In equation~\eqref{eq:discr_bal_FV_2}, the superscript $n$ indicates the previous time step, $n+1$ the actual one, $\Delta t_n = t_{n+1}-t_n$, $\overline{c}^E$ and $\overline{s}^E$ are the mean values of the storage coefficient and source terms in $E$, and $q_i^{E,E'}$ is the fluid flux exchanged by the adjacent elements $E$ and $E'$ across face $i$.
The expression for the inter-element flux $q_i^{E,E'}$ results from strongly imposing the continuity of local fluxes across face $i$ (see appendix A in \citep{Abushaikha2017} for details):
\begin{linenomath}
\begin{equation}
    q_i^{E,E'} = \frac{{B_{ii}^{E'}}^{-1} \Lambda^E - {B_{ii}^{E}}^{-1} \Lambda^{E'}}{B_{ii}^E + B_{ii}^{E'}},
    \label{eq:fluxes_MHFEM_STRONG}
\end{equation}
\end{linenomath}
where
\begin{linenomath}
\[
    \Lambda^E = L_{B_i^E} p^E - \sum_{j=1}^{N_f^E} {B_{ij}^E}^{-1} \pi_j^E \qquad \text{with} \ i \ne j
    , \qquad 
    L_{B_i^E} = \sum_{j=1}^{N_f^E} {B_{ij}^{E}}^{-1}.
\]
\end{linenomath}
%
Notice that the main consequences of this formulation are a tightened tying of the local fluxes and the enlargement of the native stencil, since $q_i^{E,E'}$ depends not only on the pressure unknowns of $E$ but also on $E'$'s.

\subsection{The MHFE-FV system of equations}
The solution to the model problem is achieved by solving the system of equations \eqref{eq:discr_bal_FV_2} at each time step, along with the strong enforcement of 
the continuity of fluxes across the faces of the grid:
\begin{linenomath}
\begin{equation}
    q_i^E + q_i^{E'} = 0, \qquad i=1, \ldots, N_f,
    \label{eq:cont_fluxes}
\end{equation}
\end{linenomath}
Along the boundary, equation~\eqref{eq:cont_fluxes} allows also to apply Neumann conditions in a strong form, just by substituting the RHS accordingly and dropping $q_i^{E'}$. Notice that equation~\eqref{eq:cont_fluxes} uses the fluxes as expressed in~\eqref{eq:fluxes_MHFEM_3}, unlike equation~\eqref{eq:discr_bal_FV_2}.
Finally, it is implicitly assumed that $\pi_i^{E}=\pi_i^{E'}$ due to continuity reasons.

The resulting system exhibits a $2 \times 2$ block structure, which is solved in a fully-implicit framework, with two types of unknowns:
\begin{linenomath}
\begin{equation}
\mathcal{A}\bm{u} = \bm{b} \ \ \Rightarrow \ \
    \begin{bmatrix}
        A_{\pi \pi} & A_{\pi p} \\
        A_{p \pi} & A_{p p}
    \end{bmatrix}
    \begin{bmatrix}
        \bm{\pi}^{n+1} \\
        \bm{p}^{n+1}
    \end{bmatrix}
    =
    \begin{bmatrix}
        \bm{f}_{\pi} \\
        \bm{f}_p
    \end{bmatrix},
    \label{eq:system}
\end{equation}
\end{linenomath}
where $A_{\pi \pi} \in \mathbb{R}^{N_f \times N_f}$, $A_{p p} \in \mathbb{R}^{N_e \times N_e}$ (with $N_f > N_e$), $\bm{\pi}^{n+1}$ and $\bm{p}^{n+1}$ gather the face and element pressure unknowns, and $\bm{f}_{\pi}$ and $\bm{f}_p$ are the relevant components of the known term. In equation~\eqref{eq:system}, the Lagrange multipliers and the element pressure unknowns are coupled by means of the rectangular blocks $A_{\pi p}$ and $A_{p \pi}$. As to the properties of $\mathcal{A}$, this matrix has a flipped generalized saddle-point structure, it is sparse, non-symmetric and usually ill-conditioned. In particular, $A_{p p}$ has a symmetric structure, though it is not symmetric, $A_{\pi \pi}$ is a symmetric negative definite matrix, and $A_{\pi p} \ne \pm A_{p \pi}^T$. 
As mentioned before, in the context of single-phase flow, system~\eqref{eq:system} is \emph{linear}.

\section{The Explicit Decoupling Factor Approximation preconditioner}
Solving accurately and efficiently the sequence of linear systems~\eqref{eq:system} arising from a MHFE-FV unsteady flow simulation is the major purpose of this study.
Iterative Krylov subspace solvers are mandatory to address the large-size and sparse systems of equations that stem from real-world 3-D models, especially for the low memory requirements and better scalability as compared to direct solvers~\cite{Saad2003}. When the system matrix is non-symmetric, the Bi-Conjugate Gradient Stabilized (Bi-CGStab)~\citep{VanderVorst1992} or the Generalized Minimal Residual (GMRES)~\citep{Saad1986} methods are usually the selected algorithms. However, improving their performance by supplying an appropriate preconditioning operator $\mathcal{P}^{-1}$ is key in order to guarantee a fast and smooth convergence.

It is well-known that an \emph{effective} preconditioner is an operator whose application to a vector should resemble as much as possible that of the inverse $\mathcal{A}^{-1}$ of the system matrix~\cite{Benzi2002,Wathen2015}. 
Therefore, a good starting point for the design of our preconditioner is to consider $\mathcal{A}^{-1}$ and take advantage of its block structure, as it is usually done in saddle-point and general block problems~\cite{Benzi2005,Ferronato2019,Bergamaschi2008,Ferronato2010}. The block LDU decomposition of the system matrix reads:
\begin{linenomath}
\begin{equation}
    \mathcal{A}
    =
    \begin{bmatrix}
        I &  \\
        A_{p \pi} A_{\pi \pi}^{-1} & I
    \end{bmatrix}
    \begin{bmatrix}
        A_{\pi \pi} &   \\
                    & S
    \end{bmatrix}
    \begin{bmatrix}
        I & A_{\pi \pi}^{-1} A_{\pi p} \\
          & I
    \end{bmatrix},
\label{eq:LDU_decomp}
\end{equation}
\end{linenomath}
where $S = A_{p p} - A_{p \pi} A_{\pi \pi}^{-1} A_{\pi p}$ is the so-called Schur complement. The exact inverse of $\mathcal{A}$ in a factorized form reads:
\begin{linenomath}
\begin{equation}
    \mathcal{A}^{-1}
    =
    \begin{bmatrix}
       I & -A_{\pi \pi}^{-1} A_{\pi p} \\
         & I
    \end{bmatrix}
    \begin{bmatrix}
       A_{\pi \pi}^{-1} &  \\
              & S^{-1}
    \end{bmatrix}
    \begin{bmatrix}
       I &  \\
       -A_{p \pi} A_{\pi \pi}^{-1} & I
    \end{bmatrix},
    \label{eq:inv_A_fac}
\end{equation}
\end{linenomath}
where the two decoupling factors are defined as:
\begin{linenomath}
\begin{equation}
    G = -A_{p \pi} A_{\pi \pi}^{-1} \quad \text{and} \quad
    F = -A_{\pi \pi}^{-1} A_{\pi p}.
\label{eq:dec_fac}
\end{equation}
\end{linenomath}
The decoupling factors $F$ and $G$ are also used to compute the Schur complement as:
\begin{linenomath}
\begin{equation}
\begin{split}
    S &= A_{p p} - A_{p \pi} A_{\pi \pi}^{-1} A_{\pi p} \\
    &= A_{p p} - A_{p \pi} A_{\pi \pi}^{-1} A_{\pi \pi} A_{\pi \pi}^{-1} A_{\pi p} \\
    &= A_{p p} - H,
    \label{eq:schur_compl}
\end{split}
\end{equation}
\end{linenomath}
with $H=GA_{\pi \pi}F$.

Considering equations~\eqref{eq:dec_fac}, $F$ and $G$ can be computed explicitly by solving two independent sets of multiple right-hand side (MRHS) systems:
\begin{linenomath}
\begin{subequations}
\begin{align}
    A_{\pi \pi}^T G^T &= -A_{p \pi}^T,  \label{eq:EDFA_systems_1} \\
    A_{\pi \pi}F &= - A_{\pi p}.  \label{eq:EDFA_systems_2}
\end{align}
\end{subequations}
\end{linenomath}
Of course, such an operation cannot be performed exactly because $F$ and $G$ are dense, hence proper approximations have to be introduced. 
The key feature of the proposed approach, denoted as Explicit Decoupling Factor Approximation (EDFA) preconditioner, is the computation of sparse explicit approximations for $F$ and $G$, $\widetilde{F}$ and $\widetilde{G}$, respectively, by means of proper \emph{restriction operators}.
The approximate decoupling factors $\widetilde{F}$ and $\widetilde{G}$ are used to compute a sparsified Schur complement $\widetilde{S}$:
\begin{linenomath}
\begin{equation}
    \widetilde{S} = A_{p p} - \widetilde{G} A_{\pi \pi} \widetilde{F} = A_{p p} - \widetilde{H}.
    \label{eq:sparse_S}
\end{equation}
\end{linenomath}
Recalling equation~\eqref{eq:inv_A_fac}, the final algebraic expression of the EDFA preconditioner reads:
\begin{linenomath}
\begin{equation}
    \mathcal{P}^{-1}
    =
    \begin{bmatrix}
       I & -\widetilde{A}_{\pi \pi}^{-1} A_{\pi p} \\
         & I
    \end{bmatrix}
    \begin{bmatrix}
       \widetilde{A}_{\pi \pi}^{-1} & \\
        & \widetilde{S}^{-1}
    \end{bmatrix}
    \begin{bmatrix}
       I &  \\
       - A_{p \pi} \widetilde{A}_{\pi \pi}^{-1} & I
    \end{bmatrix},
    \label{eq:EDFA_prec}
\end{equation}
\end{linenomath}
where $\widetilde{A}_{\pi \pi}^{-1}$ and $\widetilde{S}^{-1}$ are inexact applications of the inverse of the leading block $A_{\pi \pi}$ and the approximate Schur complement, respectively. 

\begin{rem}
\label{rem:1}
The approximate decoupling factors $\widetilde{F}$ and $\widetilde{G}$ are used only for the computation of $\widetilde{S}$ and do not replace the relevant terms in the triangular factors in equation~\eqref{eq:EDFA_prec}. In this sense, the EDFA algorithm can be regarded as a member of the class of \emph{mixed constraint preconditioners} \cite{Bergamaschi2008,Ferronato2012}, where a twofold approximation for the inverse of the leading block is inherently introduced. Similarly, it can be also viewed as an example of application in a non-symmetric context of the \emph{multigrid reduction} framework, e.g.~\cite{Bui2018,Bui2020}, where face and elemental pressures play the role of \emph{fine} and \emph{coarse} nodes, respectively, and $\widetilde{F}$ and $\widetilde{G}$ are approximations of the optimal restriction and prolongation operators from the fine to the coarse grid.
\end{rem}

The approximation of the decoupling factors $F$ and $G$ is performed by solving the sequence of MRHS systems~\eqref{eq:EDFA_systems_1} and \eqref{eq:EDFA_systems_2} inexactly in properly restricted subspaces. For the sake of simplicity, we refer to system~\eqref{eq:EDFA_systems_1}, but the same developments can be easily extended to~\eqref{eq:EDFA_systems_2}.
The $m$-th system to be solved reads:
\begin{linenomath}
\begin{equation}
    -A_{\pi \pi} \bm{g}^{(m),T} = \bm{a}_{p \pi}^{(m),T},
    \label{eq:syst_MRHS}
\end{equation}
\end{linenomath}
where $A_{\pi \pi}=A_{\pi \pi}^T$ for symmetry reasons, $\bm{g}^{(m),T} = G^T \bm{e}^{(m)}$, $\bm{a}_{p \pi}^{(m),T} = A_{p \pi}^T \bm{e}^{(m)}$ and $\bm{e}^{(m)}$ is the $m$-th vector of the canonical basis of $\mathbb{R}^{N_e}$, which plays the role of restriction operator over columns. 
The minus sign has been introduced at both sides of \eqref{eq:syst_MRHS} to obtain an SPD problem, since $A_{\pi \pi}$ is negative definite.
Let us now consider the set $Q=\{1,\ldots,N_f\}\subset\mathbb{N}$ and a sequence of (possibly overlapping) subsets $Q^{(m)} \subseteq Q$, whose size is $|Q^{(m)}| = s^{(m)}$, with $m=1,\ldots, N_e$.
The $m$-th restriction operator over rows,
\begin{linenomath}
\begin{equation}
    R_{r}^{(m)}: \mathbb{R}^{N_f} \to \mathbb{R}^{s^{(m)}}
\end{equation}
\end{linenomath}
is expressed as:
\begin{linenomath}
\begin{equation}
    R_{r}^{(m)} =
    \begin{bmatrix}
        \bm{f}_{Q_1^{(m)}}^T \\
        \vdots \\
        \bm{f}_{Q_{s^{(m)}}^{(m)}}^T
    \end{bmatrix},
\label{eq:r_r^m}
\end{equation}
\end{linenomath}
where $\bm{f}_\ell$ is the $\ell$-th column vector of the canonical basis of $\mathbb{R}^{N_f}$ and $Q_i^{(m)}$ is the $i$-th member of $Q^{(m)}$. 
The application of the operator $R_r^{(m)}$ to equation~\eqref{eq:syst_MRHS} leads to the following system (Figure \ref{fig:EBDFA_sylv}):
\begin{linenomath}
\begin{equation}
    -A_{\pi \pi}^{(m)} \widetilde{\bm{g}}^{(m),T} = R_r^{(m)} \bm{a}_{p \pi}^{(m),T},
    \label{eq:syst_MRHS_res}
\end{equation}
\end{linenomath}
where $A_{\pi \pi}^{(m)} = R_r^{(m)} A_{\pi \pi} R_r^{(m),T}$ is a symmetric restriction of $A_{\pi \pi}$ to the entries in the rows and columns with indices in $Q^{(m)}$ and $\widetilde{\bm{g}}^{(m),T} = R_r^{(m)} \bm{g}^{(m),T}$ is the restriction of the $m$-th row of $G$ to the entries in the columns with indices in $Q^{(m)}$.
Since $-A_{\pi \pi}^{(m)}$ is a symmetric square submatrix of the SPD matrix $-A_{\pi \pi}$, it is guaranteed to be SPD as well. The sequence of systems~\eqref{eq:syst_MRHS_res} can be inexpensively solved by an inner direct solver, provided that the cardinality of $Q^{(m)}$ is small enough. 
\begin{figure}[tb]
\centering
\begin{tikzpicture}[scale=0.18]
\node [align=center] at (9,22){$-A_{\pi \pi}$};
\node [align=center] at (41.5,22){$A_{p \pi}^T$};
\node [align=center] at (27.5,22){$\widetilde{G}^T$};
\node [align=center] at (6.5,-4){$-A_{\pi \pi}^{(m)}$};
\node [align=center] at (24,-2.5){$\widetilde{\bm{g}}^{(m),T}$};
\node [align=center] at (45,-5){$R_r^{(m)} \bm{a}_{p \pi}^{(m),T}$};

%
%
%
%
\draw [thin, dashed] (2,0) rectangle ++ (1,20);
\draw [thin, dashed] (5,0) rectangle ++ (1,20);
\draw [thin, dashed] (9,0) rectangle ++ (1,20);
\draw [thin, dashed] (0,17) rectangle ++ (20,1);
\draw [thin, dashed] (0,14) rectangle ++ (20,1);
\draw [thin, dashed] (0,10) rectangle ++ (20,1);
\draw [thin] (0,0) rectangle (20,20);

\fill [opacity=0.6, cyan] (4,0) rectangle ++ (1,1);
\fill [opacity=0.6, cyan] (9,0) rectangle ++ (1,1);
\fill [opacity=0.6, cyan] (14,0) rectangle ++ (1,1);
\fill [opacity=0.6, cyan] (17,0) rectangle ++ (1,1);
\fill [opacity=0.6, cyan] (19,0) rectangle ++ (1,1);
\draw [thin] (4,0) rectangle ++ (1,1);
\draw [thin] (9,0) rectangle ++ (1,1);
\draw [thin] (14,0) rectangle ++ (1,1);
\draw [thin] (17,0) rectangle ++ (1,1);
\draw [thin] (19,0) rectangle ++ (1,1);
%
\fill [opacity=0.6, cyan] (9,1) rectangle ++ (1,1);
\fill [opacity=0.6, cyan] (13,1) rectangle ++ (1,1);
\fill [opacity=0.6, cyan] (18,1) rectangle ++ (1,1);
\draw [thin] (9,1) rectangle ++ (1,1);
\draw [thin] (13,1) rectangle ++ (1,1);
\draw [thin] (18,1) rectangle ++ (1,1);
%
\fill [opacity=0.6, cyan] (3,2) rectangle ++ (1,1);
\fill [opacity=0.6, cyan] (11,2) rectangle ++ (1,1);
\fill [opacity=0.6, cyan] (17,2) rectangle ++ (1,1);
\fill [opacity=0.6, cyan] (19,2) rectangle ++ (1,1);
\draw [thin] (3,2) rectangle ++ (1,1);
\draw [thin] (11,2) rectangle ++ (1,1);
\draw [thin] (17,2) rectangle ++ (1,1);
\draw [thin] (19,2) rectangle ++ (1,1);
%
\fill [opacity=0.6, cyan] (4,3) rectangle ++ (1,1);
\fill [opacity=0.6, cyan] (10,3) rectangle ++ (1,1);
\fill [opacity=0.6, cyan] (13,3) rectangle ++ (1,1);
\fill [opacity=0.6, cyan] (15,3) rectangle ++ (1,1);
\fill [opacity=0.6, cyan] (16,3) rectangle ++ (1,1);
\draw [thin] (4,3) rectangle ++ (1,1);
\draw [thin] (10,3) rectangle ++ (1,1);
\draw [thin] (13,3) rectangle ++ (1,1);
\draw [thin] (15,3) rectangle ++ (1,1);
\draw [thin] (16,3) rectangle ++ (1,1);
%
\fill [opacity=0.6, cyan] (6,4) rectangle ++ (1,1);
\fill [opacity=0.6, cyan] (12,4) rectangle ++ (1,1);
\fill [opacity=0.6, cyan] (15,4) rectangle ++ (1,1);
\fill [opacity=0.6, cyan] (16,4) rectangle ++ (1,1);
\draw [thin] (6,4) rectangle ++ (1,1);
\draw [thin] (12,4) rectangle ++ (1,1);
\draw [thin] (15,4) rectangle ++ (1,1);
\draw [thin] (16,4) rectangle ++ (1,1);
%
\fill [opacity=0.6, cyan] (0,5) rectangle ++ (1,1);
\fill [opacity=0.6, cyan] (6,5) rectangle ++ (1,1);
\fill [opacity=0.6, cyan] (11,5) rectangle ++ (1,1);
\fill [opacity=0.6, cyan] (14,5) rectangle ++ (1,1);
\fill [opacity=0.6, cyan] (19,5) rectangle ++ (1,1);
\draw [thin] (0,5) rectangle ++ (1,1);
\draw [thin] (6,5) rectangle ++ (1,1);
\draw [thin] (11,5) rectangle ++ (1,1);
\draw [thin] (14,5) rectangle ++ (1,1);
\draw [thin] (19,5) rectangle ++ (1,1);
%
\fill [opacity=0.6, cyan] (5,6) rectangle ++ (1,1);
\fill [opacity=0.6, cyan] (13,6) rectangle ++ (1,1);
\fill [opacity=0.6, cyan] (16,6) rectangle ++ (1,1);
\fill [opacity=0.6, cyan] (18,6) rectangle ++ (1,1);
\draw [thin] (5,6) rectangle ++ (1,1);
\draw [thin] (13,6) rectangle ++ (1,1);
\draw [thin] (16,6) rectangle ++ (1,1);
\draw [thin] (18,6) rectangle ++ (1,1);
%
\fill [opacity=0.6, cyan] (8,7) rectangle ++ (1,1);
\fill [opacity=0.6, cyan] (12,7) rectangle ++ (1,1);
\fill [opacity=0.6, cyan] (15,7) rectangle ++ (1,1);
\draw [thin] (8,7) rectangle ++ (1,1);
\draw [thin] (12,7) rectangle ++ (1,1);
\draw [thin] (15,7) rectangle ++ (1,1);
%
\fill [opacity=0.6, cyan] (1,8) rectangle ++ (1,1);
\fill [opacity=0.6, cyan] (2,8) rectangle ++ (1,1);
\fill [opacity=0.6, cyan] (11,8) rectangle ++ (1,1);
\fill [opacity=0.6, cyan] (14,8) rectangle ++ (1,1);
\fill [opacity=0.6, cyan] (17,8) rectangle ++ (1,1);
\draw [thin] (1,8) rectangle ++ (1,1);
\draw [thin] (2,8) rectangle ++ (1,1);
\draw [thin] (11,8) rectangle ++ (1,1);
\draw [thin] (14,8) rectangle ++ (1,1);
\draw [thin] (17,8) rectangle ++ (1,1);
%
\fill [opacity=0.6, cyan] (3,9) rectangle ++ (1,1);
\fill [opacity=0.6, cyan] (7,9) rectangle ++ (1,1);
\fill [opacity=0.6, cyan] (10,9) rectangle ++ (1,1);
\fill [opacity=0.6, cyan] (16,9) rectangle ++ (1,1);
\draw [thin] (3,9) rectangle ++ (1,1);
\draw [thin] (7,9) rectangle ++ (1,1);
\draw [thin] (10,9) rectangle ++ (1,1);
\draw [thin] (16,9) rectangle ++ (1,1);
%
\fill [opacity=1, red] (5,10) rectangle ++ (1,1);
\fill [opacity=1, red] (9,10) rectangle ++ (1,1);
\fill [opacity=0.6, cyan] (18,10) rectangle ++ (1,1);
\fill [opacity=0.6, cyan] (19,10) rectangle ++ (1,1);
\draw [thin] (5,10) rectangle ++ (1,1);
\draw [thin] (9,10) rectangle ++ (1,1);
\draw [thin] (18,10) rectangle ++ (1,1);
\draw [thin] (19,10) rectangle ++ (1,1);
%
\fill [opacity=0.6, cyan] (2,11) rectangle ++ (1,1);
\fill [opacity=0.6, cyan] (7,11) rectangle ++ (1,1);
\fill [opacity=0.6, cyan] (8,11) rectangle ++ (1,1);
\fill [opacity=0.6, cyan] (12,11) rectangle ++ (1,1);
\draw [thin] (2,11) rectangle ++ (1,1);
\draw [thin] (7,11) rectangle ++ (1,1);
\draw [thin] (8,11) rectangle ++ (1,1);
\draw [thin] (12,11) rectangle ++ (1,1);
%
\fill [opacity=0.6, cyan] (3,12) rectangle ++ (1,1);
\fill [opacity=0.6, cyan] (7,12) rectangle ++ (1,1);
\fill [opacity=0.6, cyan] (8,12) rectangle ++ (1,1);
\fill [opacity=0.6, cyan] (10,12) rectangle ++ (1,1);
\draw [thin] (3,12) rectangle ++ (1,1);
\draw [thin] (7,12) rectangle ++ (1,1);
\draw [thin] (8,12) rectangle ++ (1,1);
\draw [thin] (10,12) rectangle ++ (1,1);
%
\fill [opacity=0.6, cyan] (0,13) rectangle ++ (1,1);
\fill [opacity=0.6, cyan] (6,13) rectangle ++ (1,1);
\fill [opacity=0.6, cyan] (14,13) rectangle ++ (1,1);
\fill [opacity=0.6, cyan] (15,13) rectangle ++ (1,1);
\draw [thin] (0,13) rectangle ++ (1,1);
\draw [thin] (6,13) rectangle ++ (1,1);
\draw [thin] (14,13) rectangle ++ (1,1);
\draw [thin] (15,13) rectangle ++ (1,1);
%
\fill [opacity=1, red] (2,14) rectangle ++ (1,1);
\fill [opacity=1, red] (5,14) rectangle ++ (1,1);
\fill [opacity=1, red] (9,14) rectangle ++ (1,1);
\fill [opacity=0.6, cyan] (13,14) rectangle ++ (1,1);
\draw [thin] (2,14) rectangle ++ (1,1);
\draw [thin] (5,14) rectangle ++ (1,1);
\draw [thin] (9,14) rectangle ++ (1,1);
\draw [thin] (13,14) rectangle ++ (1,1);
%
\fill [opacity=0.6, cyan] (1,15) rectangle ++ (1,1);
\fill [opacity=0.6, cyan] (4,15) rectangle ++ (1,1);
\fill [opacity=0.6, cyan] (16,15) rectangle ++ (1,1);
\fill [opacity=0.6, cyan] (19,15) rectangle ++ (1,1);
\draw [thin] (1,15) rectangle ++ (1,1);
\draw [thin] (4,15) rectangle ++ (1,1);
\draw [thin] (16,15) rectangle ++ (1,1);
\draw [thin] (19,15) rectangle ++ (1,1);
%
\fill [opacity=0.6, cyan] (3,16) rectangle ++ (1,1);
\fill [opacity=0.6, cyan] (7,16) rectangle ++ (1,1);
\fill [opacity=0.6, cyan] (10,16) rectangle ++ (1,1);
\fill [opacity=0.6, cyan] (17,16) rectangle ++ (1,1);
\draw [thin] (3,16) rectangle ++ (1,1);
\draw [thin] (7,16) rectangle ++ (1,1);
\draw [thin] (10,16) rectangle ++ (1,1);
\draw [thin] (17,16) rectangle ++ (1,1);
%
\fill [opacity=1, red] (2,17) rectangle ++ (1,1);
\fill [opacity=1, red] (5,17) rectangle ++ (1,1);
\fill [opacity=0.6, cyan] (8,17) rectangle ++ (1,1);
\fill [opacity=0.6, cyan] (11,17) rectangle ++ (1,1);
\draw [thin] (2,17) rectangle ++ (1,1);
\draw [thin] (5,17) rectangle ++ (1,1);
\draw [thin] (8,17) rectangle ++ (1,1);
\draw [thin] (11,17) rectangle ++ (1,1);
%
\fill [opacity=0.6, cyan] (0,18) rectangle ++ (1,1);
\fill [opacity=0.6, cyan] (1,18) rectangle ++ (1,1);
\fill [opacity=0.6, cyan] (4,18) rectangle ++ (1,1);
\fill [opacity=0.6, cyan] (11,18) rectangle ++ (1,1);
\draw [thin] (0,18) rectangle ++ (1,1);
\draw [thin] (1,18) rectangle ++ (1,1);
\draw [thin] (4,18) rectangle ++ (1,1);
\draw [thin] (11,18) rectangle ++ (1,1);
%
\fill [opacity=0.6, cyan] (0,19) rectangle ++ (1,1);
\fill [opacity=0.6, cyan] (1,19) rectangle ++ (1,1);
\fill [opacity=0.6, cyan] (6,19) rectangle ++ (1,1);
\fill [opacity=0.6, cyan] (14,19) rectangle ++ (1,1);
\draw [thin] (0,19) rectangle ++ (1,1);
\draw [thin] (1,19) rectangle ++ (1,1);
\draw [thin] (6,19) rectangle ++ (1,1);
\draw [thin] (14,19) rectangle ++ (1,1);
%
%
%
%
\draw [thin, dashed, shift={(22,20)}, rotate = -90] (0,5) rectangle ++ (20,1);
\draw [thin, shift={(22,20)}, rotate = -90] (0,0) rectangle (20,10);

%
%
\node [align=center] at (27.5,17.5){$\times$};
\node [align=center] at (27.5,14.5){$\times$};
\node [align=center] at (27.5,10.5){$\times$};
\draw [thin, shift={(22,20)}, rotate = -90] (2,5) rectangle ++ (1,1);
\draw [thin, shift={(22,20)}, rotate = -90] (5,5) rectangle ++ (1,1);
\draw [thin, shift={(22,20)}, rotate = -90] (9,5) rectangle ++ (1,1);
\node [align=center] at (34,10){$=$};
\node [align=center] at (-2,10){$-$};
%
%
%
%
%
%
%
\draw [thin, dashed, shift={(36,20)}, rotate = -90] (0,5) rectangle ++ (20,1);
\draw [thin, shift={(36,20)}, rotate = -90] (0,0) rectangle (20,10);
\fill [opacity=0.6, cyan, shift={(36,20)}, rotate = -90] (3,0) rectangle ++ (1,1);
\fill [opacity=0.6, cyan, shift={(36,20)}, rotate = -90] (5,0) rectangle ++ (1,1);
\fill [opacity=0.6, cyan, shift={(36,20)}, rotate = -90] (15,0) rectangle ++ (1,1);
\draw [thin, shift={(36,20)}, rotate = -90] (3,0) rectangle ++ (1,1);
\draw [thin, shift={(36,20)}, rotate = -90] (5,0) rectangle ++ (1,1);
\draw [thin, shift={(36,20)}, rotate = -90] (15,0) rectangle ++ (1,1);
%
\fill [opacity=0.6, shift={(36,20)}, cyan, rotate = -90] (0,1) rectangle ++ (1,1);
\fill [opacity=0.6, shift={(36,20)}, cyan, rotate = -90] (1,1) rectangle ++ (1,1);
\fill [opacity=0.6, shift={(36,20)}, cyan, rotate = -90] (9,1) rectangle ++ (1,1);
\fill [opacity=0.6, shift={(36,20)}, cyan, rotate = -90] (19,1) rectangle ++ (1,1);
\draw [thin, shift={(36,20)}, rotate = -90] (0,1) rectangle ++ (1,1);
\draw [thin, shift={(36,20)}, rotate = -90] (1,1) rectangle ++ (1,1);
\draw [thin, shift={(36,20)}, rotate = -90] (9,1) rectangle ++ (1,1);
\draw [thin, shift={(36,20)}, rotate = -90] (19,1) rectangle ++ (1,1);
%
\fill [opacity=0.6, cyan, shift={(36,20)}, rotate = -90] (7,2) rectangle ++ (1,1);
\fill [opacity=0.6, cyan, shift={(36,20)}, rotate = -90] (10,2) rectangle ++ (1,1);
\fill [opacity=0.6, cyan, shift={(36,20)}, rotate = -90] (17,2) rectangle ++ (1,1);
\draw [thin, shift={(36,20)}, rotate = -90] (7,2) rectangle ++ (1,1);
\draw [thin, shift={(36,20)}, rotate = -90] (10,2) rectangle ++ (1,1);
\draw [thin, shift={(36,20)}, rotate = -90] (17,2) rectangle ++ (1,1);
%
\fill [opacity=0.6, cyan, shift={(36,20)}, rotate = -90] (2,3) rectangle ++ (1,1);
\fill [opacity=0.6, cyan, shift={(36,20)}, rotate = -90] (6,3) rectangle ++ (1,1);
\fill [opacity=0.6, cyan, shift={(36,20)}, rotate = -90] (16,3) rectangle ++ (1,1);
\draw [thin, shift={(36,20)}, rotate = -90] (2,3) rectangle ++ (1,1);
\draw [thin, shift={(36,20)}, rotate = -90] (6,3) rectangle ++ (1,1);
\draw [thin, shift={(36,20)}, rotate = -90] (16,3) rectangle ++ (1,1);
%
\fill [opacity=0.6, cyan, shift={(36,20)}, rotate = -90] (3,4) rectangle ++ (1,1);
\fill [opacity=0.6, cyan, shift={(36,20)}, rotate = -90] (8,4) rectangle ++ (1,1);
\fill [opacity=0.6, cyan, shift={(36,20)}, rotate = -90] (12,4) rectangle ++ (1,1);
\draw [thin, shift={(36,20)}, rotate = -90] (3,4) rectangle ++ (1,1);
\draw [thin, shift={(36,20)}, rotate = -90] (8,4) rectangle ++ (1,1);
\draw [thin, shift={(36,20)}, rotate = -90] (12,4) rectangle ++ (1,1);
%
\fill [opacity=1, red, shift={(36,20)}, rotate = -90] (2,5) rectangle ++ (1,1);
\fill [opacity=1, red, shift={(36,20)}, rotate = -90] (5,5) rectangle ++ (1,1);
\fill [opacity=1, red, shift={(36,20)}, rotate = -90] (9,5) rectangle ++ (1,1);
\draw [thin, shift={(36,20)}, rotate = -90] (2,5) rectangle ++ (1,1);
\draw [thin, shift={(36,20)}, rotate = -90] (5,5) rectangle ++ (1,1);
\draw [thin, shift={(36,20)}, rotate = -90] (9,5) rectangle ++ (1,1);
%
\fill [opacity=0.6, cyan, shift={(36,20)}, rotate = -90] (3,6) rectangle ++ (1,1);
\fill [opacity=0.6, cyan, shift={(36,20)}, rotate = -90] (7,6) rectangle ++ (1,1);
\fill [opacity=0.6, cyan, shift={(36,20)}, rotate = -90] (11,6) rectangle ++ (1,1);
\draw [thin, shift={(36,20)}, rotate = -90] (3,6) rectangle ++ (1,1);
\draw [thin, shift={(36,20)}, rotate = -90] (7,6) rectangle ++ (1,1);
\draw [thin, shift={(36,20)}, rotate = -90] (11,6) rectangle ++ (1,1);
%
\fill [opacity=0.6, cyan, shift={(36,20)}, rotate = -90] (1,7) rectangle ++ (1,1);
\fill [opacity=0.6, cyan, shift={(36,20)}, rotate = -90] (6,7) rectangle ++ (1,1);
\fill [opacity=0.6, cyan, shift={(36,20)}, rotate = -90] (13,7) rectangle ++ (1,1);
\draw [thin, shift={(36,20)}, rotate = -90] (1,7) rectangle ++ (1,1);
\draw [thin, shift={(36,20)}, rotate = -90] (6,7) rectangle ++ (1,1);
\draw [thin, shift={(36,20)}, rotate = -90] (13,7) rectangle ++ (1,1);
%
\fill [opacity=0.6, cyan, shift={(36,20)}, rotate = -90] (4,8) rectangle ++ (1,1);
\fill [opacity=0.6, cyan, shift={(36,20)}, rotate = -90] (8,8) rectangle ++ (1,1);
\fill [opacity=0.6, cyan, shift={(36,20)}, rotate = -90] (15,8) rectangle ++ (1,1);
\draw [thin, shift={(36,20)}, rotate = -90] (4,8) rectangle ++ (1,1);
\draw [thin, shift={(36,20)}, rotate = -90] (8,8) rectangle ++ (1,1);
\draw [thin, shift={(36,20)}, rotate = -90] (15,8) rectangle ++ (1,1);
%
\fill [opacity=0.6, cyan, shift={(36,20)}, rotate = -90] (0,9) rectangle ++ (1,1);
\fill [opacity=0.6, cyan, shift={(36,20)}, rotate = -90] (2,9) rectangle ++ (1,1);
\fill [opacity=0.6, cyan, shift={(36,20)}, rotate = -90] (18,9) rectangle ++ (1,1);
\draw [thin, shift={(36,20)}, rotate = -90] (0,9) rectangle ++ (1,1);
\draw [thin, shift={(36,20)}, rotate = -90] (2,9) rectangle ++ (1,1);
\draw [thin, shift={(36,20)}, rotate = -90] (18,9) rectangle ++ (1,1);
%
%
%
\fill [opacity=1, red, shift={(22,-9)}] (1.2,0) rectangle ++ (1.2,1.2);
\fill [opacity=1, red, shift={(22,-9)}] (2.4,0) rectangle ++ (1.2,1.2);
\fill [opacity=1, red, shift={(22,-9)}] (0,1.2) rectangle ++ (1.2,1.2);
\fill [opacity=1, red, shift={(22,-9)}] (1.2,1.2) rectangle ++ (1.2,1.2);
\fill [opacity=1, red, shift={(22,-9)}] (2.4,1.2) rectangle ++ (1.2,1.2);
\fill [opacity=1, red, shift={(22,-9)}] (0,2.4) rectangle ++ (1.2,1.2);
\fill [opacity=1, red, shift={(22,-9)}] (1.2,2.4) rectangle ++ (1.2,1.2);
\draw [thin, shift={(22,-9)},step=1.2] (0,0) grid ++ (3.6,3.6);
\node [align=center] at (27.6,-6){\large $\times$};
\node [align=center] at (27.6,-7.2){\large $\times$};
\node [align=center] at (27.6,-8.4){\large $\times$};
\draw [thin, shift={(27,-5.4)}, step = 1.2] (0,0) grid ++ (1.2,-3.6);
\fill [opacity=1, red, shift={(31.2,-5.4)}] (0,0) rectangle ++ (1.2,-1.2);
\fill [opacity=1, red, shift={(31.2,-5.4)}] (0,-1.2) rectangle ++ (1.2,-1.2);
\fill [opacity=1, red, shift={(31.2,-5.4)}] (0,-2.4) rectangle ++ (1.2,-1.2);
\draw [thin, shift={(31.2,-5.4)}, step = 1.2] (0,0) grid ++ (1.2,-3.6);
\node [align=center] at (29.6,-7.2){\footnotesize $=$};
\node [align=center] at (21,-7.2){\footnotesize $-$};
\draw [-latex,semithick](10,-1) to [out=270, in=180] (20,-7.2);
\draw [-latex,semithick](27.5,-1) to [out=270, in=90] (27.5,-4);
\draw [-latex,semithick](41,-1) to [out=270, in=0] (33.5,-7.2);
\end{tikzpicture}
\caption{Schematic representation of the action of restriction operators in the computation of $\widetilde{G}^T$.}
\label{fig:EBDFA_sylv}
\end{figure}
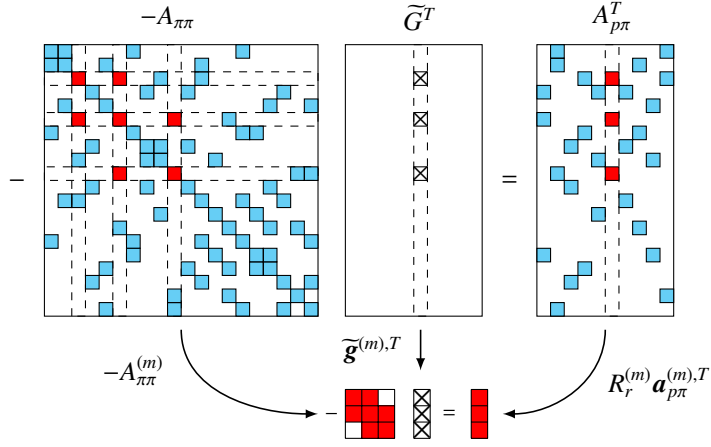

The restricted vector $\widetilde{\bm{g}}^{(m)}$ obtained from the solution of system \eqref{eq:syst_MRHS_res} is an approximation of the $m$-th row of the exact decoupling factor $G$ and inherits an optimal property, as stated by the following result.
\begin{prop}
\label{thm:min_energy}
Let $A \in \mathbb{R}^{n\times n}$ be SPD and $R \in \mathbb{R}^{m\times n}$ ($m<n$) be a restriction operator from $\mathbb{R}^n$ to $\mathbb{R}^m$. Then, for any right-hand side vector $\bm{b} \in \mathbb{R}^n$, the solution $\bm{x}\in\mathbb{R}^m$ to the restricted system:
\begin{linenomath}
\begin{equation}
    R A R^T \bm{x} = R \bm{b}
    \label{eq:SPDrestr_sys}
\end{equation}
\end{linenomath}
is such that the error $\bm{e}=A^{-1}\bm{b}-R^T\bm{x}$ has minimal energy norm with respect to the $A$-inner product.
\end{prop}
\begin{proof}
The energy norm of $\bm{e}$ with respect to $A$ reads:
\begin{equation}
    \left\| \bm{e} \right\|_A = \sqrt{\bm{e}^T A \bm{e}}.
    \label{eq:energy-norm}
\end{equation}
The contribution under square root in \eqref{eq:energy-norm} is a quadratic function $\Phi(\bm{x}):\mathbb{R}^{m}\rightarrow\mathbb{R}^+$:
\begin{eqnarray}
\Phi \left( \bm{x} \right) &=& \left( A^{-1} \bm{b} - R^T \bm{x} \right)^T A \left( A^{-1} \bm{b} - R^T \bm{x} \right) \nonumber \\
&=& \bm{x}^T R A R^T \bm{x} - 2 \bm{x}^T R \bm{b} + \bm{b}^T A^{-1} \bm{b},
\label{eq:Phi}
\end{eqnarray}
which has a unique minimum in $\mathbb{R}^m$ being $A$ SPD. Hence:
\begin{equation}
    \min_{\mathbb{R}^n} \left\| \bm{e} \right\|_A = \sqrt{\Phi \left( \bm{t} \right)}, \quad \mbox{with} \; \bm{t} = \arg\min_{\bm{x}\in\mathbb{R}^m} \Phi \left( \bm{x} \right) \quad \Longleftrightarrow \quad \nabla \Phi \left( \bm{x} \right) = \bm{0}.
    \label{eq:min_energy}
\end{equation}
Condition \eqref{eq:min_energy} applied to equation \eqref{eq:Phi} immediately yields:
\begin{equation}
    R A R^T \bm{x} - R \bm{b} = \bm{0},
    \label{eq:min_energy-2}
\end{equation}
which completes the proof.
\end{proof}
\begin{rem}
\label{rem:2}
Proposition \ref{thm:min_energy} guarantees that the restricted vector $\widetilde{\bm{g}}^{(m)}$ is the best approximation of $\bm{g}^{(m)}$ that can be computed for the components selected by the set $Q^{(m)}$, in the sense of the energy norm with respect to the $A_{\pi \pi}$-inner product. Hence, an accurate selection of such components, so as to identify the most important ones for each column, is fundamental for the quality of the approximation $\widetilde{G}$ and, similarly, of $\widetilde{F}$ and $\widetilde{H}$. 
\end{rem}

Finally, the assemblage of the $N_e$ contributions $\widetilde{\bm{g}}^{(m)}$ from equation~\eqref{eq:syst_MRHS_res}, prolonged back to $\mathbb{R}^{N_f}$, gives rise to the approximate factor $\widetilde{G}$. Recalling that restrictions and prolongations are dual operators, $\widetilde{G}$ is easily obtained as:
\begin{linenomath}
\begin{equation}
    \widetilde{G} = \sum_{m=1}^{N_e} \bm{e}^{(m)} \widetilde{\bm{g}}^{(m)}
    R_r^{(m)}.
    \label{eq:syls_syst_RED}
\end{equation}
\end{linenomath}
Operating similarly for equation~\eqref{eq:EDFA_systems_2}, we obtain:
\begin{linenomath}
\begin{equation}
    \widetilde{F} = \sum_{m=1}^{N_e} R_r^{(m),T} \widetilde{\bm{f}}^{(m)} \bm{e}^{(m),T},
    \label{eq:syls_syst_BLACK}
\end{equation}
\end{linenomath}
where $\widetilde{\bm{f}}^{(m)}$ are the solution of the $N_e$ restricted SPD systems:
\begin{linenomath}
\begin{equation}
    -A_{\pi \pi}^{(m)} \widetilde{\bm{f}}^{(m)} = R_r^{(m)} \bm{a}_{\pi p}^{(m)},
    \label{eq:syst_MRHS_res-F}
\end{equation}
\end{linenomath}
with $\bm{a}_{\pi p}^{(m)}=A_{\pi p} \bm{e}^{(m)}$. Of course, the restriction operators $R_r^{(m)}$ can be the same as for $\widetilde{G}$ or based on a different sequence of subsets $W^{(m)}\subseteq Q$.

As observed in Remark \ref{rem:2}, the sequence of subsets $Q^{(m)}$, $W^{(m)}$, along with their size $s^{(m)}$, affects the density of the approximate decoupling factors and governs the effectiveness of the EDFA preconditioner. In fact, the entries of $Q^{(m)}$ and $W^{(m)}$ are the indices of the non-zero entries computed for the $m$-th row of $\widetilde{G}$ and column of $\widetilde{F}$, respectively. To be effective, the sets $Q^{(m)}$ and $W^{(m)}$ should roughly identify for each row of $\widetilde{G}$ and column of $\widetilde{F}$ the largest entries of $G$ and $F$. 
This key operation is carried out by means of two techniques, referred to as \emph{static} and \emph{dynamic} in the sequel, aimed at selecting the most influential entries expected in $G$ and $F$. 

First of all, for the sake of simplicity we use a single sequence of sets $Q^{(m)}$ for both decoupling factors.
A natural initial guess for $Q^{(m)}$ 
is the set of indices of the non-zero entries belonging to the columns of $A_{p \pi}^T$, which is denser than $A_{\pi p}$.  
Such pattern is referred to as $Q^{(m)}_{A_{p \pi}}$. 
Figure~\ref{fig:EBDFA_sylv} schematically shows how the restricted systems can be retrieved from the global one using the set $Q^{(m)}_{A_{p \pi}}$. The two strategies for computing $Q^{(m)}$ starting from $Q^{(m)}_{A_{p \pi}}$ are as follows.

\begin{enumerate}
\item \emph{Static technique.}
The non-zero entries of $Q^{(m)}_{A_{p \pi}}$ can be derived by the discretization. In particular, the non-zeros lying in the $m$-th row of $A_{p \pi}$ identify the faces of the cells connected with the $m$-th element, as illustrated 
in Figure~\ref{fig:base_patterns}. Notice that the front and right elements have been removed for the sake of readability, being the overall patch symmetric along the three principal directions. The central element (red-filled faces) is the $m$-th cell, which is connected to six adjacent elements, and the colored faces correspond to the indices of the non-zero entries in the $m$-th row of $A_{p\pi}$. Note that, depending on whether the grid is structured or unstructured, the patterns are different. This is a direct consequence of the structure of the elemental matrices $B^{E,-1}$ of equation~\eqref{eq:B_ij}, which derives from the mutual relationships among the basis functions of the $\mathbb{RT}_0$ space for hexahedral elements. For a regular hexahedron, in fact, matrix $B^{E,-1}$ is block-diagonal, while 
this property is no longer valid for a general-shaped hexahedron. 
Since the solution of the system \eqref{eq:syst_MRHS} can be physically interpreted as the face pressures induced by the fluid fluxes related to the pressure gradients occurring in neighboring cells, the pattern $Q_{A_{p \pi}}^{(m)}$ can be extended by adding the connection to faces belonging to close cells where the pressure perturbation is expected to propagate. 
From an algebraic viewpoint, the static technique is based on partitioning the problem domain into overlapping subregions built around each cell and keeping the face connections. 
\item \emph{Dynamic technique.}
The starting pattern $Q_{(0)}^{(m)}=Q_{A_{p \pi}}^{(m)}$ is progressively enlarged during the computation of $\widetilde{\bm{g}}^{(m)}$ and $\widetilde{\bm{f}}^{(m)}$ with the aid of an iterative strategy.
After computing $\widetilde{\bm{g}}_{(0)}^{(m),T}$ from the solution of system~\eqref{eq:syst_MRHS_res}, with the restriction operator $R_{r,(0)}^{(m)}$ built on $Q_{(0)}^{(m)}$, 
the residual of the prolonged system
\begin{linenomath}
\begin{equation}
    \bm{r}_{(0)}^{(m)} = \bm{a}_{p \pi}^{(m),T} + A_{\pi \pi} R_{r,(0)}^{(m),T} \widetilde{\bm{g}}_{(0)}^{(m),T}
    \label{eq:prol_res}
\end{equation}
\end{linenomath}
is obtained and used to expand $Q_{(0)}^{(m)}$ by incorporating the indices of the largest components 
of $\bm{r}_{(0)}^{(m)}$, thus obtaining $Q_{(1)}^{(m)}$. The process can be iterated to obtain $Q_{(2)}^{(m)}$, $Q_{(3)}^{(m)}$, etc., until a certain exit criterion is met. 
The same procedure applies to the system \eqref{eq:syst_MRHS_res-F}.
\end{enumerate}
\begin{figure}[tb]
     \centering
     \subfloat[][Structured grid]{\includegraphics[width=1.5cm]{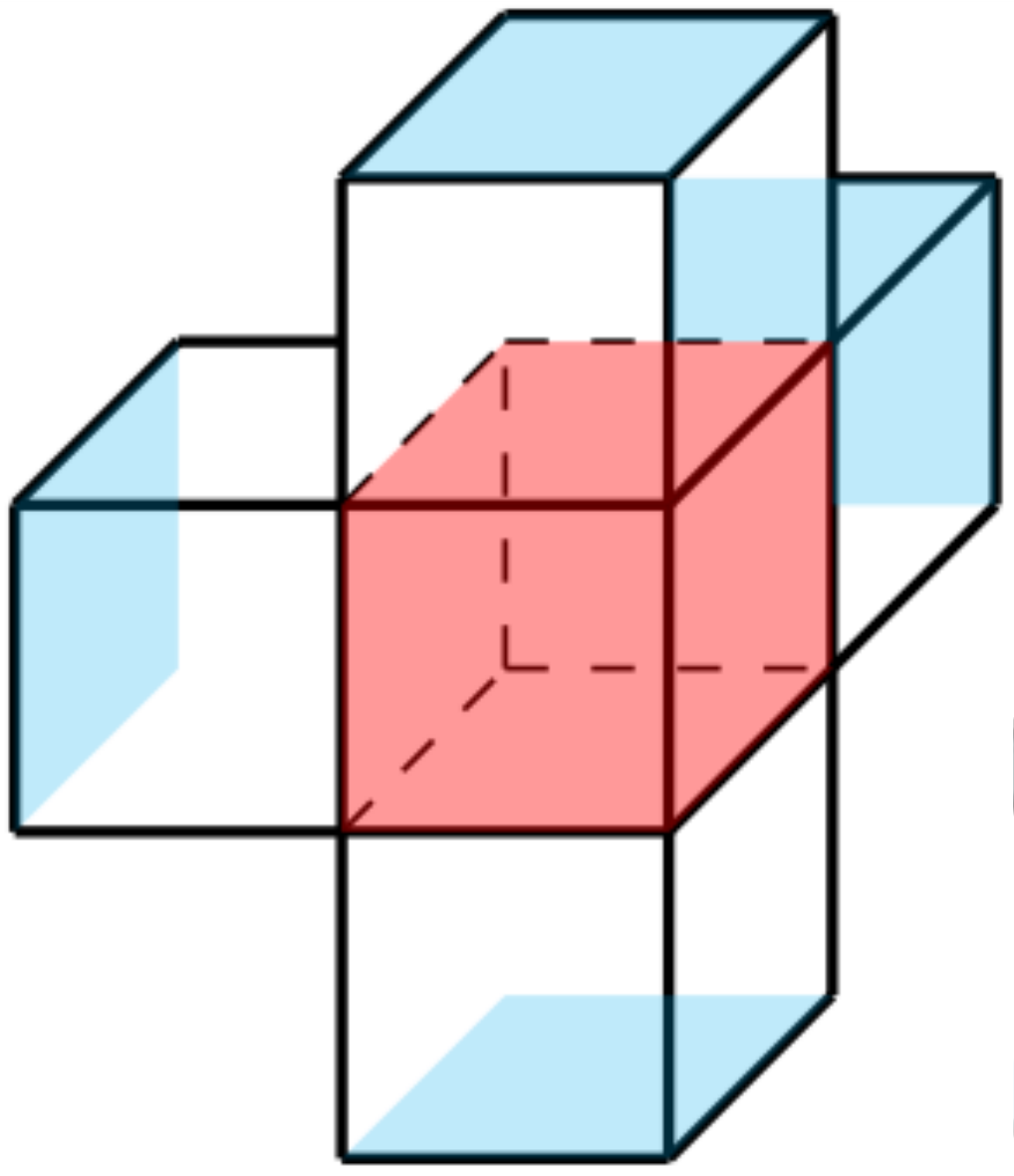} \label{fig:pattern_struct}}
     \hspace{3cm}
     \subfloat[][Unstructured grid]{\includegraphics[width=1.5cm]{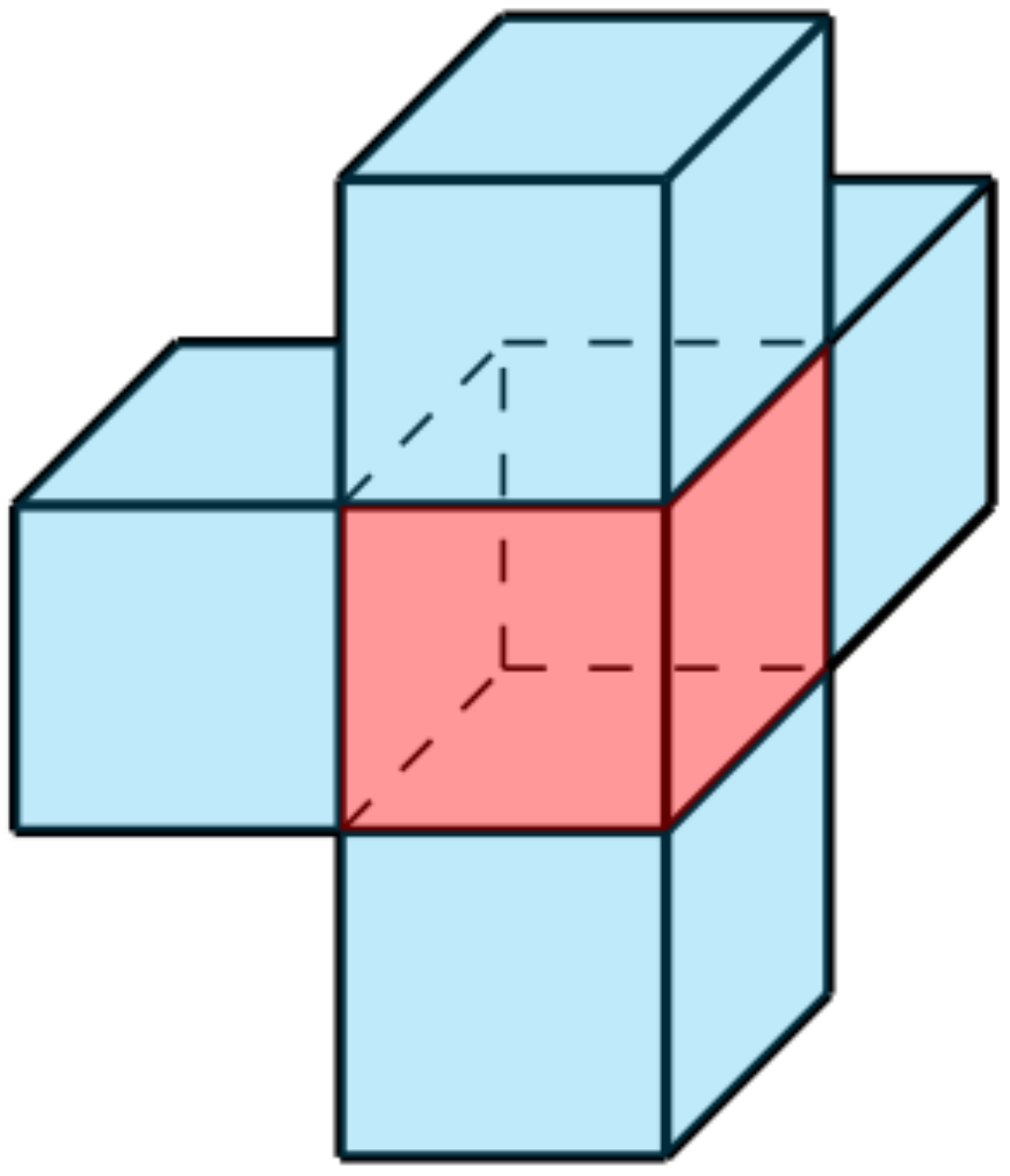} \label{fig:pattern_unstruct}}
     \caption{Sketch of the base patterns for structured and unstructured grids. The front and right elements have been removed to improve the readability of the subpanels.}
     \label{fig:base_patterns}
\end{figure}

\begin{rem}
\label{rem:3}
Equation \eqref{eq:prol_res} is not expensive to compute. In fact, the matrix $A_{\pi \pi} R_{r,0}^{(m),T}$ is the restriction of $A_{\pi \pi}$ to the columns with indices in $Q^{(m)}_{(0)}$. However, such columns are sparse and contain only the connection of a face with the faces of the two sharing cells. Hence, the only non-zero entries of $\bm{r}^{(m)}_{(0)}$ correspond to the indices of the faces of a set of neighboring elements. In practice, the dynamic strategy automatically selects the most significant entries among a subset of potential indices that should resemble the one associated with the static strategy. 
\end{rem}

\begin{rem}
\label{rem:4}
The use of the prolonged residual to select the most significant entries to be retained is strictly related to the symmetry and positive definiteness of $-A_{\pi \pi}$. In fact, $\bm{r}_{(0)}^{(m)}$ is the direction of the gradient of the quadratic form associated to $-A_{\pi \pi}$, whose absolute minimum is the exact solution to \eqref{eq:syst_MRHS}. Therefore, the dynamic strategy can be also regarded as an incomplete steepest descent process, where only the largest contributions to the gradient direction are taken into account.
\end{rem}

\begin{rem}
\label{rem:5}
The EDFA preconditioner, in both the static and dynamic variants, exhibits the remarkable feature that its computation is embarrassingly parallel. 
In fact, the row- and column-wise approach, used to tackle the restricted solution to the MRHS systems~\eqref{eq:EDFA_systems_1} and \eqref{eq:EDFA_systems_2}, allows to solve each single linear system independently of the others. All the available processing units can be assigned batches of systems that are approximately solved at the same time, with a full and effective exploitation of the most modern computational architectures.
\end{rem}

\subsection{Implementation details}
The static and dynamic variants of the EDFA preconditioner 
require a set of user-specified elements to be properly set up. 

The static technique needs the sets $Q^{(m)}\subseteq\{1,2,\ldots,N_f\}$ for $m=1,\ldots,N_e$, which correspond to the indices of faces connected to a certain cell. The level of such a connection, i.e., the neighbours, or the neighbours of the neighbours, and so on, is defined by means of a domain partition into overlapping subregions built around each cell. These subregions are defined on the basis of physical considerations related to the expected directions of fluxes.

The dynamic variant can be regarded as fully algebraic and requires a set of user-specified parameters controlling the enlargement of the initial set $Q^{(m)}_{(0)}$ defined for $m=1,\ldots,N_e$. 
Assuming $Q^{(m)}_{(0)}=Q^{(m)}_{A_{p \pi}}$, the selected
user-defined parameters are:
\begin{itemize}
    \item $n_{\text{add}}$: maximum number of entries added to $Q^{(m)}_{(k-1)}$ at the $k$-th step of the dynamic procedure;
    \item $n_{\text{ent}}$: total maximum number of new entries added to $Q^{(m)}_{(0)}$.
\end{itemize}
%
The iterative process continues until 
$n_{\text{ent}}$ has been reached.  
Alternatively,
it is also possible to set a maximum number of iterations, $it_{\max}$, instead of $n_{\text{ent}}$. 

The computation of $\widetilde{F}$, $\widetilde{G}$ and $\widetilde{S}=A_{p p}-\widetilde{H}$, with either the static or dynamic technique, is followed by a check of the non-zero entries size. 
\emph{Pre-} and \emph{post-filtration} techniques are implemented with the purpose of further sparsifying the approximate Schur complement by discarding those entries whose absolute value is smaller than a user-defined tolerance, namely $\tau_\text{filt}$, relative to the Euclidean norm of the corresponding row.
Performing pre- and/or post-filtration produces an additional cost in the preconditioner set-up, which might be anyway beneficial at the application stage. With the aim at preventing possible breakdowns in the inexact application of $\widetilde{S}^{-1}$, all the diagonal entries are preserved irrespective of the dropping threshold.

Recalling that $A_{p p}$ is the only block of $\mathcal{A}$ changing during a transient simulation, the preconditioner set-up can be split into two stages. The first one, which can be carried out only once at the beginning of the simulation and then recycled at every system solution, consists of the computation of $\widetilde{H}$, i.e., the most time demanding operation, a pre-filtration of $\widetilde{G}$ and $\widetilde{F}$, if needed, and the inner preconditioner for the inexact application of $\widetilde{A}_{\pi \pi}^{-1}$.
The second one, performed at the beginning of each time step, includes the update of $\widetilde{S} = A_{p p} - \widetilde{H}$, the post-filtration, if required, and the computation of the inner preconditioner for the inexact application of $\widetilde{S}^{-1}$. 
In summary, Algorithms~\ref{alg:EDFA_first_stage} and \ref{alg:EDFA_second_stage} provide an overview of the sequence of operations needed to compute the first and second stage of the EDFA preconditioner in both its variants.
\begin{algorithm}
\caption{ {\sc EDFA Computation: Stage 1} 
[$\widetilde{H},\widetilde{A}_{\pi \pi}^{-1}$] = EDFA\_first\_stage($N_e, Q^{(m)}, n_{\textup{ent}}, n_{\textup{add}}, it_{\max}, \tau_{\textup{filt}}, A_{\pi \pi}, A_{\pi p}, A_{p \pi}$).}
\label{alg:EDFA_first_stage}
\begin{algorithmic}[1] \small
\If{EDFA\_stat}
    \For {$m \gets 1,N_e$}
        \State $\bm{a}_{p \pi}^{(m),T} = A_{p \pi}^T \bm{e}^{(m)}$, \quad $\bm{a}_{\pi p}^{(m)} = A_{\pi p} \bm{e}^{(m)}$
        \State Build $R_r^{(m)}$ based on $Q^{(m)}$ (Equation~\eqref{eq:r_r^m})
        \State $A_{\pi \pi}^{(m)} = R_r^{(m)} A_{\pi \pi} R_r^{(m),T}$
        \State Solve $-A_{\pi \pi}^{(m)} \widetilde{\bm{g}}^{(m),T} = R_r^{(m)} \bm{a}_{p \pi}^{(m),T}$ 
        \State Solve $-A_{\pi \pi}^{(m)} \widetilde{\bm{f}}^{(m)} = R_r^{(m)} \bm{a}_{ \pi p}^{(m)}$
        \State Perform \textbf{Pre-filtration} on $\widetilde{\bm{g}}^{(m),T}$ and $\widetilde{\bm{f}}^{(m)}$ with tolerance $\tau_{\textup{filt}}$, if required
        \State $\widetilde{G} \gets \widetilde{G} + \bm{e}^{(m)} \widetilde{\bm{g}}^{(m)} R_r^{(m)}$, \quad $\widetilde{F} \gets \widetilde{F} + R_r^{(m),T} \widetilde{\bm{f}}^{(m)} \bm{e}^{(m),T}$
    \EndFor
\ElsIf{EDFA\_Dynamic}
\For {$m \gets 1,N_e$}
    \State $\bm{a}_{p \pi}^{(m),T} = A_{p \pi}^T \bm{e}^{(m)}$, \quad $\bm{a}_{\pi p}^{(m)} = A_{\pi p} \bm{e}^{(m)}$
    \State Build $R_{r,(0)}^{(m)}$ based on $Q_{(0)}^{(m)}=Q_{A_{p \pi}}^{(m)}$ (Equation~\eqref{eq:r_r^m})
    \State $A_{\pi \pi,(0)}^{(m)} = R_{r,(0)}^{(m)} A_{\pi \pi} R_{r,(0)}^{(m),T}$
    \State Solve $-A_{\pi \pi,(0)}^{(m)} \widetilde{\bm{g}}_{(0)}^{(m),T} = R_{r,(0)}^{(m)} \bm{a}_{p \pi}^{(m),T}$
    \State Compute $\bm{r}_{(0)}^{(m)} = \bm{a}_{p \pi}^{(m),T} + A_{\pi \pi} R_{r,(0)}^{(m),T} \widetilde{\bm{g}}_{(0)}^{(m),T}$
    \State $n_{\textup{prog}} = 0, \quad k=0$ \Comment{Counter of the total number of new entries added to $Q_{(0)}^{(m)}$ and total number of swipes}
    \While {$n_{\textup{prog}} < n_{\textup{ent}}$ and $k < it_{\max}$}
        \State $k \gets k + 1$
        \State $n = \min(n_{\textup{add}}, n_{\textup{ent}} - n_{\textup{prog}})$
        \State Obtain $Q_{(k)}^{(m)}$ by adding to $Q_{(k-1)}^{(m)}$ at most $n$ new indices associated with the largest components of $|\bm{r}_{(k-1)}^{(m)}|$
        \State Update $n_{\textup{prog}}$
        \State Build $R_{r,(k)}^{(m)}$ based on $Q_{(k)}^{(m)}$ (Equation~\eqref{eq:r_r^m})
        \State $A_{\pi \pi, (k)}^{(m)} = R_{r,(k)}^{(m)} A_{\pi \pi} R_{r,(k)}^{(m),T}$
        \State Solve $-A_{\pi \pi,(k)}^{(m)} \widetilde{\bm{g}}_{(k)}^{(m),T} = R_{r,(k)}^{(m)} \bm{a}_{p \pi}^{(m),T}$
        \State Compute $\bm{r}_{(k)}^{(m)} = \bm{a}_{p \pi}^{(m),T} + A_{\pi \pi} R_{r,(k)}^{(m),T} \widetilde{\bm{g}}_{(k)}^{(m),T}$
    \EndWhile
    \State Solve $-A_{\pi \pi, (k)}^{(m)} \widetilde{\bm{f}}^{(m)} = R_{r,(k)}^{(m)} \bm{a}_{ \pi p}^{(m)}$
    \State Perform \textbf{Pre-filtration} on $\widetilde{\bm{g}}_{(k)}^{(m),T}$ and $\widetilde{\bm{f}}^{(m)}$ with tolerance $\tau_{\textup{filt}}$, if required
    \State $\widetilde{G} \gets \widetilde{G} + \bm{e}^{(m)} \widetilde{\bm{g}}^{(m)} R_r^{(m)}$, \quad $\widetilde{F} \gets \widetilde{F} + R_r^{(m),T} \widetilde{\bm{f}}^{(m)} \bm{e}^{(m),T}$
\EndFor
\EndIf
\State Compute $\widetilde{H}=\widetilde{G} A_{\pi \pi} \widetilde{F}$
\State Perform \textbf{Post-filtration} on $\widetilde{H}$ with tolerance $\tau_{\textup{filt}}$, if required
\State Compute the inner preconditioner for the inexact application of $\widetilde{A}_{\pi \pi}^{-1}$
\end{algorithmic}
\end{algorithm}

\begin{algorithm}
\caption{ {\sc EDFA Computation: Stage 2} 
[$\widetilde{S}^{-1}$]=EDFA\_second\_stage($A_{p p}, \widetilde{H}, \tau_{\textup{filt}}$).}
\label{alg:EDFA_second_stage}
\begin{algorithmic}[1] \small
\State $\widetilde{S} = A_{p p} - \widetilde{H}$
\State Perform \textbf{Post-filtration} on $\widetilde{S}$ with tolerance $\tau_{\textup{filt}}$, if required
\State Compute the inner preconditioner for the inexact application of $\widetilde{S}^{-1}$
\end{algorithmic}
\end{algorithm}

\section{Numerical results}
\label{sec:num_results}
The computational performance of the EDFA preconditioner is investigated in both synthetic and real-world reservoir applications.
Four test cases are considered, with grid consisting of four layers taken from the SPE10 model~\cite{Christie2001} and comprising 51,741 elements and 171,070 faces, for a total of 222,811 unknowns. The scenario being tested, depicted in Figure~\ref{fig:SPE10_domains}, represents a reservoir with a producer located in the centre and four injectors, one at each corner. The wells intercept the full thickness of the reservoir. The initial water pressure in the reservoir is uniform and equal to 140 bar, with the producer and injectors pumping at a constant pressure of 100 and 200 bar, respectively. Different variants of the model domain have been considered. In Test 1 and 3, the grid is cartesian with a regular hexahedral discretization (Figure~\ref{fig:SPE10_struct}), whereas in Test 2 and 4 the planar structure has been deformed into a dome (Figure~\ref{fig:SPE10_unstruct}). The resulting grids are, therefore, structured and unstructured, respectively. In all tests, porosity spans the interval [2.6E-5,0.5] with a spatial distribution that follows the SPE10 benchmark properties, rock $(\alpha)$ and water $(\beta)$ compressibilities are 4.67E-5 $\frac{1}{\text{bar}}$ and 4.84E-5 $\frac{1}{\text{bar}}$, respectively, and the water specific weight $(\gamma)$ is 0.101 $\frac{\text{bar}}{\text{m}}$.

\begin{figure}
     \centering
     \subfloat[][]{\includegraphics[width=7.6cm]{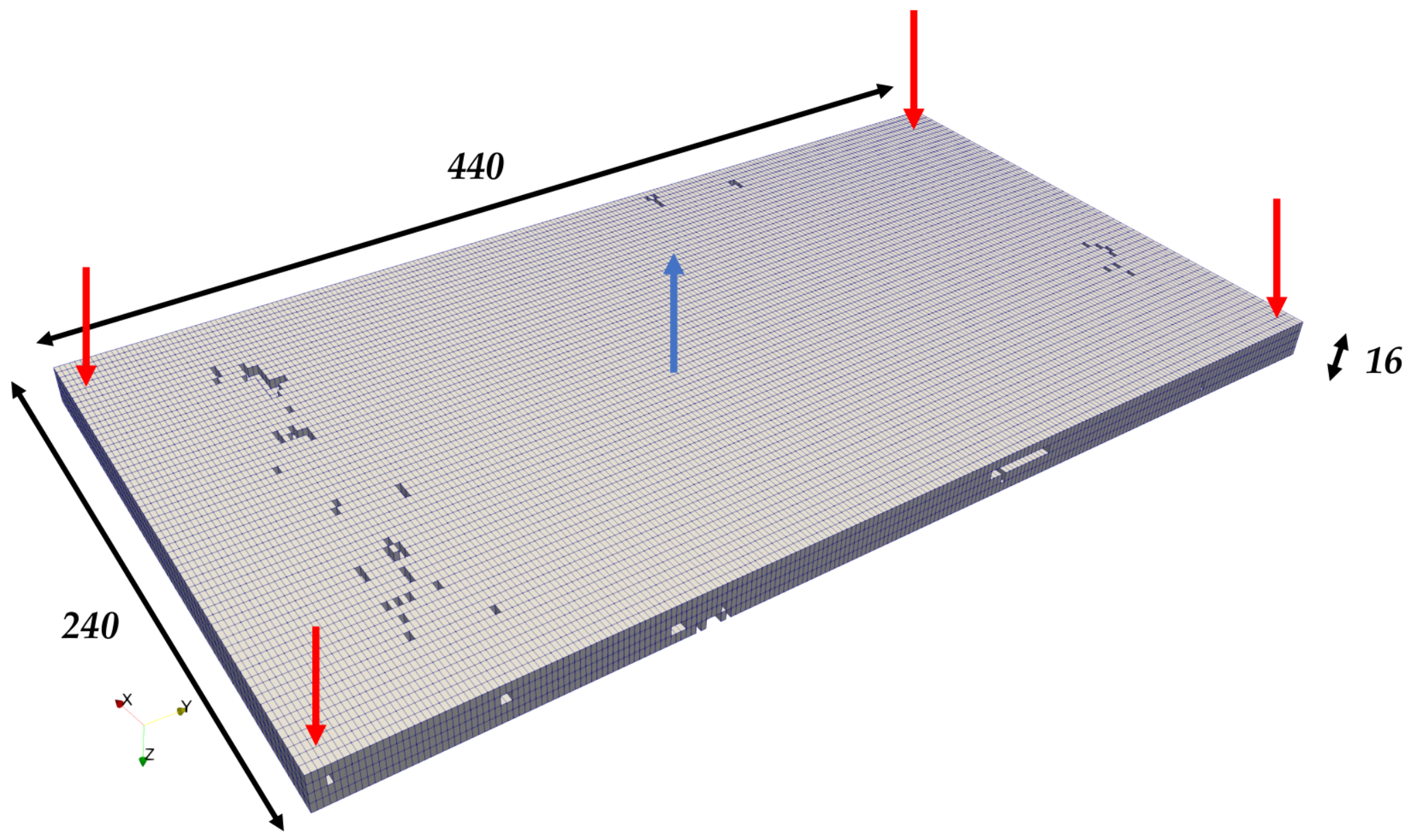} \label{fig:SPE10_struct}}
     \hspace{0.4cm}
     \subfloat[][]{\includegraphics[width=7.6cm]{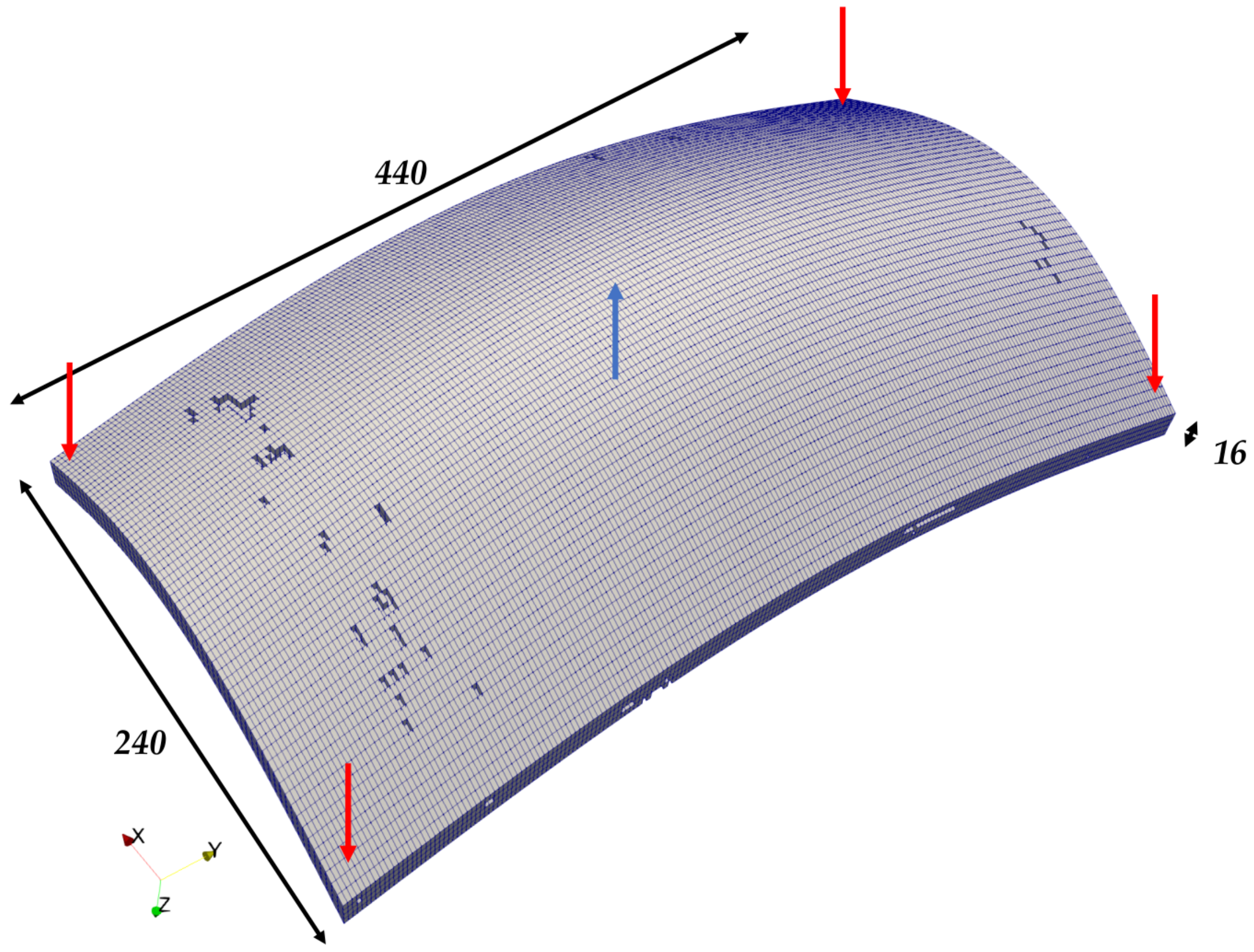} \label{fig:SPE10_unstruct}}
     \caption{Planar (a) and dome-structured (b) reservoirs, used as domains for Tests 1, 3 and 2, 4, respectively. The blue and red arrows indicate the position of the producer and injectors, respectively. The domain size is in meters.}
    \label{fig:SPE10_domains}
\end{figure}

A summary of the test cases and their main properties is reported in Table~\ref{tab:test_setup}.
Test 1, which is characterized by a homogeneous isotropic hydraulic conductivity in the form of a diagonal tensor 
is aimed at introducing the operative principles of the proposed preconditioner variants. 
A sensitivity analysis is carried out on the patterns selected for the static variant and
on the two governing user-specified parameters for the dynamic technique. 
Then, the EDFA preconditioner is employed in a transient simulation to evaluate the effect of time, and in particular the size of the time step, $\Delta t$, on its performance. 
Test 2 preserves the same hydraulic properties as Test 1, but highlights the influence of an unstructured mesh 
in the optimal setting of the preconditioner.
Finally, Tests 3 and 4 investigate the efficiency and robustness of the EDFA preconditioner in challenging real-world conditions. Specifically, Test 3 exhibits a highly heterogeneous and anisotropic conductivity distribution, as derived from the properties of the SPE10 model and expressed in the form of a diagonal tensor. 
On the contrary, the dome reservoir in Test 4 is characterized by a heterogeneous and isotropic conductivity field with a full tensor, obtained by extending the horizontal conductivity values ($K_{x,y}$) to the vertical direction ($K_z$) and rotating the principal axes of the resulting tensor so as to follow the curvature of the dome reservoir.
The sensitivity analysis on the EDFA preconditioner performance is carried out for the system at steady state, then the
overall performance is investigated in full-transient simulations. 
%
%
\begin{table}[tb] \footnotesize
\caption{Setup of the test cases and number of non-zeros of the resulting matrices, where $N_f=171,070$ and $N_e=51,741$. The values of the hydraulic conductivity in brackets are the minimum and maximum of the portion of the SPE10 model used herein. The distribution of the conductivity values throughout the domain follows that of the SPE10 benchmark.}
\label{tab:test_setup}
\centering
\vspace{0.2cm}
\begin{tabular}{lccccc}
\toprule
Test & & 1 & 2 & 3 & 4 \\
\midrule
Reservoir type & & Plain & Dome & Plain & Dome \\
Grid type & & Structured & Unstructured & Structured & Unstructured \\
\multirow{2}*{Cond. tensor properties} & & Homogeneous & Homogeneous & Heterogeneous & Heterogeneous \\
  &  & Isotropic & Isotropic & Anisotropic & Isotropic \\
Cond. tensor type   & & Diagonal & Diagonal & Diagonal & Full \\
Horiz. conductivity & $\left[ \frac{\text{m}}{\text{d}} \right]$ & 1.73E-5 & 1.73E-5 & [3.0E-7, 2.0E0] & [3.0E-7, 2.0E0] \\
Vert. conductivity & $\left[ \frac{\text{m}}{\text{d}} \right]$ & 1.73E-5 & 1.73E-5 & [3.9E-11, 6.0E-1] & [3.0E-7, 2.0E0]\\
\midrule
nnz($\mathcal{A}$) & & 1,711,914 & 4,069,590 & 1,711,914 & 4,069,590 \\
nnz($A_{\pi \pi}$) & & 481,636   & 1,723,900 & 481,636   & 1,723,900 \\
nnz($A_{\pi p}$)   & & 310,446   & 310,446 & 310,446   & 310,446   \\
nnz($A_{p \pi}$)   & & 589,299   & 1,704,711 & 589,299   & 1,704,711\\
nnz($A_{p p}$)     & & 330,533   & 330,533 & 330,533   & 330,533   \\
\bottomrule
\end{tabular}
\end{table}
%

%

Bi-CGStab~\cite{VanderVorst1992}, with the null vector used as initial guess, is elected as Krylov subspace method to solve the sequence of non-symmetric linear systems~\eqref{eq:system}. The exit criterion for the iteration count relies on the reduction of the 2-norm of the relative residual below a prescribed threshold $\tau$, i.e., $|| \bm{r}_k ||_2 / || \bm{r}_0 ||_2 \leq \tau$, where $k$ is the iteration number and $\tau = 10^{-8}$. 
The computational performance of the preconditioned Bi-CGStab solver is monitored by using the following indicators: (i) the iteration count, $n_{\text{it}}$, (ii) the preconditioner density, $\mu$, defined as
\begin{linenomath}
\begin{equation}
    \mu = \frac{\mathtt{nnz}(\widetilde{A}_{\pi \pi}^{-1}) + \mathtt{nnz}(A_{\pi p}) + \mathtt{nnz}(A_{p \pi}) + \mathtt{nnz}(\widetilde{S}^{-1})}{\mathtt{nnz}(A_{\pi \pi}) + \mathtt{nnz}(A_{\pi p}) + \mathtt{nnz}(A_{p \pi}) + \mathtt{nnz}(A_{p p})},
\end{equation}
\end{linenomath}
where the function $\mathtt{nnz}()$ provides the number of non-zeros stored for a sparse matrix,
and (iii) the CPU time split into 
$t_{p_0}$, $t_p$ and $t_s$, needed to perform the first and second stage of the EDFA preconditioner setup (Algorithm \ref{alg:EDFA_first_stage} and \ref{alg:EDFA_second_stage}, respectively)
and to iterate to convergence. 
We denote by $t_t = t_p + t_s$ the total time associated with the solution of the linear system in a single time step.

For the transient simulations, we consider also the Courant-Friedrichs-Lewy (CFL) number, which is defined as~\cite{Cao2002,Coats2001}:
\begin{linenomath}
\begin{equation}
    \chi^E = \frac{Q^E \Delta t}{\Omega^E \phi^E}
    \label{eq:CFL_number}
\end{equation}
\end{linenomath}
where $Q^E$ is the water flux through the $E$-th element during a time step of size $\Delta t$. Specifically, two measures are reported depending on the type of analysis:
\begin{linenomath}
\begin{equation}
    \chi_{\infty}= \max_E \left( \chi^E \right) \qquad \text{and} \qquad \overline{\chi_{\infty}} = \frac{\sum_{i=1}^{n_{\textup{step}}}\chi_{\infty}^i}{n_{\textup{step}}}
\end{equation}
\end{linenomath}
where $n_{\textup{step}}$ is the number of temporal steps in the simulation.
The size of the time steps is dynamically adjusted during the transient simulations in order to stabilize the pressure change between two consecutive time steps. The underlying criterion relies on the maximum pressure difference at the two previous steps, $\Delta p_{\max}=\max_E(p_n^E-p_{n-1}^E)$, and a user defined goal for the pressure change, $\Delta p_T$, to define the optimal size of the next time step:
\begin{linenomath}
\begin{equation}
    \Delta t_{n+1} = \min \left\{ \Delta t_n \min \left\{ \Delta t_{\textup{mult}}, \frac{\Delta p_T}{\Delta p_{\max}} \right\}, \Delta t_{\max} \right\}
    \label{eq:delta_t_dyn}
\end{equation}
\end{linenomath}
where $\Delta t_{\textup{mult}}$ is a predefined multiplicative factor and $\Delta t_{\max}$ the maximum time step length. A relaxation factor can be introduced also in equation~\eqref{eq:delta_t_dyn}~\cite{Abushaikha2017}. 

The number of non-zeros of matrix $\mathcal{A}$ and its submatrices for the four test cases is reported in Table~\ref{tab:test_setup}. Notice that it depends only on the grid type. 
Both the solver and the preconditioner are implemented in Matlab. For the inexact application of both $\widetilde{A}_{\pi \pi}^{-1}$ and $\widetilde{S}^{-1}$ we use the incomplete factorizations with partial fill-in degree already available in Matlab. Of course, other powerful strategies, also more prone to a fully parallel implementation, can be used and will be considered in future developments of the algorithm. 
For the computation of $\widetilde{F}$ and $\widetilde{G}$, the $\mathtt{parfor}$ operator has been exploited.
All numerical tests were carried out on an Intel\textregistered Core\texttrademark i7 Quad-Core at 2.9 GHz with 16 GB of RAM.
%
%

\subsection{Test 1: Planar reservoir with homogeneous and isotropic hydraulic conductivity}
\label{sec:test_1}
\begin{figure}[tb]
     \centering
     \subfloat[][Pattern A]{\includegraphics[width=2.2cm]{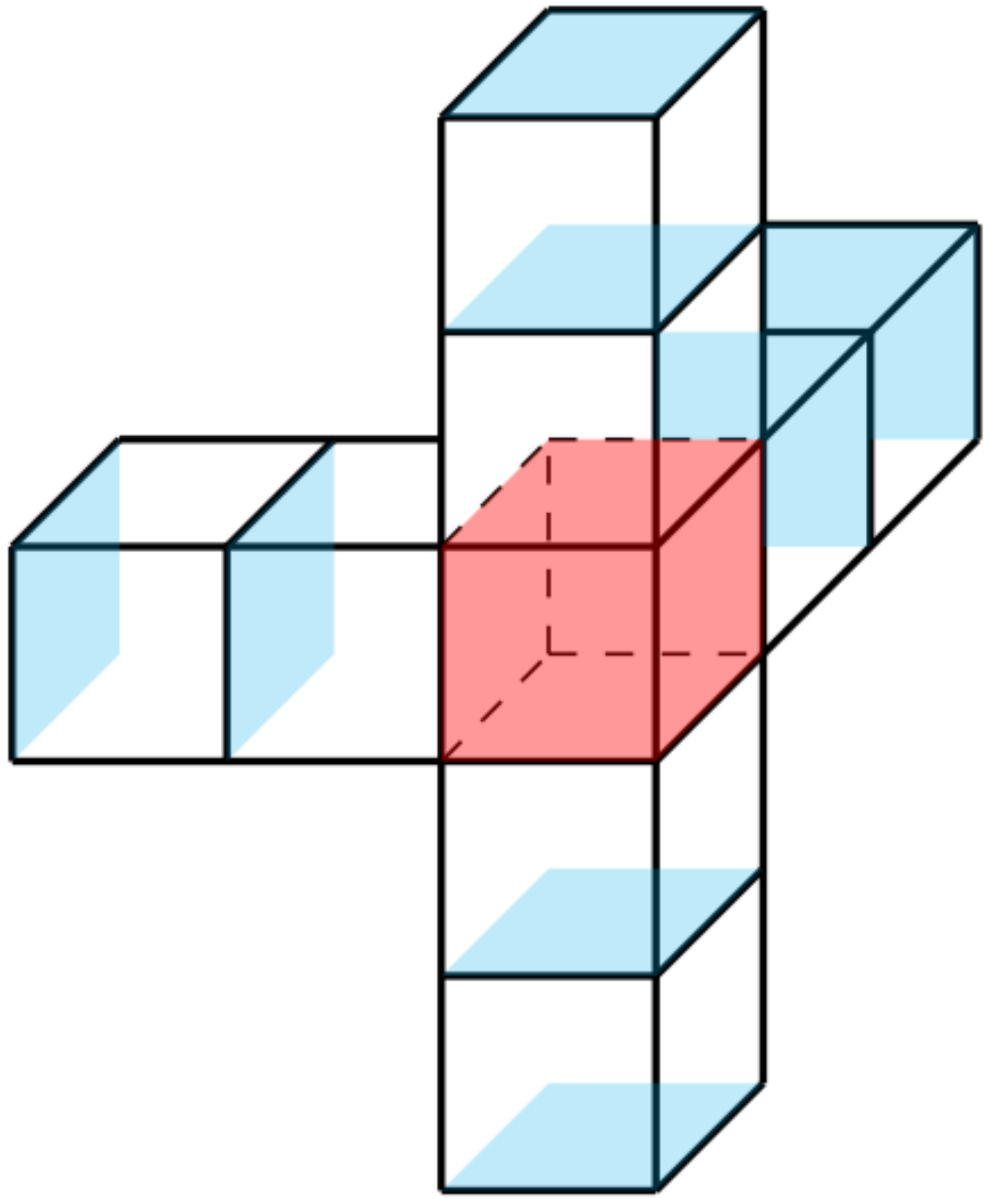} \label{fig:pattern_A}}
     \hspace{1cm}
     \subfloat[][Pattern B]{\includegraphics[width=3cm]{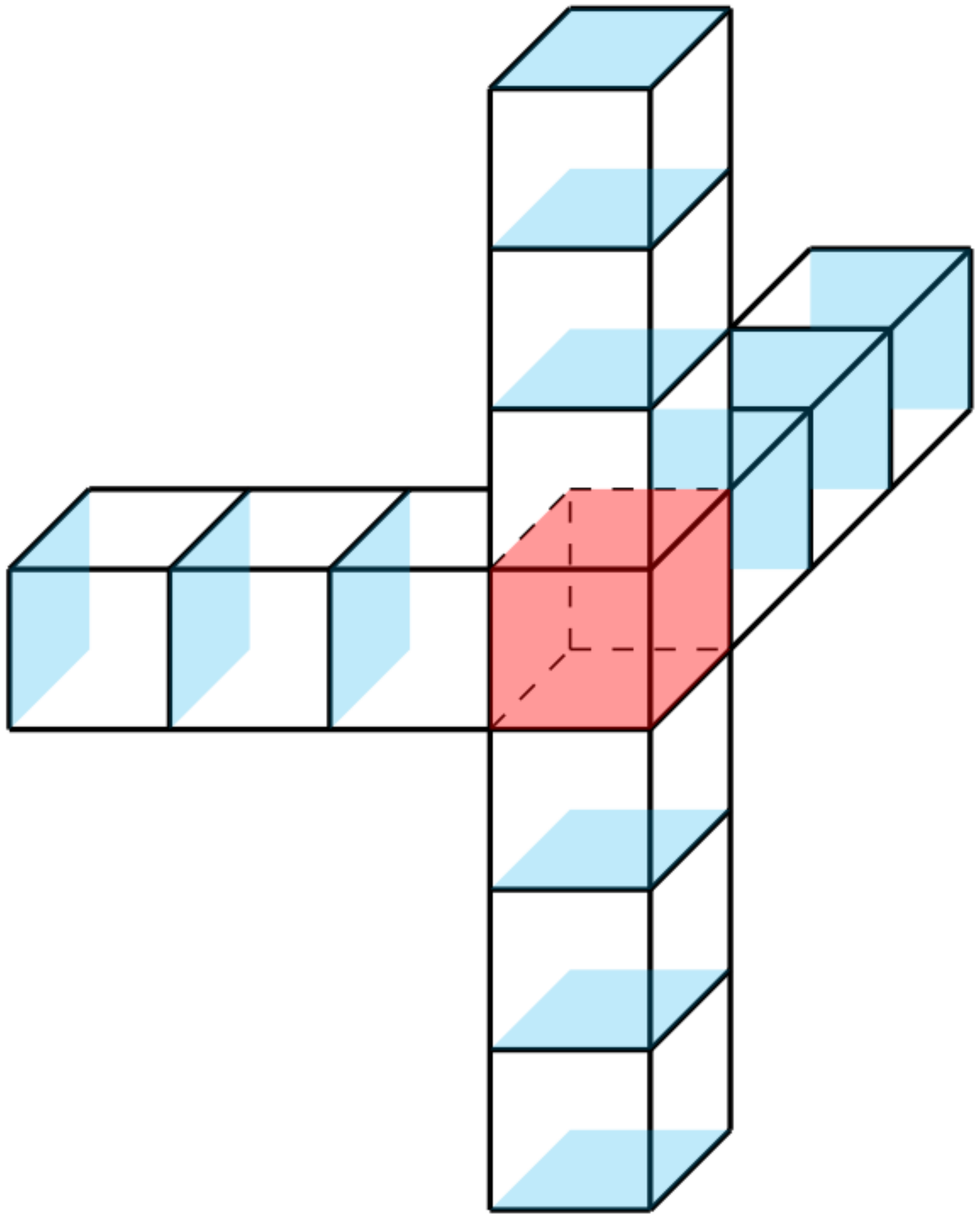} \label{fig:pattern_B}}
     \hspace{1cm}
     \subfloat[][Pattern C]{\includegraphics[width=1.5cm]{Figures/Pattern_unstruct.pdf} \label{fig:pattern_C}}
     \hspace{1cm}
     \subfloat[][Pattern D]{\includegraphics[width=3.7cm]{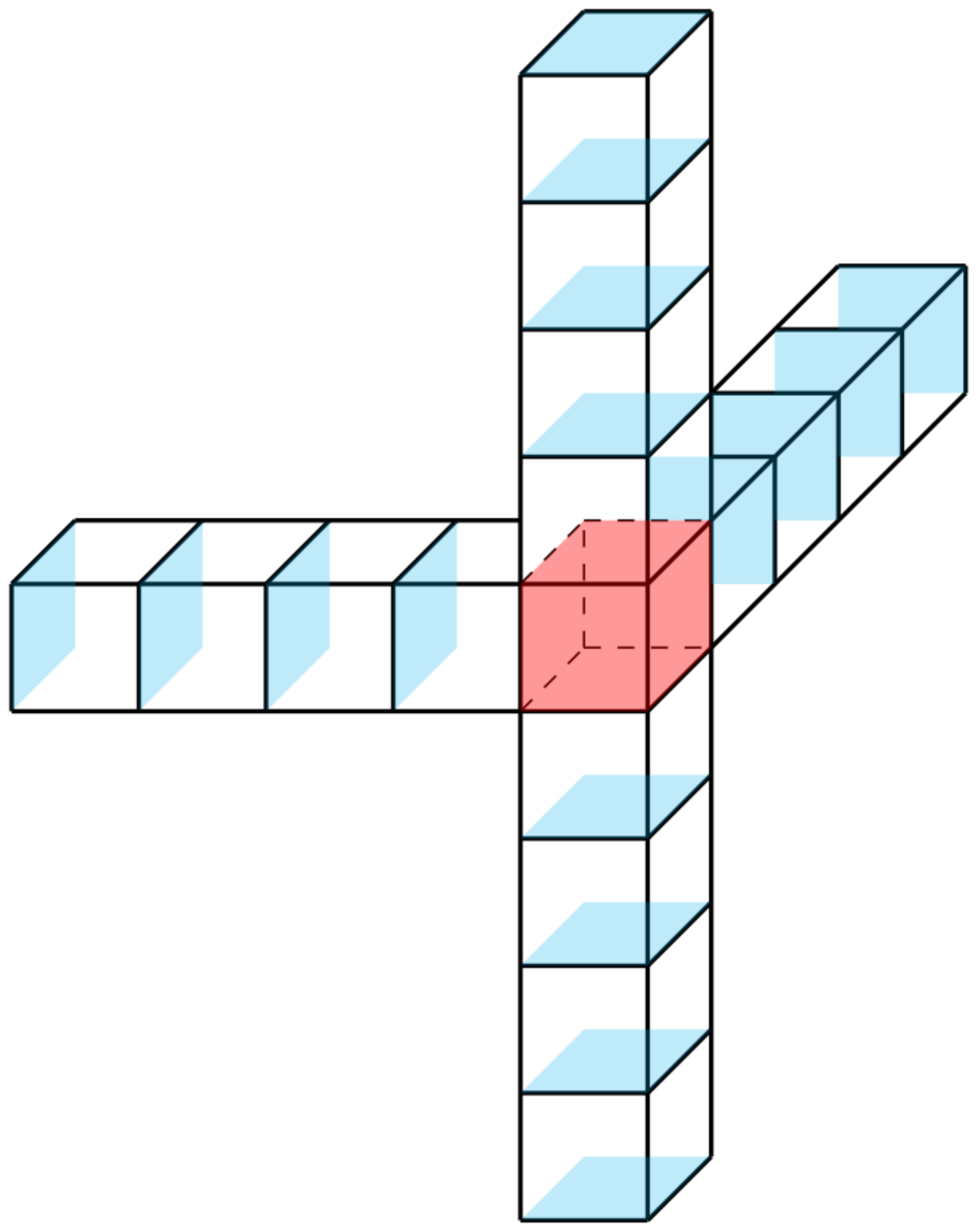} \label{fig:pattern_D}}
     \hspace{1cm}
     \subfloat[][Pattern E]{\includegraphics[width=2.2cm]{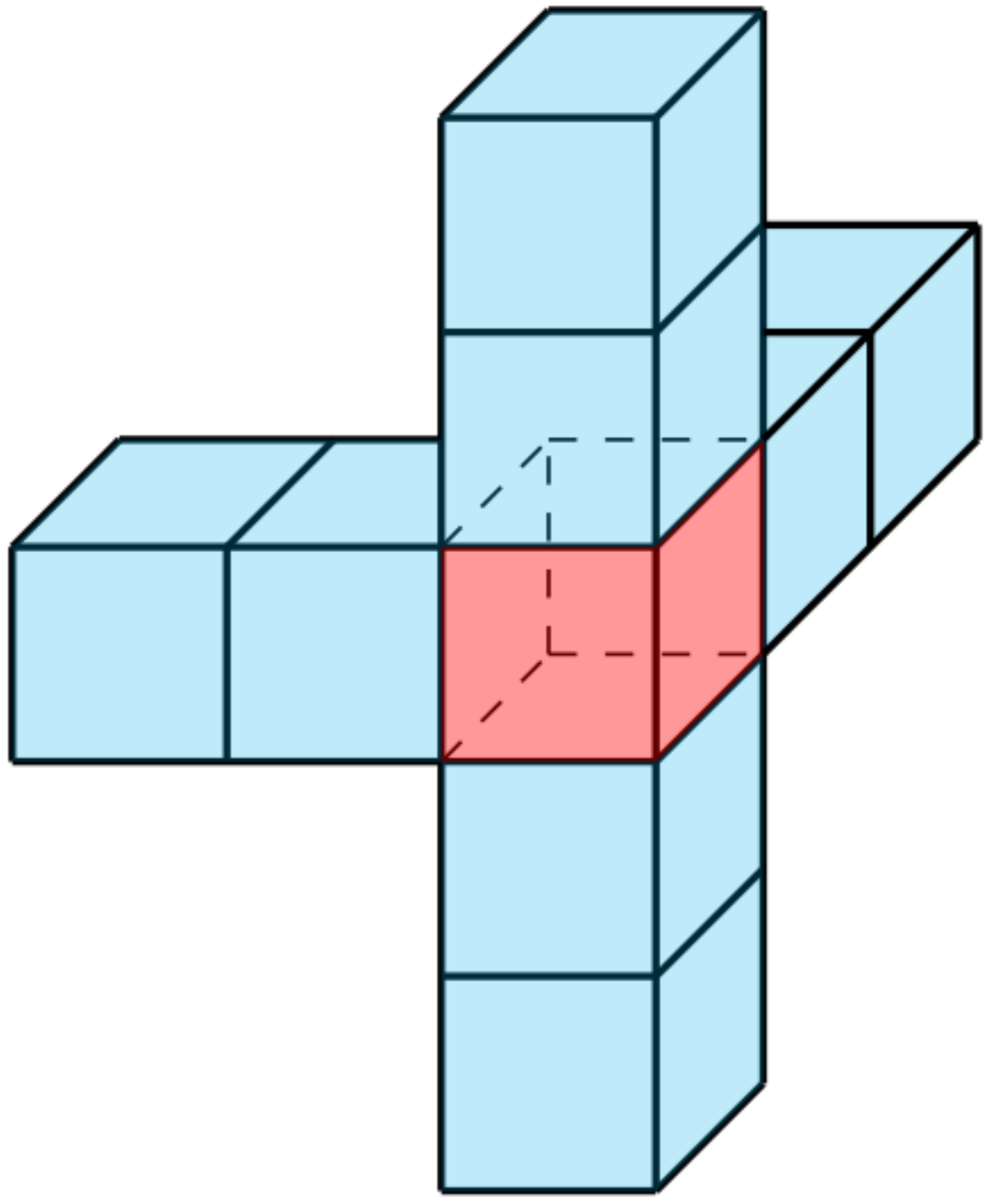} \label{fig:pattern_E}}
     \caption{Patterns for the static technique. The front and right elements have been removed to improve the readability.}
    \label{fig:static_patterns}
\end{figure}
%
First, we introduce a set of patches of cells associated with the face pattern connection used in the static EDFA variant.
The native pattern of the $m$-th column, previously introduced in Figure~\ref{fig:pattern_struct} for a structured grid, can be statically enlarged by considering the patches A, B and D (Figures~\ref{fig:pattern_A}, \ref{fig:pattern_B} and \ref{fig:pattern_D}), assuming that the flux is mainly oriented along the principal conductivity axes. 
By distinction, the connections of patterns C and E (Figures~\ref{fig:pattern_C} and \ref{fig:pattern_E}) assume the presence of significant fluxes through all directions, as it might be for instance expected in the case of a full conductivity tensor. Similarly, a significant permeability anisotropy could suggest privileging one direction with respect to the orthogonal ones.

The main results from the application of the EDFA preconditioner in its static variant to Test 1 (homogeneous and isotropic conductivity with a diagonal tensor) are reported in Table~\ref{tab:test_1_static}.
\begin{table}[tbp] \scriptsize
\caption{Test 1: Numerical performance of the static technique.}
\label{tab:test_1_static}
\centering
\vspace{0.2cm}
\begin{tabular}{ccccccccccccc}
\toprule
\# & Pat & Filt & $\tau_{\text{filt}}$  & $n_{\text{it}}$ & $t_{p_0}$ & $t_p$ & $t_s$ & $t_t$ & $\mu$\\
   &     &      &                       &                   &   [s]    &  [s]  &  [s]  & [s]   &     \\
\midrule
0  & Base & *    &    *    &  356 &  2.50 & 0.03 & 10.94 & 10.97 &  1.518 \\
1  &  A   & *    &    *    &  224 &  3.30 & 0.03 & 7.28  & 7.31  &  1.739 \\
2  &  B   & *    &    *    &  187 &  3.51 & 0.05 & 6.42  & 6.47  &  1.953 \\
3  &  C   & *    &    *    &  306 &  3.98 & 0.02 & 9.40  & 9.42  &  1.518 \\
4  &  D   & *    &    *    &  222 &  4.36 & 0.06 & 8.14  & 8.20  &  2.160 \\
5  &  E   & *    &    *    &  189 &  4.82 & 0.03 & 6.18  & 6.21  &  1.739 \\
6  &  E   & Post & 1.E-3   &  224 &  4.18 & 3.74 & 6.75 & 10.49 &  1.520 \\
7  &  E            & Pre  & 1.E-2   &  225 & 10.36 & 0.04 & 7.49 & 7.53  &  1.739 \\
\bottomrule
\end{tabular}
\end{table}
Run 0, denoted as \emph{base}, is taken as benchmark for the following considerations, since it refers to the performance of the EDFA preconditioner with the original $A_{p \pi}^T$ non-zero pattern. Expanding such initial pattern, by using the predefined connections A, B, C, D and E (Figures~\ref{fig:pattern_A} - \ref{fig:pattern_E}) is indeed beneficial, as observed in runs 1 to 5. Considering the reduction in the total CPU time $t_t$ per time step as evaluation criterion, the best results are achieved by patterns E, B and A. Specifically, the use of pattern E allows to reduce $n_{\text{it}}$ by a factor 1.88 and $t_t$ by 1.77, while increasing the preconditioner density by only 1.15. Runs 6 and 7 show the results obtained by applying pre- and post-filtration to pattern E. Pre- and post-filtration introduce a further sparsification of the approximate Schur complement, which is expected to decrease the application cost of the EDFA preconditioner at the cost of a slight increase in the iteration count. In this case, such a strategy does not appear to pay off, with the performance substantially getting back to Pattern A at a larger set-up cost. 

%
\begin{figure}
     \centering
     \subfloat[][]{\includegraphics[width=5.7cm]{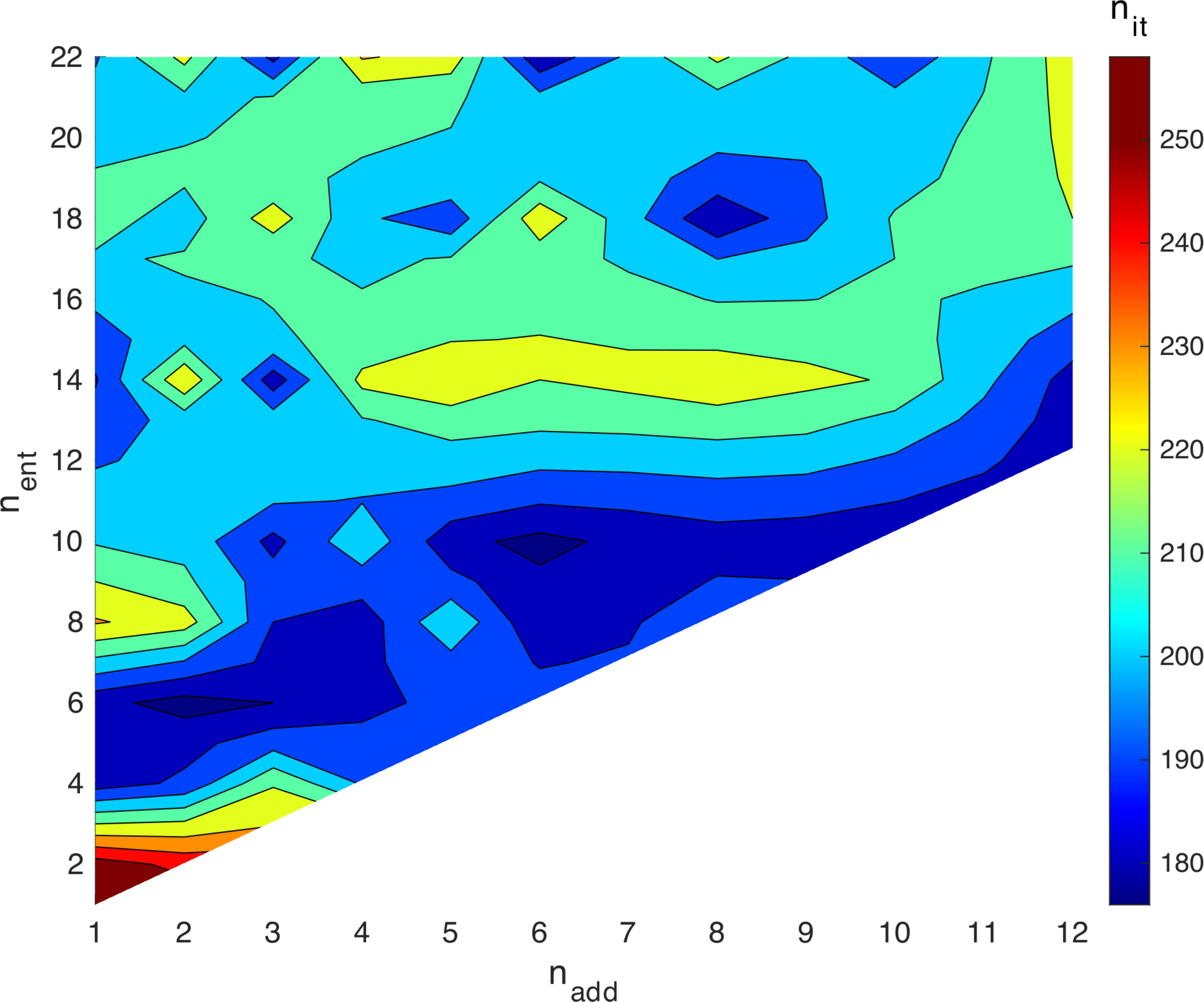} \label{fig:test_1_iter}}
     \hspace{0.2cm}
     \subfloat[][]{\includegraphics[width=5.7cm]{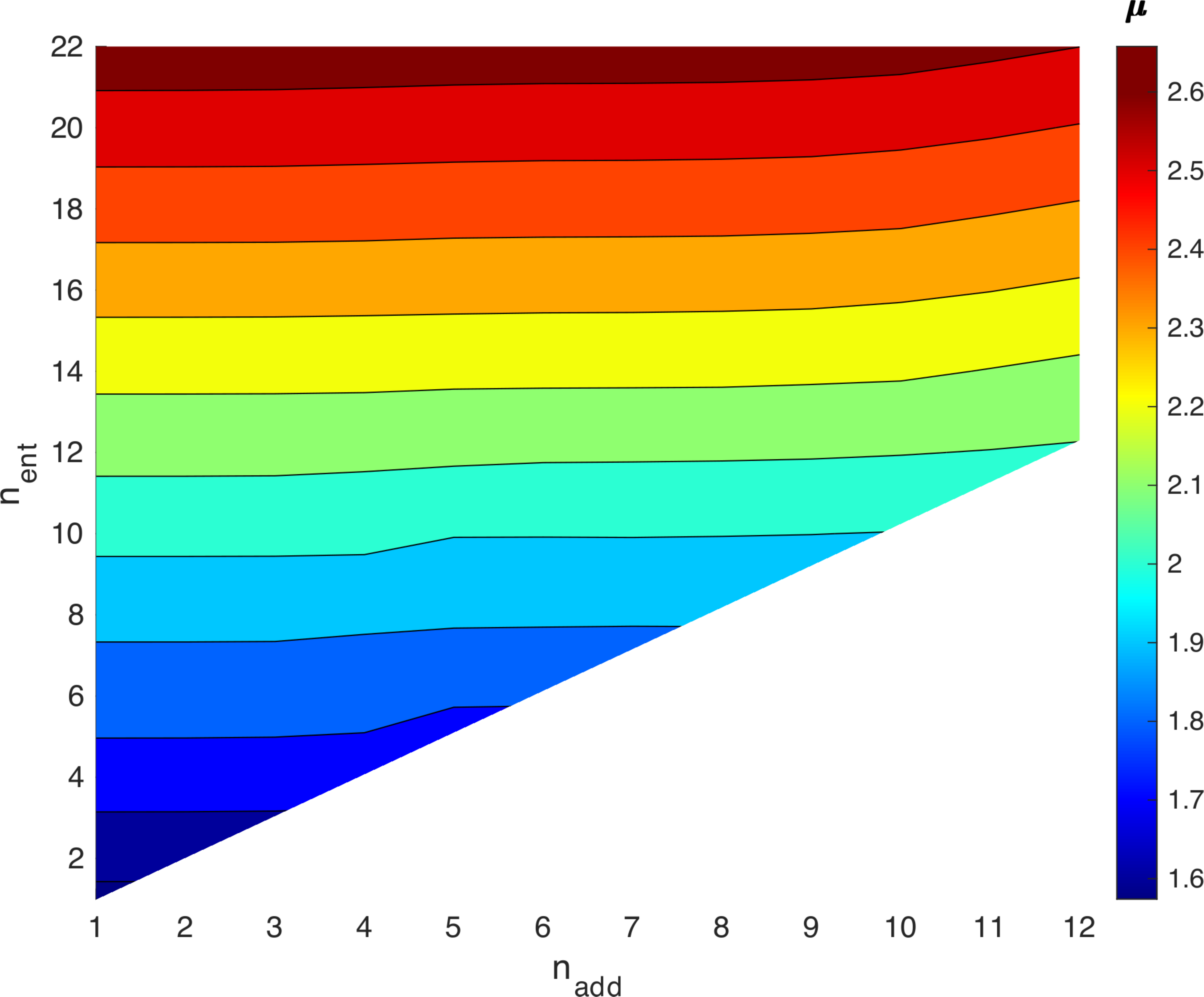} \label{fig:test_1_mu}}
     \vspace{0.3cm}
     \subfloat[][]{\includegraphics[width=5.7cm]{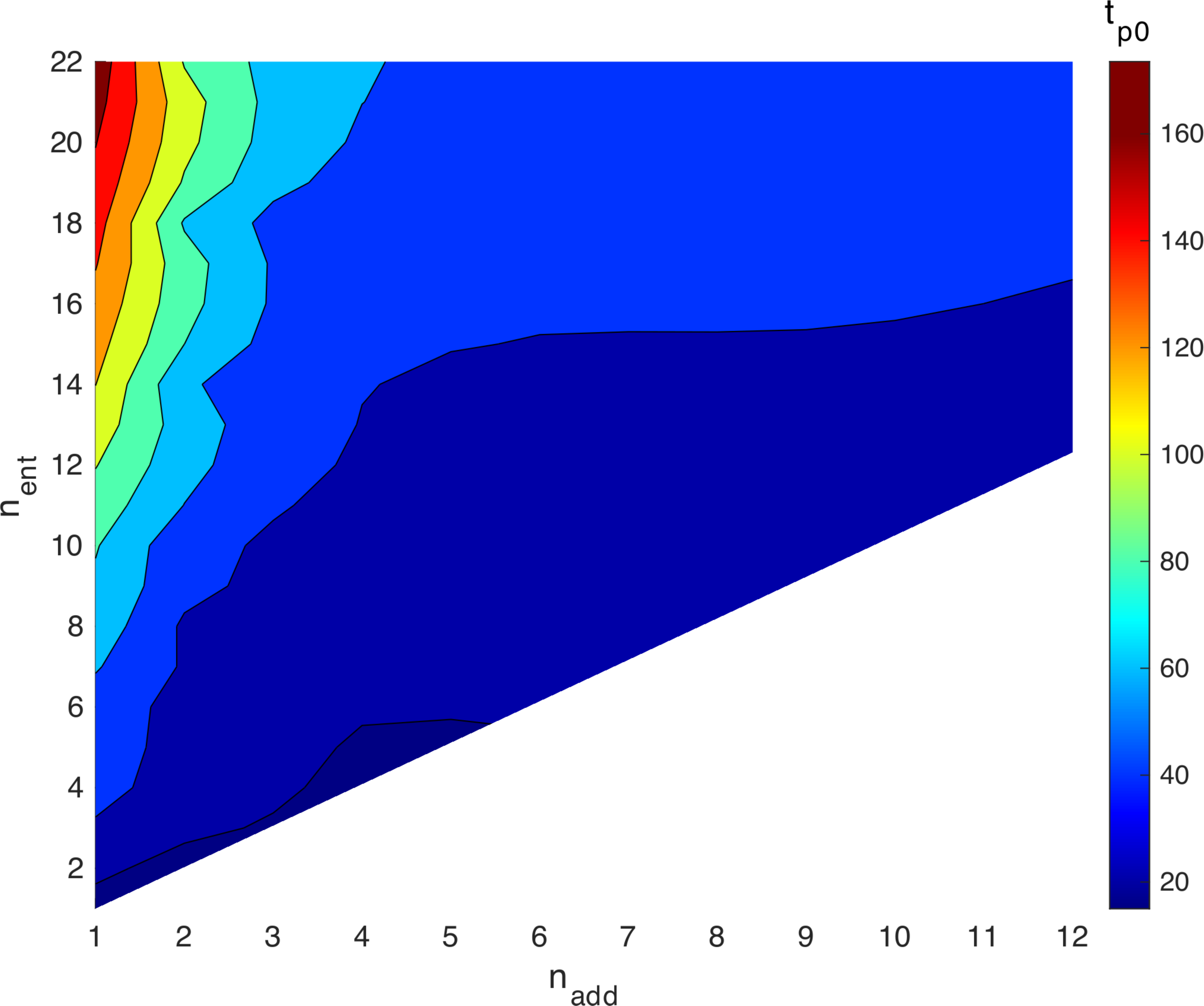} \label{fig:test_1_tp0}}
     \hspace{0.2cm}
     \subfloat[][]{\includegraphics[width=5.7cm]{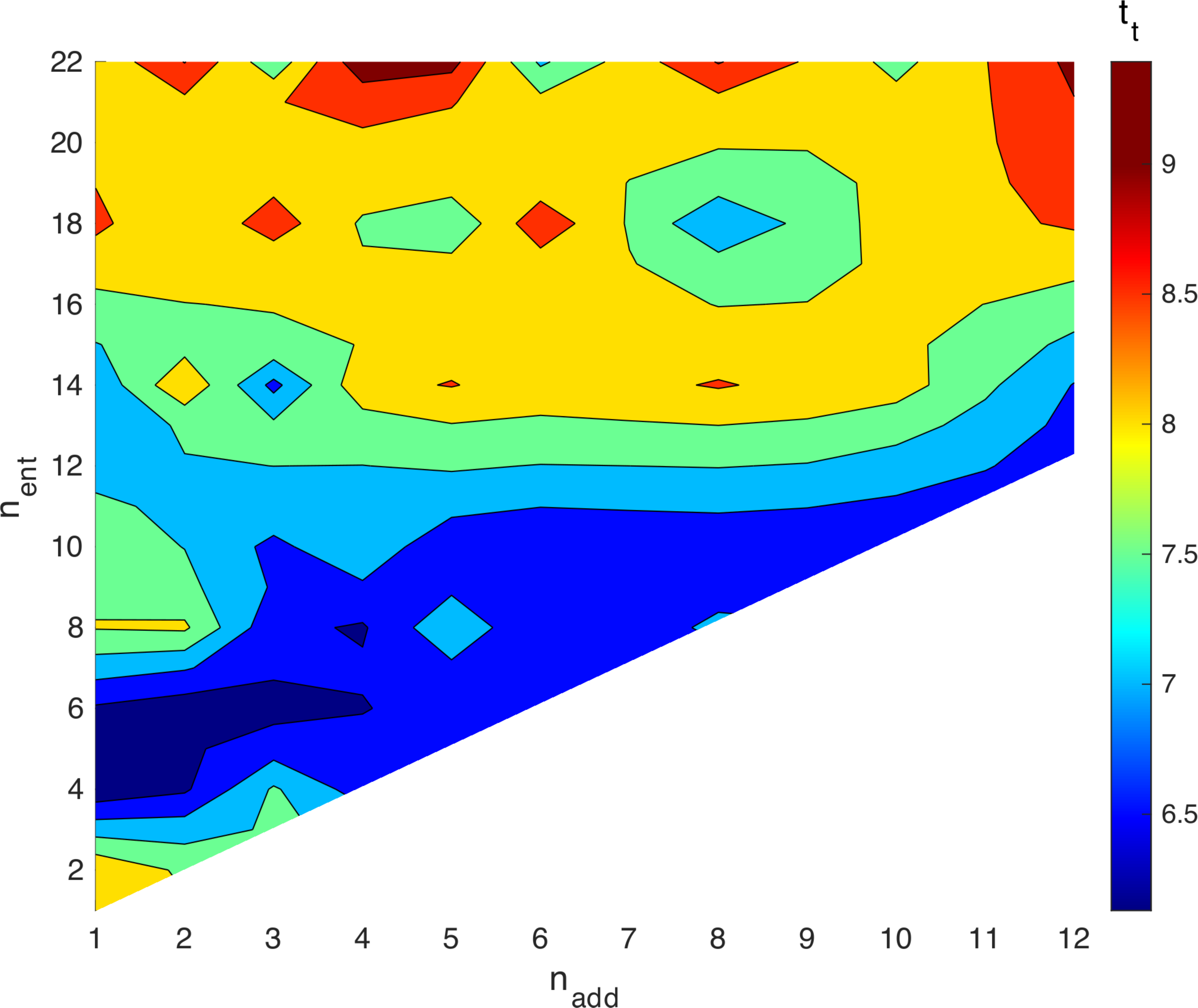} \label{fig:test_1_tt}}
     \caption{Test 1: Sensitivity analysis on the pair $n_{\text{ent}}$, $n_{\text{add}}$ in terms of iterations to converge (a), preconditioner density (b), time to compute the first stage of the preconditioner (c) and total time per time step (d).}
    \label{fig:test_1_dyn}
\end{figure}

As to the dynamic technique, Figure~\ref{fig:test_1_dyn} shows the results of a sensitivity analysis carried out on the two user-specified parameters $n_{\textup{ent}}$ and $n_{\textup{add}}$, governing the expansion of the initial pattern $Q^{(m)}_{(0)}=Q^{(m)}_{A_{p \pi}}$, versus the number of iterations to converge $n_{\text{it}}$, the preconditioner density $\mu$, the time to compute the pre-processing stage of the preconditioner $t_{p_0}$ and the total CPU time $t_t$ per time step. 
All the possible settings therein allow to accelerate the convergence compared to the base case of Table~\ref{tab:test_1_static}. The most interesting results are located in the blue to light-blue area in Figures~\ref{fig:test_1_iter} and \ref{fig:test_1_tt}, characterized by values of $n_{\textup{ent}}$ between 4 and 12. Such an interval was also confirmed by the outcome of the static technique, in particular runs 1 and 2, where the number of new entries added per column is 6 and 12, respectively, with a very similar overall performance. Figure~\ref{fig:test_1_mu} and \ref{fig:test_1_tp0} are self-explanatory; the higher $n_{\textup{ent}}$ the denser the preconditioner and the lower $n_{\textup{add}}$ the more iterations are needed to expand the native sparsity pattern, and thus the higher $t_{p_0}$. 
\begin{figure}
    \centering
    \includegraphics[width=12.5cm]{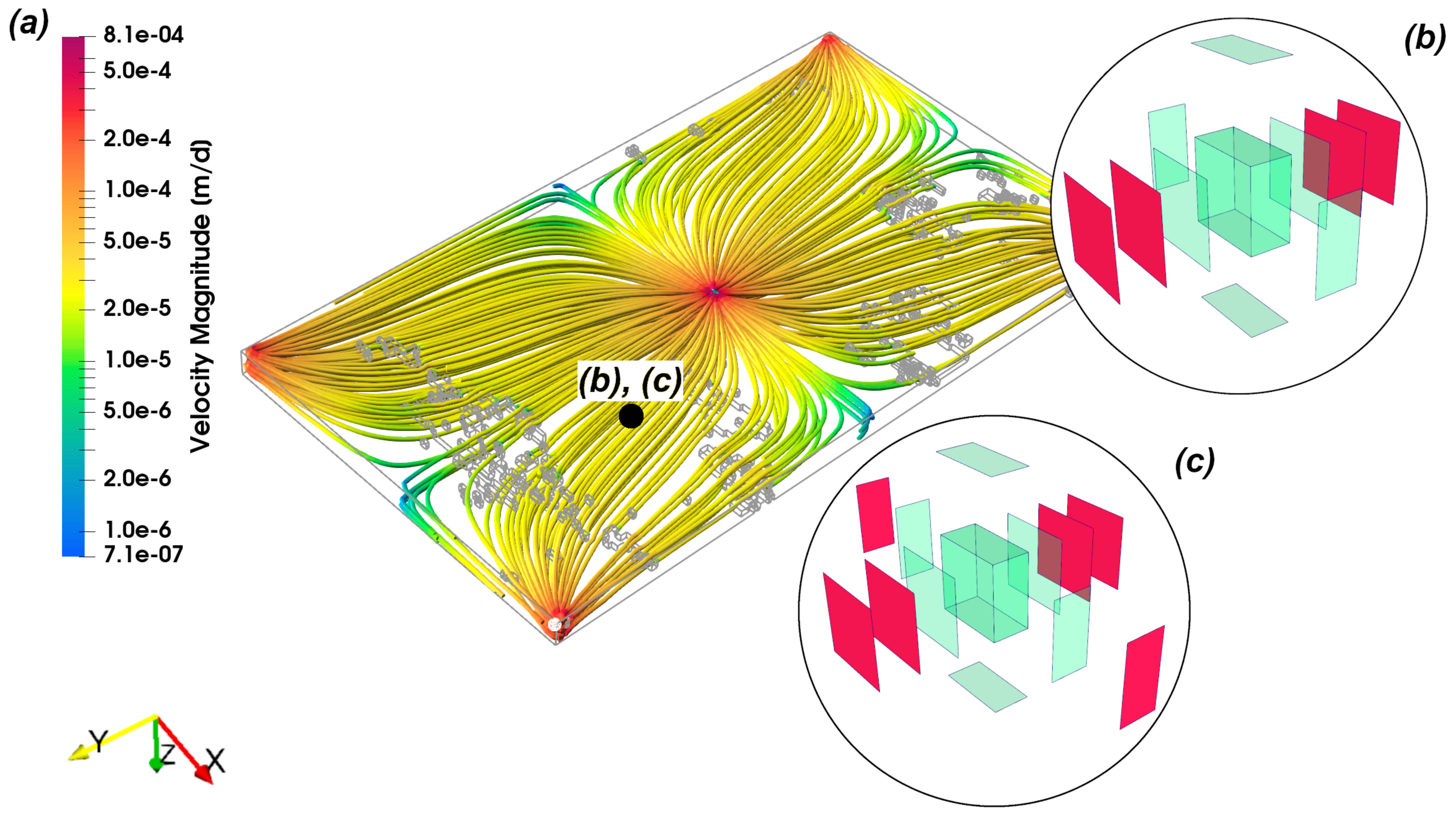}
    \caption{Test 1: Physical interpretation, based on the flux distribution (a), of two dynamic patterns, obtained with the settings (1,4) and (2,6) for the pair ($n_{\textup{add}},n_{\textup{ent}}$) (see Figure~\ref{fig:test_1_dyn}) in panels (b) and (c), respectively. The red faces represent the extension of the original $A_{p \pi}^T$ pattern in light-green. The resulting patterns, in subpanels (b) and (c), are uniform throughout the grid.}
    \label{fig:test_1_dyn_pat}
\end{figure}
It is interesting to provide the dynamic technique with a physical interpretation, so as to visually locate the position of the faces associated with the entries connected by the dynamically-formed optimal patterns. This analysis confirms the connection between the static and dynamic strategies, possibly inspiring the selection of better static connections from the the dynamic patterns. 
In particular, we focus on the pairs $(n_{\text{add}},n_{\text{ent}})$ equal to (1,4) and (2,6). 
Notice in Figure~\ref{fig:test_1_dyn_pat} that the newly added faces are primarily located following the main flow direction. 
This is consistent with the physical principle which the static technique relies on. In the scenario of Test 1, in fact, the water flow is essentially horizontal, since the wells penetrate the full thickness of the reservoir, with a principal component along the $y$ axis. These patterns are uniform throughout the grid.
\begin{table} \scriptsize
\caption{Test 1: Effect of the time step size on the behavior of EDFA preconditioner.}
\label{tab:test_1_delta_t}
\centering
\vspace{0.2cm}
\begin{tabular}{cccccccc}
\toprule
$\Delta t$ & $n_{\textup{it}}$ & $\chi_{\infty}$ & $t_{p_0}$ & $t_p$ & $t_s$ & $t_t$ & $\mu$ \\
$\text{[d]}$        &          &              & [s]       &  [s]  &  [s]  & [s]   &      \\
\midrule
0.01 &  3   & 0.004   & 35.53 & 0.04 & 0.11 & 0.15 & 1.749 \\
0.1  &  6   & 0.017   &   -   & 0.04 & 0.20 & 0.24 & 1.749 \\
1    &  21  & 0.099   &   -   & 0.04 & 0.78 & 0.82 & 1.749 \\
10   &  56  & 5.351   &   -   & 0.04 & 1.87 & 1.91 & 1.749 \\
100  &  127 & 66.061  &   -   & 0.04 & 4.23 & 4.27 & 1.749 \\
1000 &  199 & 571.157 &   -   & 0.04 & 6.66 & 6.70 & 1.749 \\
Steady & 185 &   -    &   -   & 0.04 & 6.09 & 6.13 & 1.749 \\
\bottomrule
\end{tabular}
\end{table}

Finally, the effect of the time step size on the EDFA preconditioner in a full transient simulation is assessed in Table~\ref{tab:test_1_delta_t}. The preconditioner is built by means of the dynamic technique, with the setting (1,4) for the pair $(n_{\textup{add}},n_{\textup{ent}})$. The range investigated spans the interval [0.01,1000] days, with the associated $\chi_{\infty}$ parameter up to almost 600. Notice how the number of iterations grows progressively as $\Delta t$ increases, where the steady state condition can be regarded approximately as an upper limit.

\subsection{Test 2: Dome reservoir with homogeneous and isotropic hydraulic properties}
\label{sec:test_2}
The evolution from a structured to an unstructured grid introduces new challenges for the design of efficient non-zero patterns due to the modification of the native stencils of blocks $A_{\pi \pi}$ and, most of all, $A_{p \pi}^T$, which is accompanied by the increase in the number of non-zeros of those two blocks, as shown in Table~\ref{tab:test_setup}. Specifically, $A_{p \pi}^T$ is 1.9 times denser than with the planar mesh. 
In fact, the face-to-element connection building the non-zero pattern of the typical column of $A_{p \pi}^T$ moves from  Figure~\ref{fig:pattern_struct} to  Figure~\ref{fig:pattern_unstruct}.

This modification has a very relevant effect on the numerical performance of the base case (see Table~\ref{tab:test_2_static}), which does not converge after 2,000 iterations. Enlarging the face-to-element connections with patterns A, B and D (runs 1, 2 and 4), however, improves very rapidly the quality of $\widetilde{S}$. 
By comparing Figure~\ref{fig:pattern_unstruct} with~\ref{fig:pattern_A}, \ref{fig:pattern_B} and \ref{fig:pattern_D}, the adoption of patterns A, B and D consists in a simultaneous expansion and contraction of $A_{p \pi}^T$'s column non-zero pattern, since at most 24 entries are discarded and others are added in a variable number. It is a sort of implicit moderate expansion accompanied by a significant filtration. Applying pre- or post-filtration when adopting pattern B, which means in practice sparsifying for the second time the relevant pattern, is not beneficial as proved in runs 6 and 7. Notice that the overall preconditioner density is generally higher than in Test 1 (see Table~\ref{tab:test_1_static} for a comparison).
\begin{table}[tbp] \scriptsize
\caption{Test 2: Numerical performance of the static technique. The expression NC means that the solver did not converge within 2,000 iterations.}
\label{tab:test_2_static}
\centering
\vspace{0.2cm}
\begin{tabular}{ccccccccccccc}
\toprule
\# & Pat & Filt & $\tau_{\textup{filt}}$  & $n_{\textup{it}}$ & $t_{p_0}$ & $t_p$ & $t_s$ & $t_t$ & $\mu$\\
   &     &      &                       &                   &   [s]    &  [s]  &  [s]  & [s]   &     \\
\midrule
0  & Base & *    &    *    &  NC  &  5.37 & 0.29 &  *    &   *   &  2.112 \\
1  &  A   & *    &    *    &  896 &  4.10 & 0.30 & 58.91 & 59.21 &  2.177 \\
2  &  B   & *    &    *    &  118 &  4.28 & 0.62 & 8.64  & 9.26  &  2.734 \\
3  &  C   & *    &    *    &  NC  &  6.00 & 0.29 &  *    &   *   &  2.112 \\
4  &  D   & *    &    *    &  136 &  4.48 & 1.04 & 11.24 & 12.28 &  3.372 \\
5  &  E   & *    &    *    &  NC  &  7.78 & 1.36 &  *    &  *    &  3.707 \\
6  &  B   & Post & 1.E-7  &  160 &  4.03 & 10.59 & 10.45 & 21.04 &  2.071 \\
7  &  B            & Pre  & 1.E-3   &  117 &  7.51 & 0.66 & 8.77 & 9.43  &  2.734 \\
\bottomrule
\end{tabular}
\end{table}
%
%
%
%
\begin{figure}[tb]
    \centering
    \subfloat[][]{
    \begin{tikzpicture}[scale=0.75]
        \begin{semilogxaxis}
            [ ylabel={$n_{\textup{it}}$},
              xlabel={$\tau_{\textup{filt}}$},
              grid = major,
              legend entries = {$n_{\textup{ent}}=10$,$n_{\textup{ent}}=12$,$n_{\textup{ent}}=14$,$n_{\textup{ent}}=18$},
              legend pos=outer north east,
              ytick = {0,250,500,750,1000,1250,1500,1750,2000},
              xmin = 1.e-5, xmax = 1,
              ymin = 0, ymax = 2000,
              ]
            \addplot table {Graphs/Iter_post_10.txt};
            \addplot table {Graphs/Iter_post_12.txt};
            \addplot table {Graphs/Iter_post_14.txt};
            \addplot table {Graphs/Iter_post_18.txt};
        \end{semilogxaxis}
    \end{tikzpicture}
    \label{fig:test_2_post_iter}
    }
    \subfloat[][]{
    \begin{tikzpicture}[scale=0.75]
        \begin{semilogxaxis}
            [ ylabel={$t_t$ [s]},
              xlabel={$\tau_{\textup{filt}}$},
              grid = major,
              legend entries = {$n_{\textup{ent}}=10$,$n_{\textup{ent}}=12$,$n_{\textup{ent}}=14$,$n_{\textup{ent}}=18$},
              legend pos=outer north east,
              ytick = {0,25,50,75,100,125,150,175,200},
              xmin = 1.e-5, xmax = 1,
              ymin = 0, ymax = 200,
              ]
            \addplot table {Graphs/T_t_post_10.txt};
            \addplot table {Graphs/T_t_post_12.txt};
            \addplot table {Graphs/T_t_post_14.txt};
            \addplot table {Graphs/T_t_post_18.txt};
        \end{semilogxaxis}
    \end{tikzpicture}
    \label{fig:test_2_post_tt}
    }
    \vspace{0.2cm}
    \subfloat[][]{
    \begin{tikzpicture}[scale=0.75]
        \begin{semilogxaxis}
            [ ylabel={$n_{\textup{it}}$},
              xlabel={$\tau_{\textup{filt}}$},
              grid = major,
              legend entries = {$n_{\textup{ent}}=10$,$n_{\textup{ent}}=12$,$n_{\textup{ent}}=14$,$n_{\textup{ent}}=18$},
              legend pos=outer north east,
              ytick = {0,250,500,750,1000,1250,1500,1750,2000},
              xmin = 1.e-5, xmax = 1,
              ymin = 0, ymax = 2000,
              ]
            \addplot table {Graphs/Iter_post_D_10.txt};
            \addplot table {Graphs/Iter_post_D_12.txt};
            \addplot table {Graphs/Iter_post_D_14.txt};
            \addplot table {Graphs/Iter_post_D_18.txt};
        \end{semilogxaxis}
    \end{tikzpicture}
    \label{fig:test_2_post_D_iter}
    }
    \subfloat[][]{
    \begin{tikzpicture}[scale=0.75]
        \begin{semilogxaxis}
            [ ylabel={$t_t$ [s]},
              xlabel={$\tau_{\textup{filt}}$},
              grid = major,
              legend entries = {$n_{\textup{ent}}=10$,$n_{\textup{ent}}=12$,$n_{\textup{ent}}=14$,$n_{\textup{ent}}=18$},
              legend pos=outer north east,
              xmin = 1.e-5, xmax = 1,
              ymin = 0, ymax = 200,
              ytick = {0,25,50,75,100,125,150,175,200},
              ]
            \addplot table {Graphs/T_t_post_D_10.txt};
            \addplot table {Graphs/T_t_post_D_12.txt};
            \addplot table {Graphs/T_t_post_D_14.txt};
            \addplot table {Graphs/T_t_post_D_18.txt};
        \end{semilogxaxis}
    \end{tikzpicture}
    \label{fig:test_2_post_D_tt}
    }
    \caption{Test 2: Sensitivity analysis on the post-filtration tolerance $\tau_{\textup{filt}}$ for different values of $n_{\textup{ent}}$. $n_{\textup{add}}$ is kept constant and equal to 4. Post-filtration is applied to $\widetilde{S}$, (a,b), and $\widetilde{H}$, (c,d), and the results are expressed in terms of number of iterations to converge, $n_{\textup{it}}$, (a,c) and total solution time per time step, $t_t$, (b,d). The maximum number of iterations of Bi-CGStab is set equal to 2,000. The best result, i.e., $n_{\textup{it}}=98$ and $t_t$=6.00 s, is obtained by setting $n_{\textup{ent}}$=12 and performing post-filtration on $\widetilde{H}$ with $\tau_{\textup{filt}}=$1.E-3 (c-d).} 
    \label{fig:test_2_post}
\end{figure}
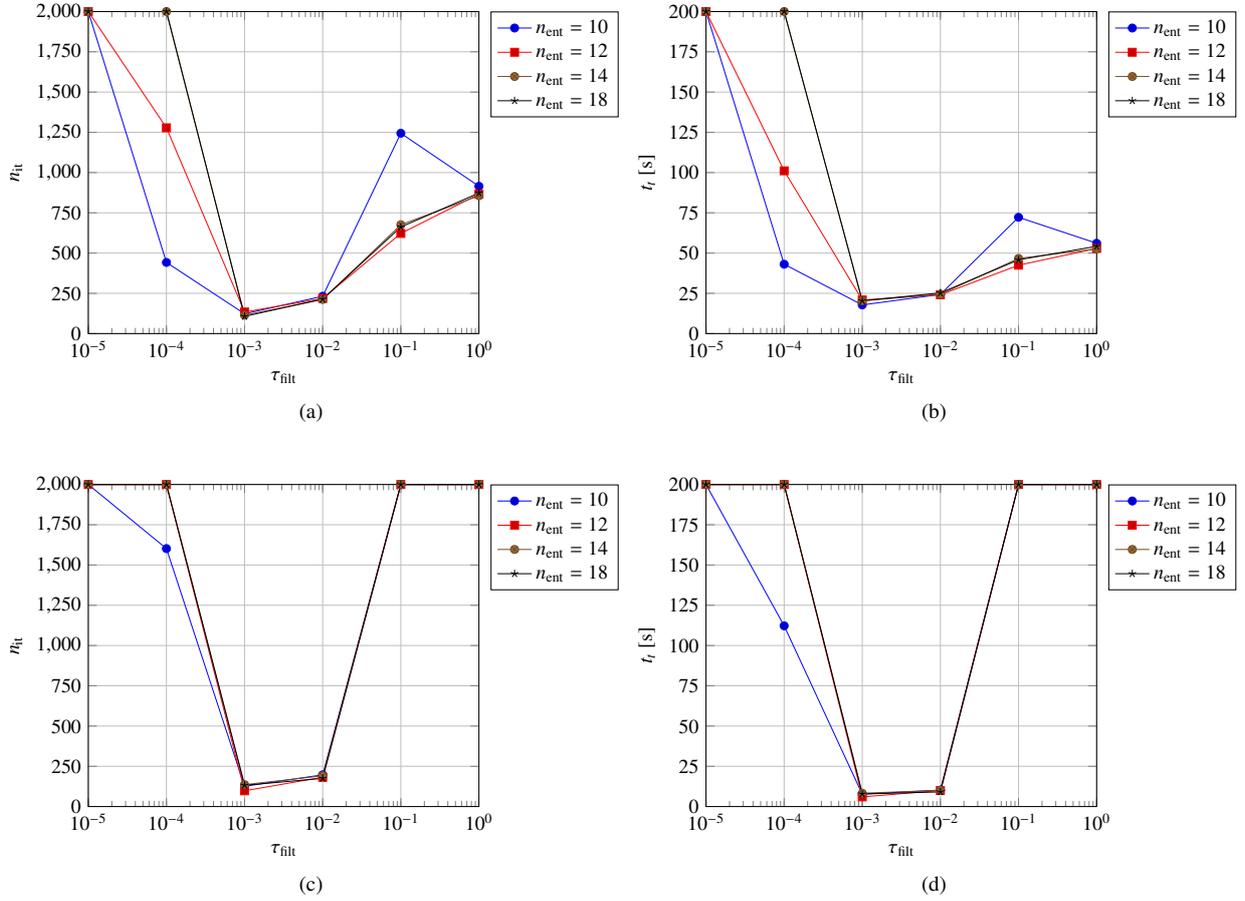

As to the dynamic technique, the same strategy, consisting of the expansion and contraction of the original $A_{p \pi}^T$ pattern, has been followed. 
Figure~\ref{fig:test_2_post} reports the outcomes of a sensitivity analysis on the post-filtration tolerance for different values of $n_{\textup{ent}}$, ranging from 10 to 18 with $n_{\textup{add}}=4$. The outcomes are expressed in terms of iteration count to converge (\ref{fig:test_2_post_iter}) and total solution time per time step (\ref{fig:test_2_post_tt}). The best performance is achieved for $\tau_{\textup{filt}}=$1.E-3, where the number of iterations to converge is similar to that obtained with pattern B in Table~\ref{tab:test_2_static}. 
An effective way to limit the post-filtration cost consists of applying it to $\widetilde{H}$ as a pre-processing effort. This strategy is successfully investigated in Figures~\ref{fig:test_2_post_D_iter} and \ref{fig:test_2_post_D_tt}. Notice that, while the number of iterations remains approximately the same or even improved, the CPU time per time step is halved, thus making the dynamically-formed preconditioner competitive to the best statically-derived one (run 2 in Table~\ref{tab:test_2_static}).

\subsection{Test 3: Plain reservoir with heterogeneous and anisotropic hydraulic conductivity}
\label{sec:test_3}
Introducing a heterogeneous and anisotropic conductivity field in the planar reservoir application worsens the conditioning properties of the associated problem, as shown by the increase in the number of iterations of the base case (Table~\ref{tab:test_3_static}) with respect to Test 1 (Table~\ref{tab:test_1_static}). Like in Test 1, patterns A and E (runs 1, 5) provide improved results that can slightly benefit from further sparsification of the approximate Schur complement. Specifically, pre- and post-filtration are characterized barely by the same performance (runs 6 to 9) with pre-filtration a little more efficient in this application.
%
\begin{table}[tb] \scriptsize
\caption{Test 3: Numerical performance of the static technique.}
\label{tab:test_3_static}
\centering
\vspace{0.2cm}
\begin{tabular}{ccccccccccccc}
\toprule
\# & Pat & Filt & $\tau_{\textup{filt}}$  & $n_{\textup{it}}$ & $t_{p_0}$ & $t_p$ & $t_s$ & $t_t$ & $\mu$\\
   &     &      &                       &                   &   [s]    &  [s]  &  [s]  & [s]   &     \\
\midrule
0  & Base & *    &    *    &  395 &  2.67 & 0.03 & 11.57 & 11.60 &  1.518 \\
1  &  A   & *    &    *    &  279 &  3.28 & 0.03 & 8.75  & 8.78  &  1.739 \\
2  &  B   & *    &    *    &  315 &  3.40 & 0.05 & 10.43 & 10.48 &  1.953 \\
3  &  C   & *    &    *    &  445 &  4.69 & 0.02 & 13.13 & 13.15 &  1.518 \\
4  &  D   & *    &    *    &  303 &  3.80 & 0.06 & 10.59 & 10.65 &  2.160 \\
5  &  E   & *    &    *    &  290 &  4.81 & 0.04 & 9.19  & 9.23  &  1.739 \\
6  &  A   & Post (on $\widetilde{H}$) & 1.E-7  &  267 &  7.08 & 0.03 & 8.78 & 8.81 &  1.730 \\
7  &  E   & Post (on $\widetilde{H}$)          & 1.E-8  &  267 &  8.12 & 0.04 & 8.77 & 8.81 &  1.733 \\
8  &  A    & Pre  & 1.E-3   &  266 &  5.81 & 0.04 & 8.50 & 8.54  &  1.729 \\
9  &  E    & Pre  & 1.E-4   &  260 &  10.36 & 0.04 & 7.78 & 7.81  &  1.731 \\
\bottomrule
\end{tabular}
\end{table}
\begin{figure}[tb]
     \centering
     \subfloat[][]{\includegraphics[width=5.7cm]{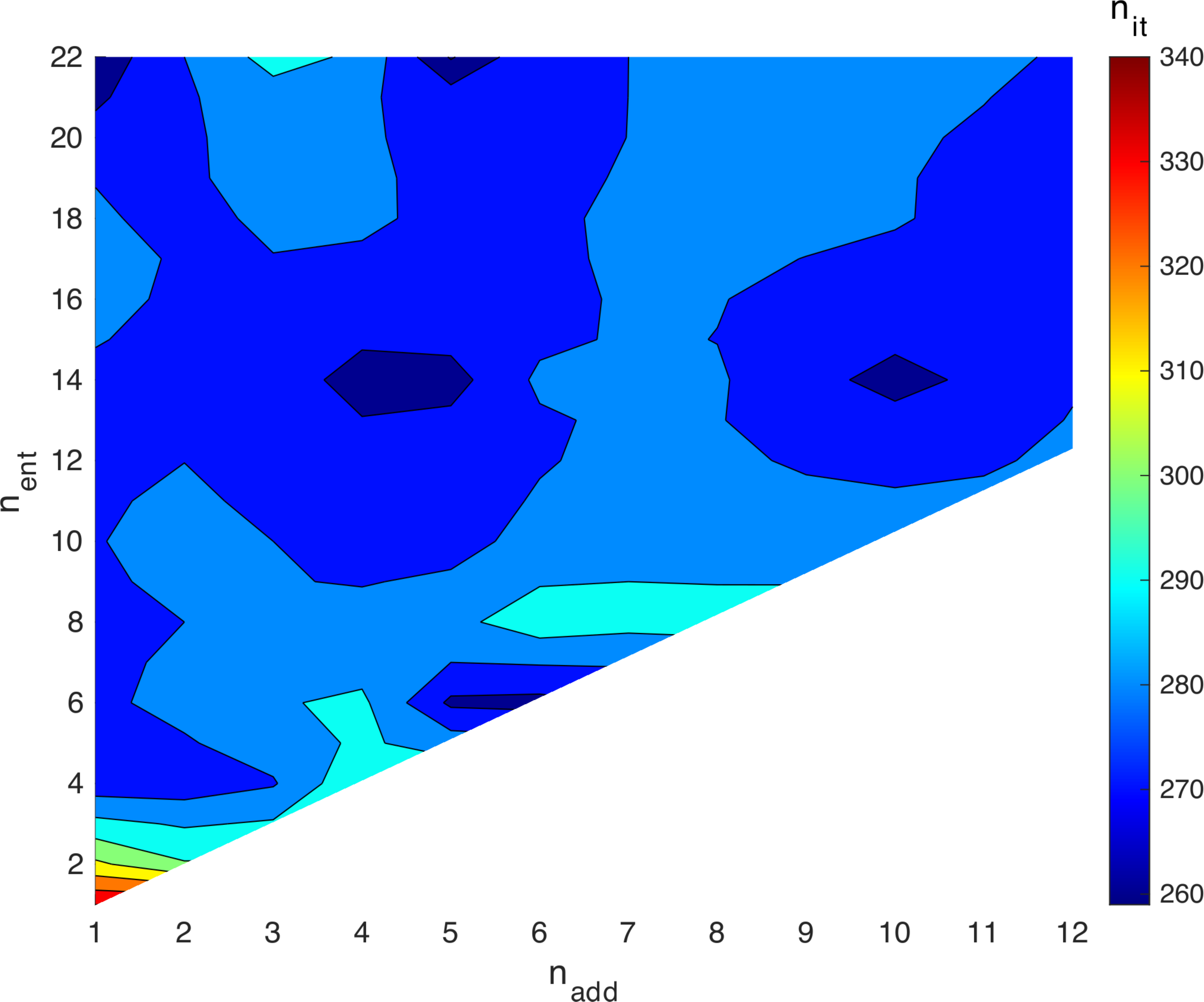} \label{fig:test_3_iter}}
     \hspace{0.2cm}
     \subfloat[][]{\includegraphics[width=5.7cm]{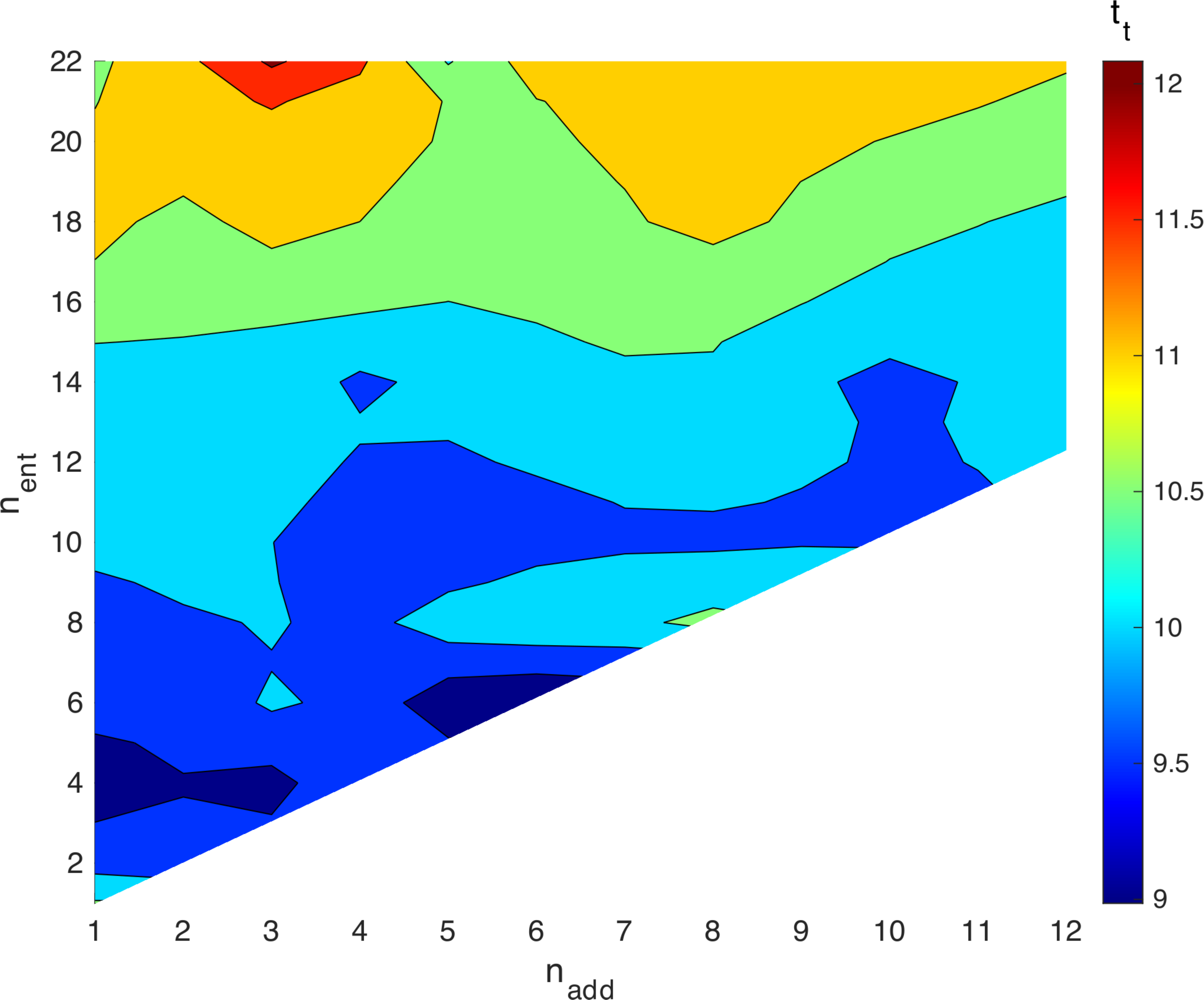} \label{fig:test_3_tt}}
     \vspace{0.3cm}
     \subfloat[][]{\includegraphics[width=5.7cm]{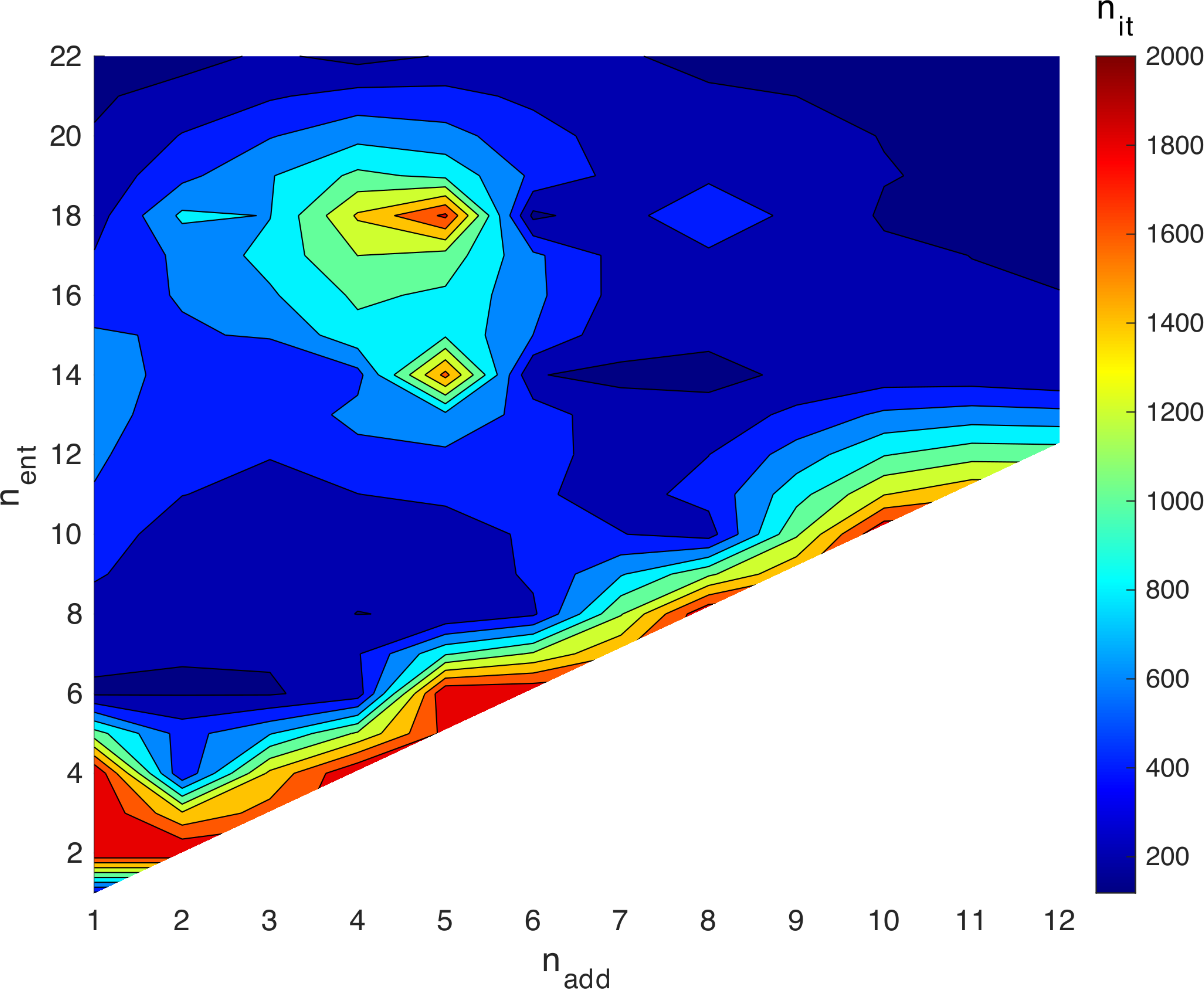} \label{fig:test_4_iter}}
     \hspace{0.2cm}
     \subfloat[][]{\includegraphics[width=5.7cm]{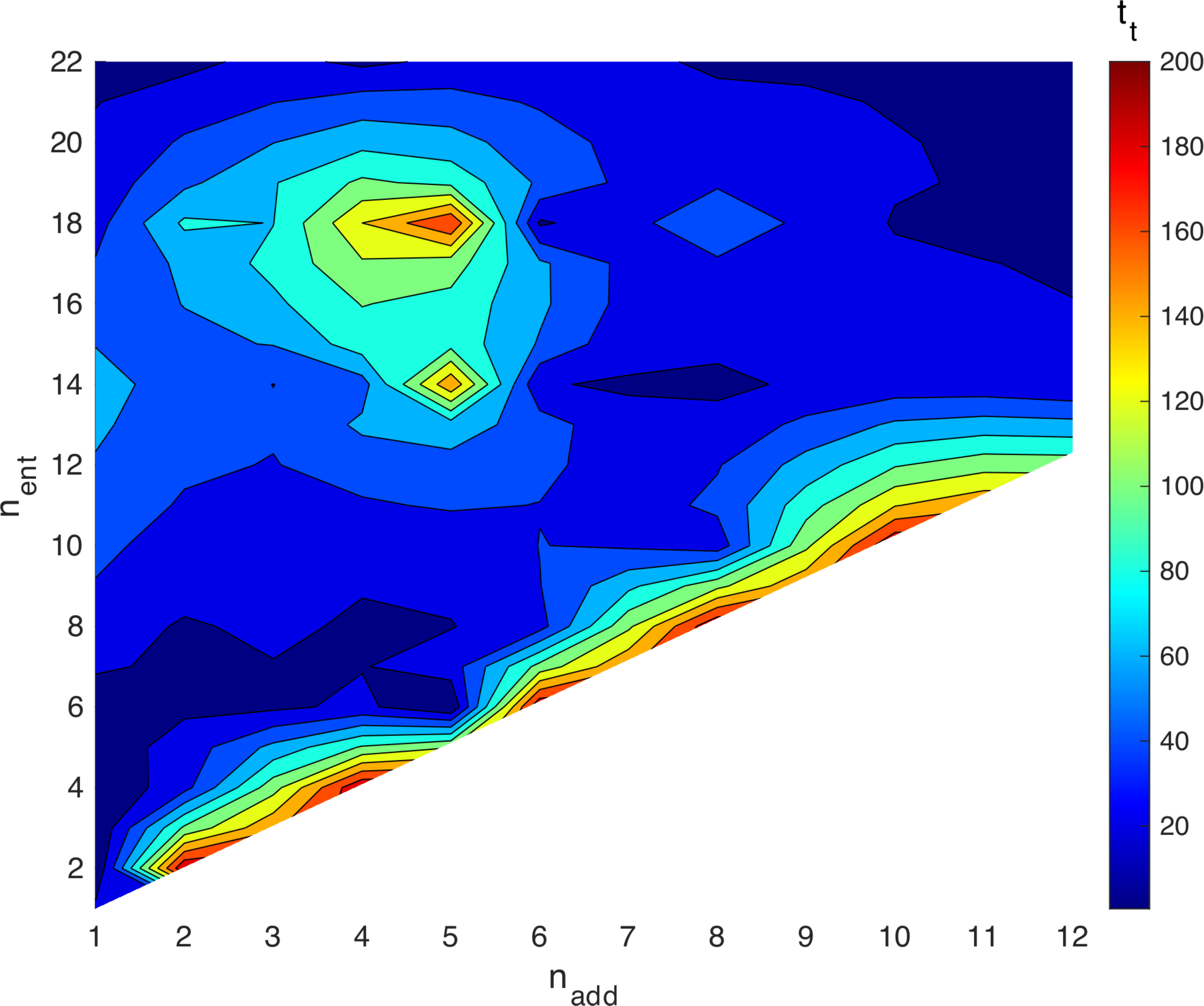} \label{fig:test_4_tt}}
     \caption{Test 3 and 4: Sensitivity analysis on the pair $(n_{\text{add}},n_{\text{ent}})$ for Test 3 (a), (b) and Test 4 (c), (d). The outcome is expressed in terms of iterations to converge (a), (c) and total time per time step (b), (d). 
     }
    \label{fig:test_3_4_dyn}
\end{figure}

A sensitivity analysis on the dynamic parameters $n_{\textup{ent}}$ and $n_{\textup{add}}$ is shown in Figures~\ref{fig:test_3_iter} and \ref{fig:test_3_tt} in terms of number of iterations to converge $n_{\text{it}}$ and total solution time per time step $t_t$, respectively. All the analyzed combinations succeed in accelerating the convergence with respect to the base case (run 0 in Table~\ref{tab:test_3_static}). The most efficient settings are located in the bottom left portion of graph~\ref{fig:test_3_tt} in the interval $2 \leq n_{\textup{ent}} \leq 8$, where $t_t$ is minimized. Specifically, the fastest convergence is achieved with the setting (6,6) for the pair $(n_{\textup{add}},n_{\textup{ent}})$, where $n_{it}=267$ and $t_t=8.98$ s. In the same figure notice how $t_t$ grows with $n_{\textup{ent}}$, depending only partially on $n_{\textup{add}}$.
\begin{table}[tb] \scriptsize
\caption{Test 3: CPU times and memory requirements for Matlab $\texttt{backslash}$ operator and Bi-CGStab preconditioned with an ILU($\tau$) approximation of the whole matrix $\mathcal{A}$, where $\tau=0.001$. Due to memory limitation the test with the direct solver has been carried out on a different platform equipped with an Intel\textregistered Xeon\texttrademark CPU E5-1620 v4 at 3.5 GHz with 64 GB of RAM.}
\label{tab:test_3_bench}
\centering
\vspace{0.2cm}
\begin{tabular}{cccccccc}
\toprule
Method  & $n_{\textup{it}}$ & $t_{p_0}$ & $t_p$ & $t_s$ & $t_t$ & Memory peak & $\mu$ \\
  &  &   [s]    &  [s]  &  [s]  & [s] & [GB] &  \\
\midrule
Matlab \textbackslash{}   & - & - & - & 205.01 & 205.01 & 23.68 & - \\
ILU($\tau$)  & 113 & - & 248.25 & 4.11 & 252.36 & - & 2.730 \\
\bottomrule
\end{tabular}
\end{table}

Finally, a direct solver, such as Matlab $\texttt{backslash}$ operator, and Bi-CGStab accelerated by a global preconditioner, such as a threshold-based ILU($\tau$) of $\mathcal{A}$ as available in Matlab,  are used to benchmark the performance of the proposed EDFA preconditoner in the steady state case. The relevant outcomes are conveyed in Table~\ref{tab:test_3_bench}, where the memory peak reached during the solving phase with Matlab $\texttt{backslash}$ replaces the preconditioner density as a measure of the solver memory footprint. In both cases, the solution time $t_t$, as well as the memory requirements, are by far higher than those obtained with the EDFA preconditioner (see Table~\ref{tab:test_3_static} and Figures~\ref{fig:test_3_iter} and \ref{fig:test_3_tt} for reference).

\subsection{Test 4: Dome reservoir with full tensor heterogeneous and isotropic hydraulic conductivity}
In this final application, the dome-structured reservoir is characterized by a full-tensor heterogeneous conductivity field, which is obtained by rotating the element local axes $x$ and $y$, following the curvature of the domain~\citep{Abushaikha2020}:
\begin{linenomath}
\begin{equation}
    \hat{K}^E = R_{xy}^E K^E R_{xy}^{E,T} \quad \text{for} \ E=1,\ldots, N_e
\end{equation}
\end{linenomath}
where $K^E$ is the element diagonal conductivity matrix and $R_{xy}^E = R_y(-\theta_y^E) R_x(-\theta_x^E)$ is the overall rotation matrix. $\theta_x^E$ and $\theta_y^E$ are the average rotations of the domain surface, within element $E$, around the global Cartesian reference system as per the right-hand rule.

The main results for the static strategy are provided in Table~\ref{tab:test_4_static}. The base case performance (run 0) is not satisfactory 
and the only pattern yielding an appreciable acceleration is D (runs 1-5). 
Moreover, in two cases the solver has not converged yet after 2,000 iterations. Filtration here appears to be mandatory and effective, as shown by runs 6 and 7, where a significant post-filtration ($\tau_{\textup{filt}}=$1.E-2) on factor $\widetilde{H}$ has been applied to patterns D and B. This confirms that the native patterns seem to involve too many connections, which turn out to be not significant to capture the complicated flux nature of this test case. 
On the other hand, pre-filtration (runs 8 and 9) does not represent a consistent alternative, due to the larger preconditioner density, which leads to an increase in the application costs even for a smaller number of iterations.  
\begin{table}[tbp] \scriptsize
\caption{Test 4: Numerical performance of the static technique.}
\label{tab:test_4_static}
\centering
\vspace{0.2cm}
\begin{tabular}{ccccccccccccc}
\toprule
\# & Pat & Filt & $\tau_{\textup{filt}}$  & $n_{\textup{it}}$ & $t_{p_0}$ & $t_p$ & $t_s$ & $t_t$ & $\mu$\\
   &     &      &                       &                   &   [s]    &  [s]  &  [s]  & [s]   &     \\
\midrule
0  & Base & *    &    *    &  667 &  5.24 & 0.30 & 42.30 & 42.60 &  2.112 \\
1  &  A   & *    &    *    &  NC  &  4.03 & 0.32 & *     & *     &  2.177 \\
2  &  B   & *    &    *    &  521 &  4.35 & 0.62 & 38.50 & 39.12 &  2.734 \\
3  &  C   & *    &    *    &  667 &  6.07 & 0.30 & 43.10 & 43.40 &  2.112 \\
4  &  D   & *    &    *    &  386 &  4.56 & 1.05 & 31.76 & 32.81 &  3.372 \\
5  &  E   & *    &    *    &  NC  &  7.72 & 1.37 & *     & *     &  3.707 \\
6  &  D   & Post (on $\widetilde{H}$) & 1.E-2  &  231 &  17.40 & 0.04 & 11.79 & 11.83 &  1.214 \\
7  &  B   & Post (on $\widetilde{H}$)    & 1.E-2  &  233 &  14.24 & 0.04 & 12.35 & 12.39 &  1.216 \\
8  &  D   & Pre  & 1.E-3   &  217 &  8.65 & 1.05 & 17.96 & 19.01  &  3.195 \\
9  &  B   & Pre           & 1.E-3   &  424 &  7.49 & 0.63 & 31.68 & 32.31  &  2.717 \\
\bottomrule
\end{tabular}
\end{table}

As to the dynamic technique, the sensitivity analysis on $n_{\textup{ent}}$ and $n_{\textup{add}}$ in Figures~\ref{fig:test_4_iter} and \ref{fig:test_4_tt} reveals that there exists a wide blue area characterized by competitive settings, both in terms of number of iterations and total solving time per step. In this regard, the most attractive portion of the graphs remains the bottom left one. 
Table~\ref{tab:test_4_dynamic} aims at assessing the effect of filtration on the dynamically-formed preconditioner obtained with the settings $n_{\textup{ent}}=6$ and  $n_{\textup{add}}=1$. 
The outcomes of post-filtration, applied to both $\widetilde{S}$ and $\widetilde{H}$ in runs 1 and 2 respectively, are here reported to emphasize the greater efficiency of the latter strategy. Although performing post-filtration on $\widetilde{H}$ gives a higher density preconditioner, the filtration is anticipated in the pre-processing stage of the preconditioner set-up, while preserving the quality in the approximation of $\widetilde{S}$. 
Post-filtration on $\widetilde{H}$ (run 2), in fact, allows to approximately halve the density of the original preconditioner (run 0) to the benefit from a 38\% reduction in the total solution time. Pre-filtration, on the other hand, seems not to be effective (run 3). 
Notice that, in this application, combinations where $n_{\textup{ent}}=n_{\textup{add}}$, i.e., all the prescribed new entries are subsumed at once, give a bad quality $\widetilde{S}$, hence at least two steps of the dynamic procedure are recommended.
\begin{table}
\scriptsize
\caption{Test 4: Pre- and post-filtration on the dynamically-formed preconditioner. Run 0 is obtained with: $n_{\textup{ent}}=6$ and $n_{\textup{add}}=1$.} 
\label{tab:test_4_dynamic}
\centering
\vspace{0.2cm}
\begin{tabular}{ccccccccc}
\toprule
\#  & Filt & $\tau_{\text{filt}}$ & $n_{\text{iter}}$  & $t_{p_0}$ & $t_p$ & $t_s$ & $t_t$ & $\mu$\\
    &      &                      &                    &  [s]      &  [s]  &  [s]  & [s]   &  \\   
\midrule
0   & *    &  *                   &   172              &  69.39    & 1.13  &  14.04 &  15.17 & 3.072 \\
1   & Post &  1.E-3               &   158              &  70.31    & 12.47 &  9.63  &  22.10 & 1.395 \\
2   & Post (on $\widetilde{H}$) &  1.E-3   &   160 &  79.21    & 0.13  &  9.33  &  9.45  & 1.552 \\
3   & Pre  &  1.E-8               &   149              &  73.38    & 1.12  &  12.09 &  13.21 & 3.072 \\
\bottomrule
\end{tabular}
\end{table}
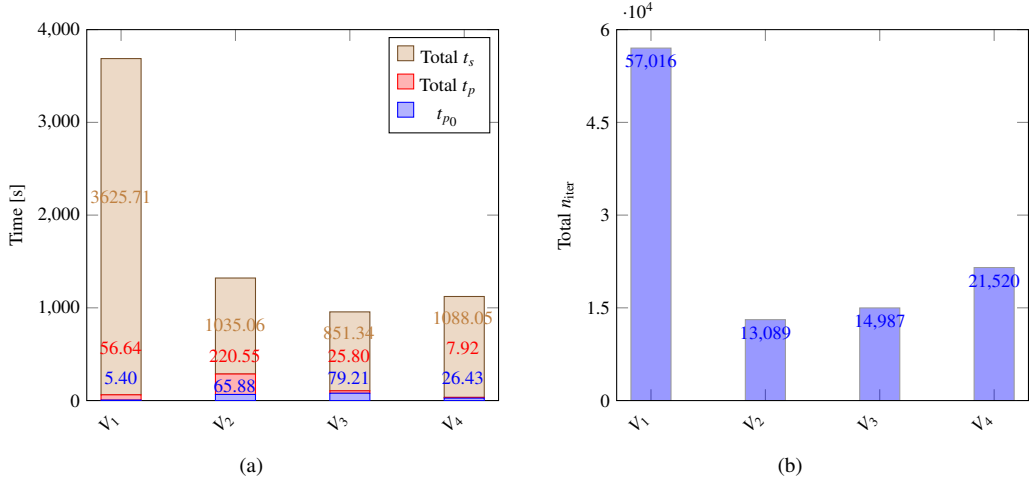
\begin{figure}[tb]
\scriptsize
\centering
\subfloat[][]{
\begin{tikzpicture}
\begin{axis}[
    height=6.5cm,
    width=7cm,
    ybar stacked,
	bar width=15pt,
    ylabel={Time [s]},
    symbolic x coords={$V_1$, $V_2$, $V_3$, $V_4$},
    xtick=data,
    x tick label style={yshift=-4pt,rotate=45,anchor=east},
    ymin = 0, ymax = 4000,
    reverse legend,
    ]
\addplot+[ybar] plot coordinates {($V_1$,5.40) ($V_2$,65.88) ($V_3$,79.21) ($V_4$,26.43)};
\addplot+[ybar] plot coordinates {($V_1$,56.64) ($V_2$,220.55) ($V_3$,25.80) ($V_4$,7.92)};
\addplot+[ybar] plot coordinates {($V_1$,3625.71) ($V_2$,1035.06) ($V_3$,851.34) ($V_4$,1088.05)};
\legend{\strut $t_{p_0}$, \strut Total $t_{p}$, \strut Total $t_{s}$ }
\end{axis}
\node [align=center,color=blue] at (0.45,0.3){5.40};
\node [align=center,color=red] at (0.45,0.7){56.64};
\node [align=center,color=brown] at (0.45,2.7){3625.71};
\node [align=center,color=blue] at (1.95,0.2){65.88};
\node [align=center,color=red] at (1.95,0.6){220.55};
\node [align=center,color=brown] at (1.95,1.0){1035.06};
\node [align=center,color=blue] at (3.45,0.3){79.21};
\node [align=center,color=red] at (3.45,0.6){25.80};
\node [align=center,color=brown] at (3.45,0.9){851.34};
\node [align=center,color=blue] at (4.95,0.3){26.43};
\node [align=center,color=red] at (4.95,0.7){7.92};
\node [align=center,color=brown] at (4.95,1.1){1088.05};
\end{tikzpicture}
\label{fig:test_4_full_sim_time}
}
\hspace{0.5cm}
\subfloat[][]{
\begin{tikzpicture}
\begin{axis}[
    ybar,
    height=6.5cm,
    width=7cm,
	bar width=15pt,
	xtick align=inside,nodes near coords,
    every node near coord/.append style={anchor=north,opacity=1,color=blue},
    legend style={at={(0.5,-0.20)},
      anchor=north,legend columns=-1},
    ylabel={Total $n_{\textup{iter}}$},
    symbolic x coords={$V_1$, $V_2$, $V_3$, $V_4$},
    xtick=data,
    ymin = 0, ymax = 60000,
    x tick label style={yshift=-4pt,rotate=45,anchor=east},
    ytick = {0,15000,30000,45000,60000},
    ]
\plot[fill=blue,opacity=0.4] plot coordinates {($V_1$,57016) ($V_2$,13089) ($V_3$,14987) ($V_4$,21520)};
\end{axis}
\end{tikzpicture}
\label{fig:test_4_full_sim_iter}
}
\caption{Analyzing the performance of four versions of the EDFA preconditioner in a full-transient simulation. The outcomes are expressed in terms of overall solution time (a), broken down into its basic components, and total number of iterations (b). The preconditioners are obtained with the following settings: $V_1 \to$ Base pattern (run 0 in Table~\ref{tab:test_4_static}), $V_2 \to$ Dynamic strategy with $n_{\textup{ent}}=6, \ n_{\textup{add}}=1$ 
(run 0 in Table~\ref{tab:test_4_dynamic}), $V_3 \to$ as $V_2$ with post-filtration on $\widetilde{H}$ and $\tau_{\textup{filt}}=$1.E-3 (run 2 in Table~\ref{tab:test_4_dynamic}) and $V_4 \to$ Static strategy with pattern D, post-filtration on $\widetilde{H}$ and $\tau_{\textup{filt}}=$1.E-2 (run 6 in Table~\ref{tab:test_4_static}).}
\label{fig:test_4_full_sim}
\end{figure}

Despite the high density, the unfiltered dynamically-formed preconditioner in Table~\ref{tab:test_4_dynamic} is by far more competitive than the corresponding static alternative in Table~\ref{tab:test_4_static} (runs 0-5). An explanation comes from the analysis of the dynamic pattern representation vs the position within the dome-grid and the flux distribution, as provided in Figure~\ref{fig:test_4_dyn_pat}. The dynamic technique is capable to flexibly catch and exploit the possibly complex physics of the fluxes behind the problem.
Due to the dome structure of the domain and the heterogeneity of the hydraulic conductivity, the fluid fluxes are not uniformly distributed in the domain. As an effect, the resulting pattern is not the same throughout the grid and can be hardly guessed, hence a statically designed homogeneous pattern is not well-suited to the specific requirements of this application. 

In conclusion, the performance of the EDFA preconditioner is evaluated in a full-transient simulation, reproducing the exploitation of a reservoir, initially undisturbed, under the conditions mentioned in Section~\ref{sec:num_results}, i.e., four injectors (one at each corner) and a producer (in the centre) operating at a constant pressure.
The overall computational times are displayed in Figure~\ref{fig:test_4_full_sim}, where a comparison among a selection of four variants of the EDFA preconditioner 
is proposed. The performance of the basic setting 
($V_1$, run 0 in Table~\ref{tab:test_4_static}) is used as benchmark to test the advantage provided by the most efficient dynamic 
($V_2$ and $V_3$, 
runs 0 and 2 in Table~\ref{tab:test_4_dynamic}) and static strategies 
($V_4$, run 6 in Table~\ref{tab:test_4_static}). The simulated time interval is 130 days 
with 200 time steps. The other relevant settings 
are: $\Delta t_{\max}=5$ d, $\Delta p_T=5$ bar, $\Delta t_{\textup{mult}}=1.1$. The best performance 
is given by variant $V_3$, which proved to be the most competitive one also in the previous steady state analysis. 
Nevertheless, the difference with $V_2$ and $V_4$ is not as significant as the steady state analysis might have suggested. 
In any case, 
all trials $V_2$, $V_3$ and $V_4$ clearly outperform $V_1$, reducing the total time to about one fourth.
\begin{figure}
    \centering
    \includegraphics[width=12.5cm]{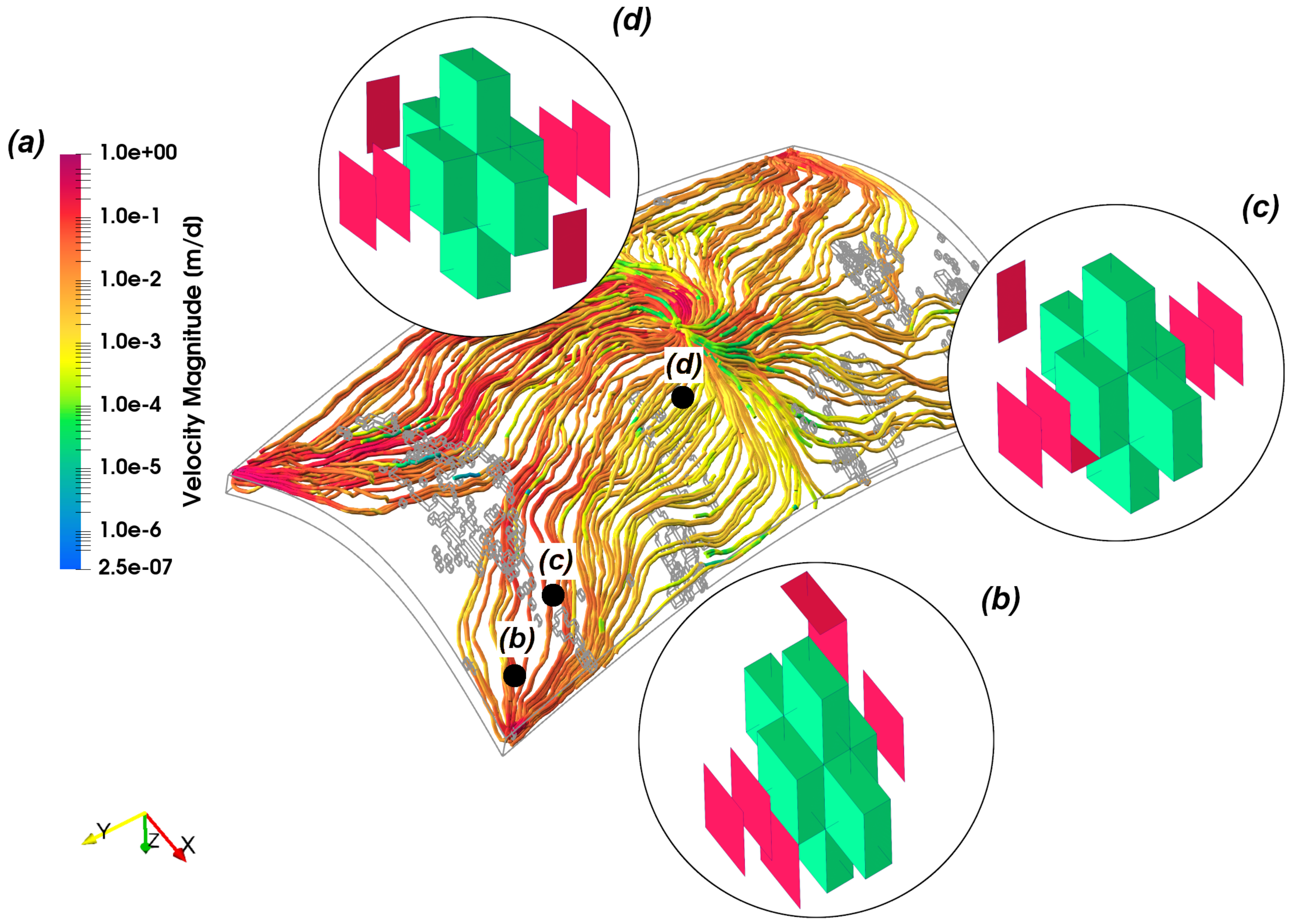}
    \caption{Test 4: Physical interpretation, based on the flux distribution (a), of the dynamic pattern, obtained with the settings $n_{\textup{ent}}=6$ and $n_{\textup{add}}=1$ 
    at three different locations. The red faces represent the extension of the original $A_{p \pi}^T$ pattern in light-green.}
    \label{fig:test_4_dyn_pat}
\end{figure}

\section{Discussion}
Achieving a fast and cheap solution to the sequence of systems~\eqref{eq:system} that stem from a full MHFE-FV single-phase flow simulation is the main objective of this paper. 
Given the peculiar properties of system~\eqref{eq:system}, which is characterized by a non-symmetric generalized saddle-point structure and usually ill-conditioned blocks, an original preconditioning strategy, denoted as EDFA, has been specifically designed to accelerate the convergence of Krylov subspace solvers. The key features of the proposed approach are twofold: (i) the exploitation of the decoupling factors, $G$ and $F$, of the system matrix $\mathcal{A}$ block LDU decomposition to recast the Schur complement avoiding the inversion of the leading block (equation~\eqref{eq:schur_compl}); (ii) the inexact computation of these factors by solving two independent sets of MRHS systems~\eqref{eq:syst_MRHS}, by means of a combination of restriction and prolongation operators based on the non-zero pattern of $\widetilde{G}$ and $\widetilde{F}$.
Such an approximation turns out to be optimal, with respect to $A_{\pi \pi}$-inner product, for the selected non-zero pattern.
The ability to recognize the most representative entries of $G$ and $F$ 
is key to obtain high quality approximations at a workable sparsity and achieve a fast convergence. This is the motivation behind the introduction of two strategies, namely static and dynamic, aimed at selecting the sparsity pattern of the two decoupling factors by customizing the face-to-element connections contained in the columns of $A_{p \pi}^T$, depending on the properties of the problem at hand.

A physics-based concept underlies the static technique, since the entries of the generic column of the decoupling factors correspond to the face pressure unknowns resulting from the fluid flow arising in a certain small and compact partition of the physical domain. Expanding such partition in order to incorporate the most relevant entries is a way for enlarging effectively the native sparsity pattern. The size of the augmented partition, as well as the location of the faces associated with the newly-added entries, give rise to a plethora 
of possible combinations. As an example, in this paper we considered five general-purpose prototypes (Figures~\ref{fig:pattern_A}-\ref{fig:pattern_E}). 
This strategy turns out to be effective when the modeller has a robust idea of the flux distribution and the resulting pattern can be more or less uniformly extended throughout the domain (compare for instance Figures~\ref{fig:test_1_dyn_pat} and \ref{fig:test_4_dyn_pat}).

On the other hand, a fully algebraic framework is introduced to define a dynamic variant, where $n_{\textup{ent}}$ new entries to the initial pattern are progressively added to an initial guess 
at the locations corresponding to the largest components of the prolonged residual (equation~\eqref{eq:prol_res}). 
Compared to the static variant, the dynamic technique is computationally more demanding because, at every step of the pattern construction, a static solution is needed, however this cost can be easily amortized during a transient simulation and take advantage of an almost ideal parallel degree. Furthermore, this techniques is more flexible, since it is capable to implicitly capture the physics behind the problem without the modeller being aware of it and might be used in a black-box fashion as well.
Pre- and post-filtration techniques have been introduced with the twofold purpose of controlling the density of $\widetilde{S}$ and improving the quality of its inexact inverse application by removing possibly detrimental near-zero entries.

The extensive experimental phase of Section~\ref{sec:num_results} helped understand the potential of the EDFA preconditioner in different settings, according to the grid type and hydraulic properties, but also suggests indications about default optimal settings.
All the tests revealed that less than 12 new entries are actually needed in addition to the original column patterns, 
but it is the grid type, i.e., whether structured or unstructured, that mostly influences 
the optimal EDFA preconditioner set-up. With a structured grid, the suggested static patterns in Figure~\ref{fig:static_patterns} seem to be appropriate, and especially E and A as shown in Tables~\ref{tab:test_1_static} and \ref{tab:test_3_static}. By comparing Figures~\ref{fig:pattern_A} and \ref{fig:pattern_E}, notice that the physical structure of these patterns is similar as the elements involved in the their definition is the same. By distinction, for the dynamic variant optimal results have been obtained by setting $n_{\textup{ent}}$ between 4 and 10, 
with $n_{\textup{add}} \approx n_{\textup{ent}}$. Filtration is not strictly necessary, even though some good results were obtained with pre-filtration and $\tau_{\textup{filt}}$ between 1.E-4 and 1.E-3 (Table~\ref{tab:test_3_static}). Conversely, with an unstructured grid, patterns B and D turned out to be the winning choice for the static technique, whereas for the dynamic one 
it is advisable to set $n_{\textup{add}} < n_{\textup{ent}} < 10$ (Figure~\ref{fig:test_4_tt}). 
A significant post-filtration with 1.E-3 $\leq \tau_{\textup{filt}} \leq$ 1.E-2 proved to be effective to accelerate, or even to allow for, convergence (see for instance Tables~\ref{tab:test_2_static} and \ref{tab:test_4_static}). 
Post-filtration on $\widetilde{H}$, rather than on $\widetilde{S}$, should be preferred. 

Finally, Table~\ref{tab:test_1_delta_t} and Figure~\ref{fig:test_4_full_sim} showed that 
the preconditioner set-up for steady-state conditions plays the role of a worst-case scenario.
In particular, the smaller $\Delta t$, the better the accuracy of $\widetilde{S}$ (equation~\eqref{eq:schur_compl}). 
The reason for this behavior comes from the structure of $A_{p p}$ (equation~\eqref{eq:expr_A_pp}), where the diagonal entries depend on the inverse of $\Delta t$. Therefore, when $\Delta t$ is small, they tend to prevail over the other contributions in $\widetilde{S}$, and $A_{pp}$ becomes diagonally dominant. 
This observation suggests a different set-up strategy of the preconditioner, because building $\widetilde{F}$ and $\widetilde{G}$ with the optimal settings at steady state might be too conservative. The two decoupling factors can be approximated more than once during a full-transient simulation, since usually $\Delta t$ increases as steady state is approached, and cheaper approximations can be effective at the initial stages. 

\section{Conclusions}
In this paper, we introduce a novel preconditioning technique, denoted as EDFA,
for the solution of the sequence of non-symmetric block linear systems arising from the original MHFE-FV discretization of flow problems in porous media developed in \cite{Abushaikha2017}. The proposed method is
based on the approximation of the decoupling factors of the system matrix by using appropriate restriction operators for the sake of the Schur complement computation. 
The experimental phase proved its robustness and reliability in different settings, depending on the structure of the grid and the properties of the hydraulic conductivity tensor. 
The EDFA preconditioning strategy exhibits several attractive features, in particular:
\begin{enumerate}
    \item Since the process for building the preconditioner is based on the solution of the set of independent MRHS systems~\eqref{eq:syst_MRHS}, the set-up stage is inherently parallel and can fully exploit the architecture of modern computing platforms;
    \item The overall set-up stage can be split in a two-step procedure, where the first stage is performed only at the beginning of a full-transient simulation and the second stage at each time step;
    \item The largest computational cost, associated with the approximation of $G$ and $F$, is concentrated in the first stage, so it can be effectively amortized during a full simulation.
\end{enumerate}
Both the static and dynamic variants proved to be overall efficient. A winner does not stand out clearly, even though it might be better to rely on the dynamic technique when the flux distribution is highly variable throughout the domain and it is hard to define a uniform static pattern prototype. All the same, it is appreciable the relatively easy setting up and cheapness of the static variant and the flexibility of the dynamic.

Research is currently ongoing to develop a fully parallel C++ implementation of the proposed solver and extend the formulation to multi-phase MHFE-FV reservoir models.

\section{Acknowledgment}

This publication was supported by the National Priorities Research Program grant NPRP10-0208-170407  from Qatar National Research Fund.

\appendix
\section{MHFE-FV matrices}
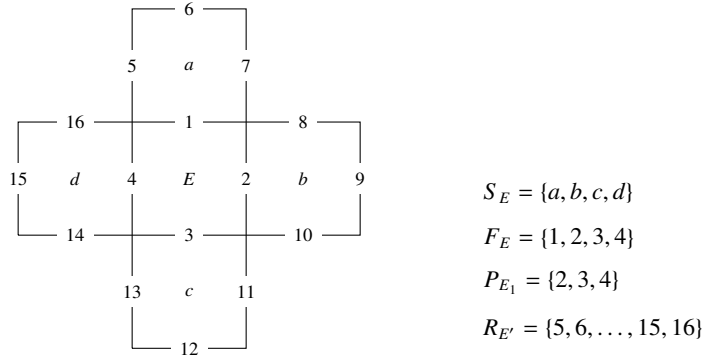
\begin{figure}[tb]
\centering
\begin{tikzpicture}[scale=1.5]
\draw (1,0) -- ++(0,3) -- ++(1,0) -- ++(0,-3) -- ++(-1,0);
\draw (0,1) -- ++(3,0) -- ++(0,1) -- ++(-3,0) -- ++(0,-1);
%
\node[align=center, font = \scriptsize] at (1.5,1.5){$E$};
\node[align=center, font = \scriptsize, fill=white] at (1.5,2.5){$a$};
\node[align=center, font = \scriptsize, fill=white] at (2.5,1.5){$b$};
\node[align=center, font = \scriptsize, fill=white] at (1.5,0.5){$c$};
\node[align=center, font = \scriptsize, fill=white] at (0.5,1.5){$d$};
%
\node[align=center, font = \scriptsize, fill=white] at (1.5,2){$1$};
\node[align=center, font = \scriptsize, fill=white] at (2,1.5){$2$};
\node[align=center, font = \scriptsize, fill=white] at (1.5,1){$3$};
\node[align=center, font = \scriptsize, fill=white] at (1,1.5){$4$};
\node[align=center, font = \scriptsize, fill=white] at (1,2.5){$5$};
\node[align=center, font = \scriptsize, fill=white] at (1.5,3){$6$};
\node[align=center, font = \scriptsize, fill=white] at (2,2.5){$7$};
\node[align=center, font = \scriptsize, fill=white] at (2.5,2){$8$};
\node[align=center, font = \scriptsize, fill=white] at (3,1.5){$9$};
\node[align=center, font = \scriptsize, fill=white] at (2.5,1){$10$};
\node[align=center, font = \scriptsize, fill=white] at (2,0.5){$11$};
\node[align=center, font = \scriptsize, fill=white] at (1.5,0){$12$};
\node[align=center, font = \scriptsize, fill=white] at (1,0.5){$13$};
\node[align=center, font = \scriptsize, fill=white] at (0.5,1){$14$};
\node[align=center, font = \scriptsize, fill=white] at (0,1.5){$15$};
\node[align=center, font = \scriptsize, fill=white] at (0.5,2){$16$};
\node[align=right, font = \small, anchor=south west] at (4,1.2){$S_E = \left\{ a,b,c,d \right\}$};
\node[align=right, font = \small, anchor=south west] at (4,0.8){$F_E = \left\{ 1,2,3,4 \right\}$};
\node[align=right, font = \small, anchor=south west] at (4,0.4){$P_{E_1} = \left\{ 2,3,4 \right\}$};
\node[align=right, font = \small, anchor=south west] at (4,0){$R_{E'} = \left\{ 5,6,\ldots,15,16 \right\}$};
\end{tikzpicture}
\caption{Two-dimensional sketch of the element/face connections with an example of the basic sets $S_E$, $F_E$, $P_{E_i}$ and $R_{E'}$.}
\label{fig:patt_A}
\end{figure}
In this number, the expressions for the four sub-matrices of system~\eqref{eq:system} are provided based on the definitions~\eqref{eq:fluxes_MHFEM_3}, \eqref{eq:discr_bal_FV_2}, \eqref{eq:fluxes_MHFEM_STRONG} and \eqref{eq:cont_fluxes}.
Let us consider the element $E$ of the grid, the set of its faces, $F_E$, and neighbours, $S_E$, where $E'$ is an element of $S_E$. $P_{E_i}$ denotes the set of faces of $E$ without the one, say $i$, shared with $E'$, i.e., $F_E - F_E \cap F_{E'} = F_E - i$. By extension, $P_{{E'}_i} = F_{E'} - i$. $R_{E'}$, formally defined as $\cup_{i \in F_E} P_{{E'}_i}$, is the set of faces of the elements in $S_E$ not shared with $E$. Figure~\ref{fig:patt_A} provides a graphical interpretation of the aforementioned sets in a 2-D setting. Let also $g: k \to p=g(k)$ be the function that converts the global matrix index $k$ into the local one $p$. The expressions for the $\mathcal{A}$ sub-blocks read:
\begin{linenomath}
\begin{align}
    [A_{\pi \pi}]_{ij} : & & & \forall i \in \{1,2, \ldots, N_f\} & & \text{if} \ i =j & & -B_{g(i)g(i)}^{E^{-1}} - B_{g(i)g(i)}^{E'^{-1}}, \notag\\
                         & & & & & \text{if} \ j \in P_{E_i} & & -B_{g(i)g(j)}^{E^{-1}}, \\
                         & & & & & \text{if} \ j \in P_{{E'}_i} & & -B_{g(i)g(j)}^{E'^{-1}}, \notag \\
             \notag \\
    [A_{\pi p}]_{ij} : & & & \forall i \in \{1,2, \ldots, N_f\} & & \text{if} \ j \in \left\{ E,E' \right\}  & & \sum_{k=1}^{N_f^E} B_{g(i)k}^{j^{-1}}, \\
                  \notag \\
    [A_{p \pi}]_{Ej} : & & & \forall E \in \{1,2, \ldots, N_e\} & & \text{if} \ j \in F_E  & & - \sum_{k \in K_E} B_{g(k)g(j)}^{E^{-1}} \frac{B_{g(k)g(k)}^{{E'}^{-1}}}{B_{g(k)g(k)}^{{E'}^{-1}}+B_{g(k)g(k)}^{E^{-1}}}, \quad \text{where} \ K_E = F_E - j, \\
                       & & & & & \text{if} \ j \in R_{E'}  & & B_{g(i)g(j)}^{{E'}^{-1}} \frac{B_{g(i)g(i)}^{E^{-1}}}{B_{g(i)g(i)}^{{E'}^{-1}}+B_{g(i)g(i)}^{E^{-1}}}, \qquad \qquad \ \ \text{where} \ i = F_E \cap F_{E'}, \notag \\
              \notag \\
    [A_{p p}]_{El} : & & & \forall E \in \{1,2, \ldots, N_e\} & &\text{if} \ E = l  & & \sum_{i \in F_E} \frac{B_{g(i)g(i)}^{{E'}^{-1}}}{B_{g(i)g(i)}^{{E'}^{-1}}+B_{g(i)g(i)}^E} \sum_{j \in F_E} B_{g(i)g(j)}^{E^{-1}} + \frac{\Omega^E \overline{c}^E}{\Delta t_n},
    \label{eq:expr_A_pp}\\
                     & & & & & \text{if} \ l \in S_E  & & - \frac{B_{g(i)g(i)}^{E^{-1}}}{B_{g(i)g(i)}^{l^{-1}}+B_{g(i)g(i)}^{E^{-1}}} \sum_{j \in F_l} B_{g(i)g(j)}^{l^{-1}}, \quad \ \text{where} \ i = F_E \cap F_l, \notag
\end{align}
\end{linenomath}


\bibliography{Bib}

\end{document}